\documentclass[11pt]{article}
\usepackage{amsfonts}
\usepackage{CJK}
\usepackage{graphics}
\usepackage{indentfirst}
\usepackage{cite}
\usepackage{latexsym}
\usepackage{amsmath}
\usepackage{amsthm}
\usepackage{amssymb}
\usepackage[dvips]{epsfig}
\usepackage{amscd}

\usepackage{color}

\newtheorem{theorem}{Theorem}[section]
\newtheorem{remark}{Remark}[section]

\newtheorem{lemma}[theorem]{Lemma}
\newtheorem{pro}{Proposition}[section]
\newtheorem{cor}[theorem]{Corollary}

\renewcommand{\div}{ {\rm div }  }

\newcommand{\na}{\nabla }

\newcommand{\pa}{\partial}
\renewcommand{\r}{\mathbb{R}}

\newcommand{\dis}{\displaystyle}
\newcommand{\ia}{\int_0^T}

\newcommand{\bt}{{\hat\theta}}
\newcommand{\bl}{\begin{lemma}}
\newcommand{\el}{\end{lemma}}
\newcommand{\et}{\end{theorem}}
\newcommand{\ga}{\gamma}

\newcommand{\curl}{{\rm curl} }
\newcommand{\te}{\theta}

\newcommand{\al}{\alpha}
\newcommand{\de}{\delta}
\newcommand{\ve}{\varepsilon}
\newcommand{\la}{\label}

\newcommand{\p}{p(\rho)  }

\newcommand{\ka}{\kappa}

\newcommand{\bn}{\begin{eqnarray}}
\newcommand{\en}{\end{eqnarray}}
\newcommand{\bnn}{\begin{eqnarray*}}
\newcommand{\enn}{\end{eqnarray*}}

\newcommand{\bnnn}{\begin{eqnarray*}}
\newcommand{\ennn}{\end{eqnarray*}}

\newcommand{\ba}{\begin{aligned}}
\newcommand{\ea}{\end{aligned}}
\newcommand{\be}{\begin{equation}}
\newcommand{\ee}{\end{equation}}
\def\O{\Omega}
\def\p{\partial}
\def\norm[#1]#2{\|#2\|_{#1}}

\newcommand{\ep}{\varepsilon}
\newcommand{\n}{\rho}
\newcommand{\si}{\sigma}

\def\la{\label}

\def\na{\nabla}
\def\on{\hat\rho}

\def\tn{1}

\def\xl{\left}
\def\xr{\right}

\makeatletter   
\@addtoreset{equation}{section}
\makeatother       

\title{
Global Classical Solutions to the  Full Compressible Navier-Stokes System   in $3$D Exterior Domains}

\author{Jiaxu L{\small I}$^{a}$, Jing L{\small I}$^{a,b,c}$, Boqiang L{\small \"U}$^{b} $  \thanks{email: jiaxvlee@gmail.com (J.X. Li), ajingli@gmail.com (J. Li), lvbq86@163.com (B.Q.   L\"u) } \\
{\normalsize a.  School of Mathematical Sciences,}\\
{\normalsize  University of Chinese Academy of Sciences, Beijing 100049, P. R. China;}\\
{\normalsize b. Department of Mathematics }\\ {\normalsize \& Institute of Mathematics and Interdisciplinary Sciences,}\\ {\normalsize  Nanchang University, Nanchang 330031, P. R. China;}
 \\
{\normalsize c. Institute of Applied Mathematics, AMSS,} \\ {\normalsize \&   Hua Loo-Keng Key Laboratory of Mathematics,}\\
{\normalsize  Chinese Academy of Sciences,    Beijing 100190, P. R. China}}

\date{}
\begin{document}
\maketitle

\begin{abstract}
The full compressible Navier-Stokes system (FNS) describing the motion of a viscous, compressible, heat-conductive, and Newtonian polytropic fluid in a three-dimensional (3D) exterior domain is studied.
For the initial-boundary-value problem with the slip boundary conditions on the velocity and the Neumann one on the temperature, it is shown that there exists a  unique global    classical solutions with the initial data which are of small energy but possibly large oscillations. In particular, both the density and temperature are allowed to vanish initially. This is the first result about classical solutions of FNS system in 3D exterior domain.
\end{abstract}


\section{Introduction}

The motion of a compressible viscous, heat-conductive, and Newtonian polytropic fluid occupying a spatial domain $\Omega \subset \mathbb{R}^{3}$ is governed by the following full compressible Navier-Stokes system:
\be \la{a0}
\begin{cases}
 	\n_t+{\rm div} (\n u)=0,\\
 	(\n u)_t+{\rm div}(\n u\otimes u)+\na P=\div \mathbb{S},\\
 	(\n E)_t+\div (\n Eu+Pu)=\div (\kappa \na \te)+\div (\mathbb{S}u),
\end{cases}\ee
where   $\mathbb{S}$ and $E$ are respectively the viscous stress tensor and the total energy given by
$$\mathbb{S}=2\mu\mathfrak{D}(u) +\lambda \div  u \mathbb{I}_3,~~E=e+\frac{1}{2}|u|^2,$$
with $\mathfrak{D}(u)  = (\nabla u + (\nabla u)^{\rm tr})/2$ and $\mathbb{I}_3$  denoting the deformation tensor and the $3\times3$ identity matrix respectively.
Here, $t\ge 0$ is time, $x\in \Omega$ is the spatial coordinate, and $\n$, $u=\left(u^1,u^2,u^3\right)^{\rm tr},$ $e$, $P,$ and $\te $ represent respectively the fluid density, velocity, specific internal energy, pressure, and  absolute temperature.
The viscosity coefficients $\mu$ and $\lambda$  are constants satisfying the physical restrictions:
\be\la{h3} \mu>0,\quad 2 \mu + 3\lambda\ge 0.\ee
The heat-conductivity coefficient $\ka$ is a positive constant.
We consider the ideal  polytropic fluids so that $P$ and $e$ are given by the state equations:
\be \la{pg}
P(\n,e)=(\ga-1)\n e=R\n \te, \quad e=\frac{R\theta}{\ga-1},
\ee
where $\ga>1$  is the adiabatic constant  and $ R$ is a positive constant.

Note that for the classical solutions, the system \eqref{a0} can be rewritten as
\be \la{a1}
\begin{cases}
	\n_t+{\rm div} (\n u)=0,\\
	\n (u_t+  u\cdot \na u )=\mu\Delta u+(\mu+\lambda)\na({\rm div}u)-\na P,\\
	\frac{R}{\ga-1}\n ( \te_t+u\cdot\na \te)=\ka\Delta\te-P\div u+\lambda (\div u)^2+2\mu |\mathfrak{D}(u)|^2.
\end{cases}
\ee
Let $\Omega=\r^3-\bar{D}$ be the exterior of a simply connected bounded domain $D\subset\r^3$ with smooth boundary $\partial\Omega,$   and we consider the system \eqref{a1}  subjected  to the given initial data
\be \la{h2}
(\rho,\n u, \n\te)(x,{t=0})=(\rho_0,\n_0u_0,\n_0\te_0)(x), \quad x\in \Omega,
\ee
and the slip boundary conditions
\be \la{h1}
u\cdot n=0,~~\mbox{curl} u\times n=0,~~\na \theta\cdot n=0,~~~~\text{on}~\p\Omega \times(0,T),
\ee
with the far field behavior
\be \la{ch2} (\n,u,\te)(x,t)\rightarrow (1,0,1),\,\,\text{as}\,\, |x|\rightarrow\infty, \ee
where $n =(n^1, n^2, n^3)^{\rm{tr}}$ is the unit outward normal vector on $\partial\Omega$. Indeed,  choosing a positive real number $d$  such that $\bar{D}\subset B_d$, one can extend the unit outer normal $n$ to $\Omega$ such that \be\la{ljq10} n\in C^3(\bar{\Omega}),\quad n\equiv 0  \mbox{ on }\r^3\setminus  B_{2d}.\ee

The  well-posedness of solutions to the compressible Navier-Stokes systems has been studied intensively by many authors.
 The local existence and uniqueness of classical solutions without vacuum is  established in \cite{se1,Na,Tani}, and   the local  strong solutions with initial vacuum are   derived  by \cite{cho1,K2,sal,choe1,liliang}. The first breakthrough in solving global well-posedness of Cauchy problem came in the work of Matsumura-Nishida \cite{M1}, who obtained the global classical solutions with initial data close to a non-vacuum equilibrium in some Sobolev space $H^s.$ Later, Hoff \cite{Hof1,Hof2} established the global weak solutions with strictly positive initial density and temperature for discontinuous initial data.
When the initial vacuum is allowed, the issue on  the existence of solutions becomes much more complicated.
The major breakthrough on barotropic case is due to Lions \cite{L1} (see also  Feireisl \cite{feireisl1}), where he obtained the global existence of weak solutions, defined as  finite energy solutions, if the adiabatic exponent $\gamma$ is sufficiently large.
Recently,  Li-Xin \cite{2dlx} and Huang-Li-Xin \cite{hulx} obtained the global classical solutions to the 2D and 3D Cauchy problem respectively with the initial data which are of small energy but possibly large oscillations.
More recently, for initial-boundary-value problem (IBVP) with Navier-slip boundary condition,
Cai-Li \cite{C-L} and Cai-Li-L\"u \cite{C-L-L} obtained the  global classical solutions with initial vacuum in 3D bounded domains and 3D exterior domains respectively, provided that the initial energy is suitably small.

Compared with the barotropic flows, it is more difficult and complicated to study the global well-posedness of solutions to full compressible Navier-Stokes system (\ref{a0}) with initial vacuum, where some additional difficulties arise, such as the degeneracy of both momentum and temperature equations,  and the strong coupling between the velocity and temperature,  etc. Feireisl \cite{feireisl,feireisl1} proved the global existence of variational weak solutions with large data in the case of real gases.
Later, Bresch-Desjardins \cite{bd} obtained global stability of weak solutions to (\ref{a0}) with  viscosity coefficients depending on the density.
Xin \cite{xin98} (see also \cite{xin13}) proved the smooth or strong solutions will blow up in finite time if the initial data have an isolated mass group, no matter how small the initial data are.
Recently, Huang-Li \cite{H-L} established the global existence and uniqueness for the classical solutions to the 3D Cauchy problem with interior vacuum provided the initial energy is small enough.  Later,  Wen-Zhu  \cite{W-C} considered the case of vanishing far field conditions under the assumption that the initial mass is sufficiently small or both viscosity and heat-conductivity coefficients are large enough. More recently, the global existence of classical solutions  which are of small energy but possibly large oscillations was established in Li-L\"u-Wang \cite{L-L-W} for  the IBVP  with Navier-slip boundary condition  in general 3D bounded domains.
 This paper aims to study the global  classical solutions to full compressible Navier-Stokes equations in 3D exterior domains.

Before stating the main results, we first introduce the notations and conventions used throughout this paper. We denote
\bnn\int fdx\triangleq\int_{\Omega}fdx.\enn
For $1\le p\le \infty $ and integer $k\ge 0,$  we adopt the following notations for some Sobolev spaces as
follows:
\be\ba\notag \begin {cases}
L^p=L^p(\Omega),\quad W^{k,p}=W^{k,p}(\Omega),\quad H^k=W^{k,2},\\
 D^{k,p}=\{\left.f\in L^1_{loc}(\O)\right|\na^{k}f \in L^p(\O)\},\quad D^k=D^{k,2}\\
 H_s^2= \left.\left\{f\in H^2 \right | f\cdot n=0,\,\curl f\times n=0 \rm \,\,{on}\,\, \p\O\right\}.\end{cases}\ea\ee
The initial energy  $C_0$  is defined as follows:
\be\la{e}\ba  C_0\triangleq &\frac{1}{2}\int\n_0
|u_0|^2dx+R \int  \left( \n_0\log {\n_0}-\n_0+1 \right)dx\\
&+\frac{R}{\ga-1}\int  \n_0\left(\te_0- \log {\te_0} -1
\right) dx .\ea\ee
Let $D_t$ and $G$ be the material derivative and  effective viscous flux respectively, defined by
\be \la{hj1} D_t f \triangleq \dot f\triangleq f_t+u\cdot\nabla f,\quad G\triangleq(2\mu + \lambda)\div u -  R(\rho\te-1). \ee

The  main result in this paper is stated  as follows:

\begin{theorem}\la{th1}
Let $\Omega=\r^3-\bar{D}$ be the exterior of a simply connected bounded domain $D\subset\r^3$  and its boundary $\partial\Omega$ is smooth.
For  given numbers $M>0$ (not necessarily
small), $q\in (3,6),$
  $\on> 2,$ and $\bt>1,$
suppose that the initial data $(\n_0,u_0,\te_0)$ satisfies
\be
\la{co3} \n_0-1\in H^2\cap W^{2,q}, \quad  u_0 \in H^2_s, \quad
\te_0-1\in H^1,\quad \na \te_0\cdot n|_{\p\O}=0 , \ee
\be \la{co4} 0\le\inf\rho_0\le\sup\rho_0<\hat{\rho},\quad 0  \le\inf\te_0\le\sup\te_0\le \bt, \quad \|\na u_0\|_{L^2} \le M,
\ee
   and the compatibility condition
\be
\la{co2}-\mu \Delta u_0-(\mu+\lambda)\na\div u_0+R\na (\n_0\te_0)
=\sqrt{\n_0} g \ee
with  $g\in L^2. $ Then there exists a positive constant $\ve$
depending only
 on $\mu,$ $\lambda,$ $\ka,$ $ R,$ $ \ga,$  $\on,$ $\bt$, and $M$ such that if
 \be
 \la{co14} C_0\le\ve,
   \ee  the   problem  (\ref{a1})--(\ref{ch2})
admits a unique global classical solution $(\rho,u,\te)$ in
   $\Omega\times(0,\infty)$ satisfying
  \be\la{h8}
  0\le\rho(x,t)\le 2\hat{\rho},\quad \te(x,t)\ge 0,\quad x\in \Omega,\,~~ t\ge 0,
  \ee
 and \be
   \la{h9}\begin{cases}
   \rho-1\in C([0,T];H^2\cap W^{2,q}),\\
   u \in C([0,T];H^1)\cap C((0,T];D^2), \ \te-1   \in  C((0,T];H^2),\\
u\in  L^\infty(0,T;H^2)\cap L^\infty(\tau,T;H^3\cap W^{3,q}),\  \te-1\in   L^\infty(\tau,T;H^4), \\
  (u_t,\te_t)\in
L^{\infty}(\tau,T;H^2)\cap H^1(\tau,T;H^1), \end{cases} \ee for any $0<\tau<T<\infty.$
Moreover, the following large-time behavior holds:
 \be\la{h11}
  \lim_{t\rightarrow \infty}\left( \|\n(\cdot,t)-1\|_{L^p} +\|\na u(\cdot,t)\|_{L^r }+\|\na \te(\cdot,t)\|_{L^r }\right)=0,  \ee
  for any $p\in  (2 ,\infty)$ and  $r\in [2,6)$.
\end{theorem}

Next, as a direct application of \eqref{h11}, the following Corollary \ref{th3} gives the  large-time behavior of the gradient of density  when the vacuum   appears  initially. The proof is similar to that of     \cite[Theorem 1.2]{hulx} (see also \cite[Corollary 1.3]{H-L}).

\begin{cor} \la{th3}
In addition to the conditions  of Theorem \ref{th1}, assume further
that there exists some point $x_0\in \Omega$ such that $\rho
_0(x_0)=0.$ Then  the unique global    classical solution
$(\rho,u,\te)$ to the   problem (\ref{a1})--(\ref{ch2})
obtained in Theorem \ref{th1} has to blow up as $t\rightarrow
\infty,$ in the sense that for any $r>3,$
\be\label{decay}\lim\limits_{t\rightarrow \infty}\|\nabla \rho(\cdot,t)
\|_{L^r }=\infty.\ee
\end{cor}

A few remarks are in order:

\begin{remark}
One can deduce from (\ref{h9}) and the Sobolev imbedding  theorem that for any $0<\tau<T<\infty,$
\be \la{hk1}
(\n-1,\, \na\n, u) \in C(\overline{\Omega} \times[0,T] ),\quad   \te-1 \in C(\overline{\Omega} \times(0,T] ),\ee
and
\be \ba
(\na u,\,\na^2 u) \in  C( [\tau,T];L^2 )\cap L^\infty(\tau,T;  W^{1,q})\hookrightarrow  C(\overline{\Omega} \times[\tau,T] ),\\
(\na \te,\,\na^2 \te) \in  C( [\tau,T];L^2 )\cap L^\infty(\tau,T;  H^2)\hookrightarrow  C(\overline{\Omega} \times[\tau,T] ),
\ea \ee
which together with  (\ref{a1}), (\ref{h9}),  and (\ref{hk1}) gives
\be  \la{hk3} (\n_t,\,u_t,\,\te_t)\in   C(\overline{\Omega}\times[\tau,T] ).\ee
Therefore, the solution $(\n,u,\te)$ obtained in Theorem \ref{th1} is a classical one to  the   problem (\ref{a1})--(\ref{ch2}) in $\Omega\times (0,\infty).$
\end{remark}

\begin{remark} In  \cite{C-L-L}, Cai-Li-L\"u studied
 the global existence result of the  barotropic flows in  the exterior domains.  For the full compressible Navier-Stokes system, our Theorem \ref{th1}   is the first result concerning the global existence of classical solutions with  vacuum to problem (\ref{a1})--(\ref{ch2}) in  the exterior domain. Moreover, although it's energy is small, the oscillations could be arbitrarily large. 
 \end{remark}

\begin{remark} Similar to Li-L\"u-Wang \cite{L-L-W} where they consider   the IBVP in 3D bounded domains,  we only need the compatibility condition on the velocity (\ref{co2}). More precisely,  there is no need to suppose  the following compatibility condition on the temperature
	\be\la{co1} \ka\Delta \te_0+\frac{\mu}{2}|\na u_0+(\na u_0)^{\rm tr}|^2+\lambda (\div u_0)^2=\sqrt{\n_0}g_1, \quad g_1 \in L^2,\ee
which is required in \cite{choe1, H-L}.
	\end{remark}

We now comment on the analysis of this paper.  Our strategy is first extend the local classical solutions without vacuum (see Lemma \ref{th0}) globally in all time provided that the initial energy is suitably small (see Proposition \ref{pro2}), then let the lower bound of the initial density tend to zero to obtain the global classical solutions with vacuum. To  do so, one needs to  establish global a priori estimates, which are  independent of the lower bound of the density,  on smooth solutions to the   problem (\ref{a1})--(\ref{ch2}) in suitable higher norms.
 It turns out that the key issue in this paper is to derive both
  the time-independent upper bound for the density and the
   time-dependent higher norm estimates of the smooth solution $(\rho, u,\theta)$.

Compared to    the Cauchy problem in 3D whole space  (\cite{H-L})   and the IBVP in 3D bounded domains (\cite{L-L-W}), we need to handle the difficulties on both the unbounded domains and the boundary estimates.
More precisely, for the standard energy $E(t)$ defined by
\be \label{enet}E(t) \triangleq \int  \left( \frac{1}{2}\n |u|^2+R(1+\n\log
\n-\n)+\frac{R}{\ga-1}\n(\te-\log \te-1)\right)dx, \ee the following  basic energy equality (or inequality) on $E(t)$
\be\ba\la{11a}
 E(t)+\int_0^T\int \left( \frac{\lambda(\div u)^2+2\mu |\mathfrak{D}(u)|^2}{{\te}}+\ka \frac{|\na \te|^2}{\te^2} \right)dxdt\le CC_0,
\ea\ee
which plays a crucial role in  the whole analysis of \cite{H-L}, is invalid  due to the slip boundary condition, see \eqref{la2.7}. Motivated by the ideas in \cite{L-L-W}, we can obtain the following ``weaker" basic energy estimate (see also \eqref{a2.8}):
\be\la{weaken} E(t) \le CC_0^{1/4},\ee
under a priori assumption on $A_2(T)$ (see \eqref{3.q2}). However, the ``weaker" basic energy estimate \eqref{weaken} will  bring us some essential difficulties in controlling all the a priori estimates (see Proposition \ref{pr1}).

Then,  following the similar arguments in \cite{L-L-W} concerning the nonlinear coupling of $\te$ and $u$ (see \eqref{a2.225}--\eqref{a2.23}),  we can establish the energy-like estimate on $A_2(T)$ which includes the key bounds on the $L^2_tL^2_x$-norm of $(\na u, \na \te)$, provided the initial energy is small (see Lemma \ref{le3}). Indeed, we will first  adopt the approach due to Hoff \cite{Hof1} (see also Huang-Li \cite{H-L})  to  close the a priori estimates  $A_3(T)$ concerning the bounds on   the  $L^\infty_tL^2_x$-norm of $(\na u, \na \te)$ and some elementary estimates on $(\dot u, \dot \te)$.  To proceed, we adopt some  ideas in  Cai-Li-L\"u \cite{C-L-L} to handle the arguments on $(\nabla u,~\div u,~\curl u)$ and the boundary integrals as follows:
\begin{itemize}
  \item On the one hand,   we establish some necessary inequalities related to $\nabla u$, $\div u$, and $\curl u$ in  the exterior domains, which are important to estimate   $\nabla u$ by means of $\div u$ and $\curl u$, see Lemmas \ref{crle1}-\ref{crle5} in Section 2.
  \item On the other hand, we introduce a smooth `cut-off' function defined on a ball containing $\r^3-\Omega$, which not only enables us to eliminate the first derivative in the boundary integral above by using the method in Cai-Li \cite{C-L} based on the divergence theorem and the fact $u=u^\perp\times n$ (see \eqref{eee1}) on the boundary, but also reduces the boundary integral to the integral on the ball. For example, one can see the detailed calculations in \eqref{bz3} for the following boundary integral
      $$\int_{\partial\Omega} G (u\cdot \na) u\cdot\na n\cdot u dS.$$
\end{itemize}

Thus, with the lower estimates at hand, one can derive   the crucial time-independent upper bound of the density (see Lemma \ref{le7}) and then obtain the higher-order estimates (see Section 4)
just under the compatibility condition on the velocity. Note  that all the a priori estimates are independent of the lower bound of the initial density, thus after a standard approximation procedure, we can obtain the global existence of classical solutions with vacuum.

The rest of the paper is organized as follows: In Section 2, we collect some basic facts and inequalities which will be used later. Section 3 is devoted to deriving the lower-order a priori estimates on classical solutions which are needed to extend the local solutions to all time. The higher-order estimates are established in Section 4.
Finally, with all a priori estimates at hand, the main result Theorem \ref{th1} is proved in Section 5.

\section{Preliminaries}\la{se2}


First, the following well-known local existence theory with strictly positive initial density, can be shown by the standard contraction mapping argument as in \cite{choe1,Tani,M1}.

\begin{lemma}   \la{th0} Let $\O$ be as in Theorem \ref{th1}. Assume  that
 $(\n_0,u_0,\te_0)$ satisfies \be \la{2.1}
\begin{cases}
(\n_0-1,u_0,\te_0-1)\in H^3, \quad \inf\limits_{x\in\Omega}\n_0(x) >0, \quad \inf\limits_{x\in\Omega}\te_0(x)> 0,\\
u_0\cdot n=0,~~\curl u_0\times n=0,~~ \na \te_0\cdot n =0,~~~~\text{on}~\p\Omega.
\end{cases}\ee
Then there exist  a small time
$0<T_0<1$ and a unique classical solution $(\rho , u,\te )$ to the problem  (\ref{a1})--(\ref{ch2}) on $\Omega\times(0,T_0]$ satisfying
\be\la{mn6}
  \inf\limits_{(x,t)\in\Omega\times (0,T_0]}\n(x,t)\ge \frac{1}{2}
 \inf\limits_{x\in\Omega}\n_0(x), \ee and
 \be\la{mn5}
 \begin{cases}
 ( \rho-1,u,\te-1) \in C([0,T_0];H^3),\quad
 \n_t\in C([0,T_0];H^2),\\   (u_t,\te_t)\in C([0,T_0];H^1),
 \quad (u,\te-1)\in L^2(0,T_0;H^4).\end{cases}\ee
 \end{lemma}

 \begin{remark} Applying the same arguments as in \cite[Lemma 2.1]{H-L}, one can deduce that the classical solution $(\rho , u,\te )$  obtained in Lemma \ref{th0} satisfies
\be \la{mn1}\begin{cases} (t u_{t},  t \te_{t}) \in L^2( 0,T_0;H^3)  ,\quad (t u_{tt},  t\te_{tt}) \in L^2(
0,T_0;H^1), \\  (t^2u_{tt},  t^2\te_{tt}) \in L^2( 0,T_0;H^2), \quad
 (t^2u_{ttt},  t^2\te_{ttt}) \in L^2(0,T_0;L^2).
\end{cases}\ee
Moreover, for any    $(x,t)\in \Omega\times [0,T_0],$ the following estimate holds:
   \be\la{mn2}
\te(x,t)\ge
\inf\limits_{x\in\Omega}\te_0(x)\exp\left\{-(\ga-1)\int_0^{T_0}
 \|\div u\|_{L^\infty}dt\right\}.\ee
\end{remark}

The following well-known Gagliardo-Nirenberg-Sobolev-type inequality (see \cite{fcpm}) will be used later frequently.

\begin{lemma}
 \la{11} Assume that $\Omega$ is the exterior of a simply connected domain $D$ in $\r^3$. For $r\in [2,6]$, $p\in(1,\infty)$, and $ q\in
(3,\infty),$ there exist  positive
 constants $C$ which may depend  on $r,\,p,\, q$ such that for $f\in H^1(\Omega)$,
 $g\in L^p(\Omega)\cap D^{1,q}(\Omega)$,  and $\varphi,\,\psi\in H^2(\Omega)$,
 \be
\la{g1}\|f\|_{L^r} \le C \| f\|_{L^2}^{{(6-r)}/{2r}}  \|\na f\|_{L^2}^{{(3r-6)}/{2r}},\ee
\be \la{g2}\|g\|_{C\left(\overline{\Omega}\right)}
       \le C \|g\|_{L^p}^{p(q-3)/(3q+p(q-3))}\|\na g\|_{L^q}^{3q/(3q+p(q-3))},\ee
\be\la{hs} \|\varphi\psi\|_{H^2}\le C \|\varphi\|_{H^2}\|\psi\|_{H^2}.\ee
\end{lemma}

Then, the following inequality is an consequence of \eqref{g1}, which will play an important role in our analysis.
\begin{lemma}
	\la{lll} Let the function $g(x)$ defined in $\Omega$ be non-negative and satisfy $g-1\in L^2(\Omega)$. Then there exist a universal positive constant C such that for $s\in[1,2]$ and any open set $\Sigma\in\Omega$, the following estimate holds
	\be\la{eee8}\int_\Sigma|f|^sdx\leq C\int_\Sigma g|f|^sdx+C\|g-1\|_{L^2(\Omega)}^{(6-s)/3}\|\na f\|_{L^2(\Omega)}^s\ee
	for all $f\in\{D^1(\Omega)~\big|~g|f|^s\in L^1(\Sigma)\}$.
\end{lemma}
\begin{proof}
	In fact, take $r=6$ in \eqref{g1}, we have
	\begin{equation*}\ba
		\int_\Sigma|f|^sdx\leq &\int_\Sigma g|f|^sdx+\int_\Sigma|g-1||f|^sdx\\
		\leq &\int_\Sigma g|f|^sdx+\|g-1\|_{L^2(\Omega)}\|f\|_{L^s(\Sigma)}^{s(3-s)/(6-s)}\|f\|_{L^6(\Omega)}^{3s/(6-s)}\\
		\leq &\int_\Sigma g|f|^sdx+\frac{1}{2}\int_\Sigma|f|^sdx+C\|g-1\|_{L^2(\Omega)}^{(6-s)/3}\|\na f\|_{L^2(\Omega)}^s,
			\ea\end{equation*}
	which implies \eqref{eee8} directly. The proof of Lemma \ref{lll} is completed.
\end{proof}

Considering the Neumann boundary value problem
\be\la{cxtj1}\begin{cases}
	-\Delta v=\div f\,\, &\text{in}~ \Omega, \\
	\frac{\partial v}{\partial n}=-f\cdot n\,&\text{on}~\partial\Omega,\\
	\nabla v\rightarrow0,\,\,&\text{as}\,\,|x|\rightarrow\infty,
\end{cases} \ee
where $v=(v^{1},v^{2},v^{3})^{\text{tr}}$ and $f=(f^{1},f^{2},f^{3})^{\text{tr}}$, it's easy to find that the problem \eqref{cxtj1} is equivalent to
$$\int\nabla v\cdot\nabla\varphi dx=\int f\cdot\nabla\varphi dx,\,\,\forall\varphi\in C_0^{\infty}(\Omega).$$
Thanks to \cite[Lemma 5.6]{ANIS}, we have the following conclusion.
\begin{lemma}   \la{zhle}
	Considering the system \eqref{cxtj1}, for $q\in(1,\infty),$ one has
	
	(1) If $f\in L^q$, then there exists a unique (modulo constants) solution $v\in D^{1,q}$ such that
	$$\|\nabla v\|_{L^q}\leq C(q,\O)\|f\|_{L^q}.$$
	
	(2) If $f\in W^{k,q}$ with $k\geq 1$, it holds that $\na v\in W^{k,q}$ and
	$$\|\nabla v\|_{W^{k,q}}\leq C(k,q,\O) \|f\|_{W^{k,q}}.$$	
	In particular, if $f\cdot n=0$ on $\partial\Omega$, it holds
	$$\|\nabla^2v\|_{L^q}\leq C(q,\O) \|\div f\|_{L^q}.$$	
\end{lemma}

 The $L^p$-estimate of $\na v$ for $v$ satisfying the boundary condition $v\cdot n=0$ or $v\times n=0$ on $\partial\Omega$, is shown in the following Lemmas \ref{crle1}, \ref{crle2}, and  \ref{crle5}, whose proof can be found in \cite[Theorem 3.2]{vww}, \cite[Theorem 5.1]{lhm}, and \cite[Lemma 2.9]{C-L-L}, respectively.

\begin{lemma}  \cite[Theorem 3.2]{vww} \la{crle1}
Let $\Omega=\r^3-\bar{D}$ be the exterior of a simply connected bounded domain $D\subset\r^3$  with $C^{1,1}$ boundary.
For $v\in D^{1,q}$ with $v\cdot n=0$ on $\partial\Omega$, it holds that
	\be\la{ljq01}\|\nabla v\|_{L^{q}}\leq C(q,D)(\|\div v\|_{L^{q}}+\|\curl v\|_{L^{q}})~~~\,\,\,for\, any\,\, 1<q<3,\ee
	and
	$$\|\nabla v\|_{L^q}\leq C(q,D)(\|\div v\|_{L^q}+\|\curl v\|_{L^q}+\|\nabla v\|_{L^2})~~~\,\,\,for\, any\,\, 3\leq q<+\infty.$$
\end{lemma}

\begin{lemma} \cite[Theorem 5.1]{lhm}  \la{crle2}
	Let $\Omega$ be given in Lemma \ref{crle1}, for any $v\in W^{1,q}\,\,(1<q<+\infty)$ with $v\times n=0$ on $\partial\Omega$, it holds that
	$$\|\nabla v\|_{L^q}\leq C(q,D)(\|v\|_{L^q}+\|\div v\|_{L^q}+\|\curl v\|_{L^q}).$$
\end{lemma}

\begin{lemma}  \cite[Lemma 2.9]{C-L-L} \la{crle5}
Assume  $\Omega=\r^3-\bar{D}$  is the same as in Theorem \ref{th1}.
For any $q\in[2,4] $,    there exists some positive constant $C=C(q,D)$  such that for  every $v\in \{ D^{1,2}(\O)| v(x)\rightarrow 0  \mbox{ as } |x|\rightarrow\infty  \}$, it holds
\be\la{eee0}\ba\|v\|_{L^q(\partial\Omega)}\leq C\|\na v\|_{L^2(\Omega)}.\ea\ee
Moreover, for $p\in[2,6] $ and $k\geq 1,$  if $v\in \{D^{k+1,p}\cap D^{1,2}| v(x)\rightarrow 0  \mbox{ as } |x|\rightarrow\infty  \}$  with $v\cdot n|_{\partial\Omega}=0$ or $v\times n|_{\partial\Omega}=0$, then there exists some constant $C=C(p,k,D)$ such that
	\be\la{uwkq}\ba\|\nabla v\|_{W^{k,p}}\leq C(\|\div v\|_{W^{k,p}}+\|\curl v\|_{W^{k,p}}+\|\nabla v\|_{L^2}).\ea\ee
\end{lemma}

Now, we will give some a priori estimates for $G$, $\curl u$, and $\nabla u$, which will be used frequently later.

\begin{lemma} \la{le4}
  Assume  $\Omega=\r^3-\bar{D}$  is the same as in Theorem \ref{th1}. Let $(\rho,u,\te)$ be a smooth solution of
   (\ref{a1})--(\ref{ch2}).
    Then for any $p\in [2,6],$ there exists a generic positive
   constant $C$ depending only on $p$, $\mu,$   $\lambda,$ $R$,  and $D$ such that
   \be\la{h19}\|\na G\|_{L^p} \le C\|\rho\dot{u}\|_{L^p},\ee
   \be\la{h191}  \|{\nabla \curl u}\|_{L^p}
   \le C (\|\rho\dot{u}\|_{L^p} + \|\rho\dot{u}\|_{L^2} + \|\na {u}\|_{L^2}),\ee
   \be  \la{h20}\|G\|_{L^p}
   \le C \|\rho\dot{u}\|_{L^2}^{(3p-6)/(2p) }
   \left(\|{\nabla u}\|_{L^2}
   +  \|\rho\te-1\|_{L^2}\right)^{(6-p)/(2p)},\ee
   \be  \la{h21} \|\curl u\|_{L^p}
   \le C \|\rho\dot{u}\|_{L^2}^{(3p-6)/(2p) }
   \|{\nabla u}\|_{L^2}
   ^{(6-p)/(2p)} + C\|{\nabla u}\|_{L^2}.\ee
  Moreover, it holds that
  \be\ba\la{h17} \|\na u\|_{L^p}\le& C \left( \|\rho\dot{u}\|_{L^2}+\|\rho\te-1\|_{L^6}\right)^{(3p-6)/(2p) }  \|{\nabla
  	u}\|_{L^2}^{(6-p)/(2p)} + C\|{\nabla	u}\|_{L^2}.
  \ea\ee
\end{lemma}
\begin{proof}
	By $(\ref{a1})_2 $, it is easy to find that $G$ satisfies
	$$\int\nabla G\cdot\nabla\varphi dx=\int\rho\dot{u}\cdot\nabla\varphi dx,\,\,\forall\varphi\in C_0^{\infty}(\Omega).$$
	Consequently, by Lemma \ref{zhle}, it holds that for $q\in(1,\infty)$,
	\be\la{x266}\ba
	\|\nabla G\|_{L^q}\leq C\|\rho\dot{u}\|_{L^q},
	\ea\ee
	and for any integer $k\geq 1$,
	\be\la{x2666}\ba
	\|\nabla G\|_{W^{k,q}}\leq C\|\rho\dot{u}\|_{W^{k,q}}.
	\ea\ee
	
	On the other hand, one can rewrite $(\ref{a1})_2 $    as follows
	\be\la{a111}\ba\n \dot u&=\na G- \mu \na \times \curl u.
\ea\ee
	Notice that $\curl u\times n|_{\partial\Omega}=0$  and $\div(\nabla\times\curl u)=0,$ we deduce from Lemmas \ref{crle2}--\ref{crle5} and \eqref{x266} that
	\be\la{x267}\ba
	\|\nabla\curl u\|_{L^q}&\leq C(\|\nabla\times\curl u\|_{L^q}+\|\curl u\|_{L^q})\\
	&\leq C(\|\rho\dot{u}\|_{L^q}+\|\curl u\|_{L^q}),
	\ea\ee
	and for any integer $k\geq 1$,
	\be\la{x268}\ba
	\|\nabla\curl u\|_{W^{k,p}}&\leq C(\|\nabla\times\curl u\|_{W^{k,p}}+\|\nabla\curl u\|_{L^2})\\
	&\leq C(\|\rho\dot{u}\|_{W^{k,p}}+\|\rho\dot{u}\|_{L^2}+\|\nabla u\|_{L^2}).
	\ea\ee
	Therefore, by Gagliardo-Nirenberg's inequality and \eqref{x267}, one gets for $p\in[2,6]$,
	\bnn\ba
	\|\nabla\curl u\|_{L^p}&\le C(\|\rho\dot{u}\|_{L^p}+\|\curl u\|_{L^p} ) \\
	&\le C(\|\rho\dot{u}\|_{L^p}+\|\nabla\curl u\|_{L^2}+\|\curl u\|_{L^2}) \\
	&\le C(\|\rho\dot{u}\|_{L^p}+\|\rho\dot{u}\|_{L^2}+\|\curl u\|_{L^2}) \\
	&\le C(\|\rho\dot{u}\|_{L^p}+\|\rho\dot{u}\|_{L^2}+\|\nabla u\|_{L^2}),
	\ea\enn
	which gives \eqref{h191}.
	
	Furthermore, it follows from   \eqref{g1} and \eqref{h19} that for $p\in[2,6]$,
	\be\la{x2610}\ba
	\|G\|_{L^p}&\leq C\|G\|_{L^2}^{(6-p)/2p}\|\nabla G\|_{L^2}^{(3p-6)/2p}\\
	&\leq C\|\rho\dot{u}\|_{L^2}^{(3p-6)/2p}(\|\nabla u\|_{L^2}+\|\n\te-1\|_{L^2})^{(6-p)/2p}.
	\ea\ee
	Similarly, it has
	\be\la{x2612}\ba
	\|\curl u\|_{L^p}&\leq C\|\curl u\|_{L^2}^{(6-p)/(2p)}\|\nabla\curl u\|_{L^2}^{(3p-6)/(2p)}\\
	&\leq C(\|\rho\dot{u}\|_{L^2}+\|\nabla u\|_{L^2})^{(3p-6)/(2p)}\|\nabla u\|_{L^2}^{(6-p)/(2p)}\\
	&\leq C\|\rho\dot{u}\|_{L^2}^{(3p-6)/(2p)}\|\nabla u\|_{L^2}^{(6-p)/(2p)}+C\|\nabla u\|_{L^2}.
	\ea\ee
	
	Finally, by virtue of Lemma \ref{crle1}, \eqref{g1}, \eqref{h19}, and \eqref{h21}, it indicates that
	\bnn\ba
	\|\nabla u\|_{L^p}&\le C\|\nabla u\|_{L^2}^{(6-p)/(2p)}\|\nabla u\|_{L^6}^{(3p-6)/(2p)} \\
	&\le C\|\nabla u\|_{L^2}^{(6-p)/(2p)}(\|\rho\dot{u}\|_{L^2}+\|\n\te-1\|_{L^6})^{(3p-6)/(2p)}+C\|\nabla u\|_{L^2},
	\ea\enn where in the second inequality we have used
	\bnn\ba
	\|\nabla u\|_{L^6}  &\le C (\|\div u\|_{L^6}+\|\curl u\|_{L^6}+\|\nabla u\|_{L^2}) \\
	&\le C (\|G\|_{L^6}+\|\curl u\|_{L^6}+\|\n\te-1\|_{L^6}+\|\nabla u\|_{L^2})  \\
	&\le C (\|\rho\dot{u}\|_{L^2}+\|\n\te-1\|_{L^6} + \|\nabla u\|_{L^2}),
	\ea\enn due to   \eqref{x2610}--\eqref{x2612}  with $p=6.$
	The proof of Lemma \ref{le4} is finished.
\end{proof}

Next, we give the following estimate  on $\na \dot u$ with $u \cdot n|_{\p \O}=0$.

\begin{lemma}\la{uup1}
	Let  $\Omega=\r^3-\bar{D}$  is the same as in Theorem \ref{th1}.
	Assume that $u$ is smooth enough and $u \cdot n|_{\p \O}=0$,
	then there exists a generic positive constant $C=C(D)$   such that
	\be\la{tb11}\ba
	\|\nabla\dot{u}\|_{L^2}\le C(\|\div \dot{u}\|_{L^2}+\|\curl \dot{u}\|_{L^2}+\|\nabla u\|_{L^4}^2+\|\nabla u\|_{L^2}^2).
	\ea\ee
\end{lemma}
\begin{proof}
	Denote $u^{\perp}\triangleq-u\times n$,
	it follows from the boundary condition (\ref{h1}) that
	\be\la{pzw1} u\cdot\nabla u\cdot n=-u\cdot\nabla n\cdot u ~~\quad \mbox{on}~\p \O\ee
	and
	\begin{align}\la{eee1}
		u=u^{\perp} \times n~~~~~\text{on}\, \p \O,
	\end{align}
     which gives
     $$(\dot{u}-(u\cdot\nabla n)\times u^{\perp})\cdot n=0 ~~\quad \mbox{on}~\p \O.$$
     This combined with (\ref{ljq01}) yields
     \be\la{tb111} \ba\|\nabla\dot{u}\|_{L^2}
     &\leq C(\|\div \dot{u}\|_{L^2}+\|\curl \dot{u}\|_{L^2}+\|\nabla[(u\cdot\nabla n)\times u^\perp]\|_{L^2})\\&\leq C(\|\div \dot{u}\|_{L^2}+\|\curl \dot{u}\|_{L^2}+\|\nabla u\|_{L^2}^2+\|\nabla u\|_{L^4}^2),
     \ea\ee where in the second inequality we have used
     \bnn \ba
     \|\nabla[(u\cdot\nabla n)\times u^\perp]\|_{L^2(\Omega)}
     &=\|\nabla[(u\cdot\nabla n)\times u^\perp]\|_{L^2(B_{2d})}\\
     &\leq C(d)(\||u||\nabla u|\|_{L^2(B_{2d})}+\|u\|_{L^4(B_{2d})}^2)\\
     &\le C(d)(\|\nabla u\|_{L^4(B_{2d})}^2+\|u\|_{L^6(B_{2d})}^2)\\
     &\le C(\|\nabla u\|_{L^4}^2+\|u\|_{L^6}^2)\\
     &\le C(\|\nabla u\|_{L^4}^2+\|\nabla u\|_{L^2}^2),
     \ea\enn
     due to  \eqref{ljq10}, \eqref{g1}, and H\"{o}lder's inequality.  The proof of Lemma \ref{uup1} is finished.
\end{proof}

The following Gr\"{o}nwall-type inequality will be used to get the uniform (in time) upper bound of the density $\n$, whose proof is similar to \cite[Lemma 2.5]{H-L}.
\begin{lemma} \la{le1}
	Let the function $y\in W^{1,1}(0,T)$ satisfy
	\be \la{2.32}
	y'(t)+\al y(t)\le  g(t)\mbox{  on  } [0,T] ,\quad y(0)=y_0,
	\ee
	where $\al$ is a positive constant and  $ g \in L^p(0,T_1)\cap L^q(T_1,T)$  for some $p,\,q\ge 1, $  $T_1\in [0,T].$ Then it has
	\be \la{2.34}
	\sup_{0\le t\le T} y(t) \le |y_0| + (1+\al^{-1}) \left(\|g\|_{L^p(0,T_1)} + \|g\|_{L^q(T_1,T)}\right).
	\ee
\end{lemma}

Finally, in order to estimate $\|\nabla u\|_{L^\infty}$ for the further higher order estimates, we need the following Beale-Kato-Majda-type inequality, which was first proved in \cite{bkm,kato} when $\div u\equiv 0,$
whose detailed proof is similar to that of the case of slip boundary condition  in \cite[Lemma 2.7]{C-L} (see also \cite{h101,h1x}).
\begin{lemma}\la{le9}
	Let  $\Omega=\r^3-\bar{D}$  is the same as in Theorem \ref{th1}. For $3<q<\infty$, assume that $u\in \{ f\in L^1_{loc}|\na f\in L^2(\O)\cap D^{1,q}(\O) \,and \,f\cdot n=0, \curl f\times n=0 \text{ on } \partial \Omega \}$,  then there is a constant  $C=C(q)$ such that
\bnn\ba
\|\na u\|_{L^\infty}\le C\left(\|{\rm div}u\|_{L^\infty}+\|\curl u\|_{L^\infty} \right)\ln(e+\|\na^2u\|_{L^q})+C\|\na u\|_{L^2} +C.
\ea\enn
\end{lemma}

\section{\la{se3} A priori estimates (I): lower-order estimates}

This section focuses on the a priori bounds for the local-in-time smooth solution to problem (\ref{a1})--(\ref{ch2}) obtained in Lemma \ref{th0}.
Let $(\n,u,\te)$ be a smooth solution to the  problem (\ref{a1})--(\ref{ch2})  on $\Omega\times (0,T]$ for some fixed time $T>0,$ with  initial data $(\n_0,u_0,\te_0)$ satisfying \eqref{2.1}.

For
$\si(t)\triangleq\min\{1,t\}, $  we  define
$A_i(T)(i= 1,2,3)$ as follows:
  \be\la{As1}
  A_1(T) \triangleq \sup_{t\in[0,T] }\|\nabla u \|_{L^2}^2
  + \int_0^{T}\int  \rho|\dot{u}|^2dxdt,
  \ee
  \be\label{AS1}
  A_2(T) \triangleq \frac{R}{2(\ga-1)}\sup_{t\in[0,T] }\int\n (\te-1)^2dx
  +\int_0^T\left( \|\na u\|_{L^2}^2+\|\na \te\|_{L^2}^2\right)dt,
  \ee
  \be\ba \label{AS2}
  A_3(T) \triangleq &\sup_{t\in(0,T]}\left(\si \|\na u\|_{L^2}^2+\si^2\int\n |\dot u|^2dx + \si^2\|\na\te\|_{L^2}^2 \right)\\
  & + \int_0^T\int\left(\si\n |\dot u|^2 +\sigma^2|\nabla\dot{u}|^2 +\sigma^2\n|\dot \te|^2 \right)dxdt.
  \ea\ee

We have the following key a priori estimates on $(\n,u,\te)$.
\begin{pro}\la{pr1}
For  given numbers $M>0$, $\on> 2,$  and $\bt> 1,$ assume further that $(\rho_0,u_0,\te_0) $  satisfies
\be \la{3.1}
0<\inf \rho_0 \le\sup \rho_0 <\on,\quad 0<\inf \te_0 \le\sup \te_0 \le \bt, \quad \|\na u_0\|_{L^2} \le M.
\ee
Then there exist  positive constants $K $ and $\ep_0$ both depending on $\mu,\,\lambda,\, \ka,\, R,\, \ga,\, \on,\,\bt,\,\O $, and $M$ such that if $(\rho,u,\te)$ is a smooth solution to the problem (\ref{a1})--(\ref{ch2}) on $\O\times (0,T]$ satisfying
\be \la{z1}
0<  \rho\le 2\on, \,\,\, A_1(\sigma(T))\le 3 K, \,\,\, A_i(T) \le 2C_0^{1/(2i)}, \,\,\,(i=2,3),
\ee
the following estimates hold:
\be \la{zs2}
0< \rho\le 3\on/2,\,\,\, A_1(\sigma(T))\le 2 K, \,\,\, A_i(T) \le C_0^{1/(2i)},  \,\,\,(i=2,3),
\ee
provided \be\la{z01}C_0\le \ve_0.\ee      \end{pro}

\begin{proof}
Proposition \ref{pr1} is a straight consequence of
the following Lemmas \ref{le2}, \ref{le6}, \ref{le3}, and \ref{le7},  with $\ve_0$ as in (\ref{xjia11}).
\end{proof}

In this section, we always assume that $C_0\le 1$ and let $C$ denote some generic positive constant depending only on $\mu$,  $\lambda$, $\ka$,  $R$, $\ga$, $\on$, $\bt$, $\O$,  and $M,$ and we write $C(\al)$ to emphasize that $C$ may depend  on $\al.$

First, we have the following basic energy estimate, which plays an important role in the whole analysis.

\begin{lemma}\la{a13.1} Under the conditions of Proposition \ref{pr1}, there exists a positive constant $C$ depending on $\mu$, $R$,   and  $\on$  such that if $(\rho,u,\te)$ is a smooth solution to the problem (\ref{a1})--(\ref{ch2})  on $\Omega\times (0,T] $ satisfying
\be\la{3.q2}
0<\n\le 2\on ,\quad A_2(T)\le 2C_0^{1/4},
\ee
the following estimate holds:
\be \la{a2.112}
\sup_{0\le t\le T}\int\left( \n |u|^2+(\n-\tn)^2\right)dx \le C  C_0^{1/4}.
\ee
\end{lemma}

\begin{proof}
First, it follows from (\ref{3.1}) and  (\ref{mn2}) that, for all $ (x,t)\in \Omega\times(0,T),$
\be \la{3.2}
\te(x,t)>0 .
\ee

Note that
\be\la{1111}
\Delta u = \na \div u - \na \times \curl u,
\ee
one can rewrite $(\ref{a1})_2$  as
\be\la{a11}\ba
\n (u_t+  u\cdot \na u )&=(2\mu+\lambda)\na{\rm div}u- \mu \na \times \curl u-\na P.\ea\ee
Adding $(\ref{a11})$ multiplied by $u$ to $(\ref{a1})_3$ multiplied by $1-\te^{-1}$ and  integrating the resulting equality over $\Omega$ by parts,  we obtain
after   using $(\ref{a1})_1$, \eqref{h3}, \eqref{3.2}, \eqref{h1},  and  \eqref{ch2}   that
\be\la{la2.7}\ba
E'(t)&=-\int \left( \frac{\lambda(\div u)^2+2\mu |\mathfrak{D}(u)|^2}{{\te}}+\ka \frac{|\na \te|^2}{\te^2} \right)dx \\
&\quad - \mu \int \left(|\curl u|^2+2(\div u)^2 - 2 |\mathfrak{D}(u)|^2\right)dx\\
&\le 2\mu \int  |\na u|^2 dx,
\ea\ee
where  $E(t)$ is  the basic energy defined by \eqref{enet}.

Then, integrating \eqref{la2.7} with respect to $t$ over $(0,T)$ and using \eqref{3.q2}, one has
 \be\la{a2.8}\ba
&\sup_{0\le t\le T} E(t)
\le C_0+ 2\mu \int_{0}^{T} \int  |\na u|^2 dxdt\le C C_0^{1/4},\ea\ee
which together with
\be\la{a2.9}\ba
 (\n-1)^2\ge 1+\n\log\n-\n&=(\n-1)^2\int_0^1\frac{1-\al}{\al (\n-1)+1}d\al  \ge \frac{(\n-1)^2}{ 2(2\on+1)  }
 \ea\ee
gives (\ref{a2.112}).  The proof of Lemma \ref{a13.1} is  finished.
\end{proof}

The next lemma provides an estimate on  $A_1(\sigma(T))$.
\begin{lemma}\la{le2}
	Under the conditions of Proposition \ref{pr1}, there exist positive constants  $K $  and $\ep_1 $ both depending only  on $\mu,\,\lambda,\, \ka,\, R,\, \ga,\, \on,\,\bt, $ $\O$, and $M$ such that if  $(\rho,u,\te)$ is a smooth solution to the problem  (\ref{a1})--(\ref{ch2}) on $\Omega\times (0,T] $ satisfying
	\be\la{3.q1}  0<\n\le 2\on ,\quad A_2(\sigma(T))\le 2C_0^{1/4},\quad A_1(\sigma(T))\le 3K,\ee
	the following estimate holds:
	\be\la{h23} A_1(\sigma(T))\le 2K ,  \ee
	provided   $C_0\le \ep_1$ with $\ep_1$ given in (\ref{3.q0}).
\end{lemma}

\begin{proof}
 First, integrating $(\ref{a11})$  multiplied by $2u_t $ over $\Omega $ by parts gives
 \be\ba \la{hh17}
 &\frac{d}{dt}\int \left(  {\mu} |\curl u|^2+ (2\mu+\lambda)(\div u)^2\right)dx+ \int\rho |\dot u|^2dx \\
&\le 2\int  P\div u_t dx+ \int \n|u\cdot \na u|^2dx\\
 &=  2R\frac{d}{dt}\int  (\rho\te-1) \div u  dx-2\int P_t \div u dx+\int \n|u\cdot \na u|^2dx\\
 &= 2R\frac{d}{dt}\int  (\rho\te-1) \div u  dx-\frac{R^2}{2\mu+\lambda}\frac{d}{dt}\int (\rho\te-1)^2 dx\\
 &\quad-\frac{2}{2\mu+\lambda}\int P_t G dx+ \int \n|u\cdot \na u|^2dx ,
 \ea\ee
where in the last equality  we have used (\ref{hj1}).

Next, it follows from Holder's inequality, \eqref{le2}, \eqref{g1}, and \eqref{a2.112} that for any $p\in [2,6],$
\be\la{p}\ba  \|\rho\te-1\|_{L^p}&= \| \rho(\te-1) + (\rho-1)\|_{L^p}\\
 &\le \|\n(\te-1)\|_{L^2}^{(6-p)/(2p)}
 \|\n(\te-1)\|_{L^6}^{ 3(p-2)/(2p)}+ \|\n-1\|_{L^p}
 \\&\le C(\on)C_0^{(6-p)/(16p)}
 \|\na\te \|_{L^2}^{ 3(p-2)/(2p)}+ C(\hat\n)C_0^{1/(4p)},
  \ea\ee
which together with (\ref{h17}) yields
\be \la{3.30}  \|\na
u\|_{L^6} \le C(\hat\n) \left(  \|\n \dot
u\|_{L^2}+\|\na u\|_{L^2}+ \|\na \te\|_{L^2}+C_0^{1/24}\right).
\ee
Noticing that (\ref{a1})$_3$ implies
\be \la{op3} \ba
P_t=&-\div (Pu) -(\gamma-1) P\div u+(\ga-1)\ka \Delta\te\\&+(\ga-1)\left(\lambda (\div u)^2+2\mu |\mathfrak{D}(u)|^2\right),
\ea\ee
we thus obtain after using  integration by parts,  (\ref{g1}), (\ref{h19}), (\ref{p}),  (\ref{3.30}), and (\ref{a2.112}) that
\be\la{a16}\ba
&\left|\int   P_t Gdx\right| \\
&\le C\int P(|G||\na u|+ |u||\na G|)dx+ C\int\left( |\na\te||\na G|+|\na u|^2|G|\right)dx \\
&\le C\int \n(|G||\na u|+|u||\na G|)+C\int\n|\te-1| (|G||\na u|+|u||\na G|)dxdx \\
&\quad + C \|\na G\|_{L^2} \|\na \te\|_{L^2} + C \| G\|_{L^6} \|\na u\|_{L^2}^{3/2}  \|\na u\|_{L^6}^{1/2} \\
&\le C(\hat \n)(\|\na u\|_{L^2}+\|\n\te-1\|_{L^2})\|\na u\|_{L^2}+ C\|\n u\|_{L^2} \|\na G\|_{L^2}  \\
& \quad+ C\|\n(\te-1)\|_{L^2}^{1/2}\|\na\te\|_{L^2}^{1/2}\|\na G\|_{L^2}\|\na u\|_{L^2}
+C \|\na G\|_{L^2} \|\na \te\|_{L^2}\\
& \quad+ C \|\na G\|_{L^2} \|\na u\|_{L^2}^{3/2} \left(\|\n\dot u\|_{L^2}+\|\na u\|_{L^2}+\|\na\te\|_{L^2}+1 \right)^{1/2} \\
&\le \de \|\na G\|_{L^2}^2 +\de \|\rho\dot u\|^2_{L^2}+C(\de,\on) \left( \|\na u\|_{L^2}^2+ \|\na \te\|_{L^2}^2+ \|\na u\|^6_{L^2}+1\right) \\
&\le C(\on)\de\|\rho^{1/2} \dot u\|^2_{L^2}     +C(\de,\on) \left( \|\na u\|_{L^2}^2+ \|\na \te\|_{L^2}^2 +\|\na  u\|^6_{L^2}+1\right),
\ea\ee

Then, it follows from (\ref{g1}), \eqref{3.q1}, and (\ref{3.30}) that
\be\la{op1}\ba
\int \n|u\cdot \na u|^2dx&\le C(\on)\|u\|_{L^6}^2 \|\na u\|_{L^2} \|\na u\|_{L^6}  \\
&\le \de\|\rho^{1/2} \dot u\|_{L^2}^2+ C(\de,\on)\left(\|\na u\|_{L^2}^2+\|\na\te\|_{L^2}^2+\|\na
u\|_{L^2}^6\right).
\ea\ee

Finally, substituting (\ref{a16})  and (\ref{op1}) into (\ref{hh17}) and choosing $\de$ suitably small,
one gets after integrating (\ref{hh17}) over $(0,\sigma(T))$ and using (\ref{3.q1}), \eqref{ljq01}, and \eqref{p} that
\bnn\la{h81} \ba
&\sup_{0\le t\le \sigma(T)}\|\na u\|_{L^2}^2+ \int_0^{\sigma(T)}\int\rho|\dot{u}|^2dxdt\\
&\le CM^2+C(\on)C_0^{1/4} + C(\on ) C_0^{1/4}\sup_{0\le t\le \sigma(T)}\|\na u\|_{L^2}^4\\
&\le K+9K^2C(\on)C_0^{1/4}\\
&\le 2K,
\ea \enn
with $K\triangleq CM^2+C(\on) +1$, provided
\be\la{3.q0}C_0\le \ep_1 \triangleq \min\left\{1,\xl(9C(\on)K\xr)^{-4}\right\}.\ee
The proof of Lemma \ref{le2} is completed.
\end{proof}

Next, we use the approach from Hoff \cite{Hof1} (see also Huang-Li \cite{H-L})  to establish the following elementary estimates on $\dot u$ and $\dot \te$, where the boundary terms are handled by the ideas in Cai-Li-L\"u \cite{C-L-L}.

\begin{lemma}\la{a113.4}
	Under the conditions of Proposition \ref{pr1}, let $(\rho,u,\te)$ be a smooth solution to the problem (\ref{a1})--(\ref{ch2}) on $\Omega\times (0,T] $ satisfying (\ref{z1}) with $K$ as in Lemma \ref{le2}.  Then  there exist positive constants $C$, $ C_1$, and $C_2$ depending only on $\mu,\,\lambda, \,k,\, R,\, \ga,\, \on,\,\bt, $ $\O$, and $ M$  such that, for any $\beta,\eta\in (0,1]$ and $m\geq0,$
the following estimates hold:
\be\ba  \la{an1}
(\sigma B_1)'(t) + \frac{3}{2}\int \sigma \rho |\dot u|^2dx
\le &  C C_0^{1/4} \sigma' +2\beta\si^2\|\n^{1/2}\dot\theta\|_{L^2}^2+C\si^2\|\na u\|_{L^4}^4\\ &+C\beta^{-1}\left(\|\na u\|_{L^2}^2+\|\na\te\|_{L^2}^2\right),
\ea\ee
\be\la{ae0}\ba
&\left(\sigma^{m}\|\rho^{1/2}\dot{u}\|_{L^2}^2\right)_t+C_1 \sigma^{m}\|\na\dot{u}\|_{L^2}^2\\
&\le - 2\left(\int_{\p \O}  \sigma^m (u \cdot \na n \cdot u) G dS\right)_t + C(\si^{m-1}\si'+\si^m)  \|\rho^{1/2} \dot u\|_{L^2}^2 \\&\quad+
C_2  \si^m \|\rho^{1/2} \dot \te\|_{L^2}^2+ C\|\na u\|^2_{L^2}+C \si^m \|\na u\|^4_{L^4} + C \si^m  \|\te \na u\|_{L^2}^2,\\\ea\ee
  and
 \be\la{nle7}\ba  &(\si^mB_2 )'(t)+\si^m \int\n|\dot \te|^2dx\\
&\le C \eta \si^m\|\na\dot u\|_{L^2}^2+C \|\na
\te \|_{L^2}^2+C\si^m \|\na u\|_{L^4}^4+C\eta^{-1} \si^m \|\te\na u\|_{L^2}^2,\ea\ee
where
\be\la{an2} B_1(t)\triangleq \mu\|\curl u\|_{L^2}^2+ (2\mu+\lambda) \|\div u\|_{L^2}^2-2 R\int \div u(\n\te-1) dx, \ee
and
 \be\la{e6}
B_2(t)\triangleq\frac{\ga-1}{R}\left(\ka \|\na
\te\|_{L^2}^2-2 \int (\lambda (\div u)^2+2\mu|\mathfrak{D}(u)|^2)\te dx\right).\ee
\end{lemma}
\begin{proof} The proof is divided into  the following three parts.

\noindent{}\textbf{Part I: The proof of (\ref{an1}).}

Multiplying
$(\ref{a1})_2 $ by $\sigma \dot{u}$ and integrating the resulting
equality over $\Omega $ by parts, one gets
\be\la{m0} \ba  \int \sigma
\rho|\dot{u}|^2dx & = \int (\sigma \dot{u}\cdot\nabla G - \si \mu \na \times \curl u\cdot\dot{u})dx  \\
&=\int_{\partial \O} \sigma (u\cdot\na u \cdot n) G dS - \int \sigma \div \dot{u} G dx - \mu \int  \si  \curl u \cdot \curl \dot{u} dx \\
& \triangleq \sum_{i=1}^{3}M_i. \ea \ee

For the term $M_1$, it can be deduced from \eqref{pzw1}, \eqref{eee0},  \eqref{h19},  and \eqref{z1} that
\be \ba \la{bb2}
M_1&=-\int_{\partial \O} \sigma  (u\cdot \na n\cdot u) G dS\\
 &\le C\sigma \|u\|_{L^4(\partial\Omega)}^2\|G\|_{L^2(\partial\Omega)}\\
 &\le C\sigma \|\na u\|_{L^2} ^2\|\na G\|_{L^2}\\
 &\le \delta \sigma\|\rho^{1/2} \dot u\|_{L^2}^2+C(\de,\on,M) \sigma\|\na u\|_{L^2}^2.
\ea \ee
where in the last inequality we have used the following simple fact:
\be \la{infty12}
\sup_{t\in[0,T]}\|\na u\|_{L^2}\leq A_1(\sigma(T))+A_3(T)\leq C(\on,M).
\ee

Notice that
\be \ba \label{jia1}
P_t=(R\n \te)_t=R\n \dot{\te}-\div (Pu),
\ea \ee
which along with some straight calculations gives
\be \ba\la{pt1}
\div \dot u G=& (\div u_t + \div (u \cdot \na u))((2\mu+\lambda)\div u - R(\rho\theta-1))\\
=& \frac{2\mu+\lambda}{2} (\div u)^2_t - (R(\rho\theta-1) \div u)_t +(R\rho\theta)_t \div u \\
&+ (2\mu+\lambda) \div (u \cdot \na u)\div u - R(\rho\theta-1) \div (u \cdot \na u)\\
=& \frac{2\mu+\lambda}{2} (\div u)^2_t - (R(\rho\theta-1)\div u)_t +R\rho \dot\te \div u - \div(Pu) \div u\\
& + (2\mu+\lambda) \div (u \cdot \na u)\div u -R(\rho\theta-1)\div (u \cdot \na u)\\
=& \frac{2\mu+\lambda}{2} (\div u)^2_t - (R(\rho\theta-1)\div u)_t +R\rho \dot\te \div u \\
& + (2\mu+\lambda) \na u : (\na u)^{\rm tr} \div u + \frac{2\mu+\lambda}{2} u \cdot \na(\div u)^2 \\
&- \div(R(\rho\theta-1)u \div u)- R(\rho\theta-1) \na u : (\na u)^{\rm tr}-R(\div u)^2.
\ea \ee
This together with integration by parts and \eqref{z1} implies that for any $\beta\in
(0,1],$
\be\la{m1} \ba
M_2    =&  -\frac{2\mu+\lambda}{2} \left(\int \sigma  (\div u)^2 dx \right)_t
       + \frac{2\mu+\lambda}{2} \si' \int  (\div u)^2 dx \\
&+\left(\int \si R(\rho\theta-1) \div u dx\right)_t    - R\si' \int (\rho\theta-1) \div u dx - R\si \int \rho \dot\te \div u dx\\
&-(2\mu+\lambda)  \si \int  \na u : (\na u)^{\rm tr} \div u dx + \frac{2\mu  +\lambda}{2} \si \int  (\div u)^3 dx\\
&+ R\si \int (\rho\theta-1) \na u : (\na u)^{\rm tr}dx +  R\si \int (\div u)^2 dx\\
\le &  - \frac{2\mu+\lambda}{2} \left(\int \sigma  (\div u)^2 dx \right)_t +\left(R\int \si (\rho\theta-1) \div u dx\right)_t \\
& + C \si' \|\rho\theta-1\|_{L^2}^2+ \beta \si^2 \|\rho^{1/2} \dot\te\|_{L^2}^2 \\
&+ C \si^2 \|\na u\|_{L^4}^4 +  C(\on) \beta^{-1}  \|\na u\|_{L^2}^2 + C(\on) \si \int
\te |\na u|^2dx\\
\le &  - \frac{2\mu+\lambda}{2} \left(\int \sigma  (\div u)^2 dx \right)_t  +\left(R\int \si (\rho\theta-1) \div u dx\right)_t \\
&+ C(\on)  C_0^{1/4}\si'+ \beta \si^2 \|\rho^{1/2} \dot\te\|_{L^2}^2  + C \si^2 \|\na u\|_{L^4}^4 \\
&+  C(\de,\on,M) \beta^{-1} ( \|\na u\|_{L^2}^2 + \|\na \te \|_{L^2}^2)+ C(\on) \de \si \|\rho^{1/2} \dot u \|_{L^2}^2,\\
\ea \ee where in the last inequality we have used   (\ref{p}) and
the following simple fact:
\be \la{2.48}\ba   \int
\te |\na u|^2dx   & \le
\int|\te-1||\na u|^2dx+  \int |\na u|^2dx\\ &\le C
\|\te-1 \|_{L^6}\|\na u\|_{L^2}^{3/2}
\|\na u\|_{L^6}^{1/2}+  \|\na u\|_{L^2}^2
\\ &\le C \|\na\te\|_{L^2}\|\na u\|_{L^2}^{3/2}
\left(   \|\n \dot u\|_{L^2}+\|\na u\|_{L^2}+ \|\na\te\|_{L^2}+1\right)^{1/2}\\
&\quad+ C\|\na  u\|_{L^2}^2\\ &\le \de
\left(  \|\na\te\|^2_{L^2}  + \|\n^{1/2}  \dot
u\|_{L^2}^2 \right) + C(\de,\hat\n, M)  \|\na
u\|_{L^2}^2 \ea\ee
due to (\ref{3.30}), (\ref{z1}) and (\ref{infty12}).

For the term $M_3$, it holds that
 \be\la{m2} \ba
M_3
& = -\frac{\mu }{2}\int\sigma |\curl u|^2_t dx
    -\mu \sigma  \int \curl u \cdot \curl (u\cdot \na u)dx \\
& = -\frac{\mu }{2}\left(\sigma \|\curl u\|_{L^2}^2\right)_t +
\frac{\mu }{2}\si'  \|\curl u\|_{L^2}^2
-\mu \sigma  \int \curl u \cdot (\na u^i \times \na_i  u)  dx \\&\quad+ \frac{\mu}{2}\si
\int |\curl u|^2\div udx \\
& \le  -\frac{\mu }{2}\left(\sigma \|\curl u\|_{L^2}^2\right)_t + C
\|\na u\|_{L^2}^2 +   C\sigma^2 \|\na u\|_{L^4}^4  . \ea \ee

Now, substituting \eqref{bb2}, (\ref{m1}),  and  (\ref{m2}) into (\ref{m0}), we obtain  (\ref{an1}) after choosing $\de$ suitably small.

\noindent{}\textbf{Part II:  The proof of  (\ref{ae0}).}

For $m\ge 0,$ operating $ \si^m\dot u^j[\pa/\pa t+\div (u\cdot)]$ to $ (\ref{a111})^j$ and integrating the resulting equality over $\Omega$ by parts lead to
\be\la{m4} \ba
& \left(\frac{\sigma^m}{2}\int\rho|\dot{u}|^2dx \right)_t -\frac{m}{2}\sigma^{m-1}\si'\int\rho|\dot{u}|^2dx\\
 &=  \int_{\p \O}  \sigma^m \dot{u} \cdot n G_t dS - \int  \sigma^m [\div \dot{u} G_t + u \cdot \na \dot u \cdot \na G]dx \\
&\quad- \mu \int\sigma^m\dot{u}^j\xl[(\na \times \curl u)_t^j + \div (u (\na \times \curl u)^j)\xr] dx  \triangleq\sum_{i=1}^{3}N_i.
\ea \ee

One can deduce from \eqref{h1}, \eqref{pzw1},  and \eqref{pzw1} that
\be \la{bz8}\ba
N_1&=- \int_{\p \O}  \sigma^m (u \cdot \na n \cdot u) G_t dS\\
&=- \left(\int_{\p \O} \sigma^m (u \cdot \na n \cdot u) G dS\right)_t + m \si^{m-1} \si' \int_{\p \O}( u \cdot \na n \cdot u) G dS \\
&\quad+  \int_{\p \O}  \sigma^m (\dot u \cdot \na n \cdot u) G dS+ \int_{\p \O}  \sigma^m ( u \cdot \na      n \cdot \dot u) G dS\\
&\quad-  \int_{\p \O}  \sigma^m G( u \cdot \na) u \cdot \na n \cdot u  dS -\int_{\p \O}  \sigma^m G u \cdot \na n \cdot (u\cdot \na )u  dS\\
&\le - \left(\int_{\p \O}  \sigma^m (u \cdot \na n \cdot u) G dS\right)_t + C\si^{m-1} \si' \| \na u\|_{L^2}^2\|\na G\|_{L^2} \\
&\quad+\de \si^m \|\na\dot u\|_{L^2}^2+ C(\de) \si^m \|\na u\|_{L^2}^2 \|\na G\|_{L^2}^2\\
& \quad- \int_{\p \O}  \sigma^m G( u \cdot \na) u \cdot \na n \cdot u  dS -\int_{\p \O}  \sigma^m G u \cdot \na n \cdot (u\cdot \na )u  dS,
\ea \ee
where one has used
\be \ba\notag
\int_{\partial \O}  (\dot u\cdot \na n\cdot u+ u\cdot \na n\cdot\dot u) G dS &\le C \|\dot u\|_{L^4(\partial\Omega)} \| u\|_{L^2(\partial\Omega)} \|G\|_{L^4(\partial\Omega)} \\&\le C \|\na\dot u\|_{L^2} \|\na u\|_{L^2} \|\na G\|_{L^2},
\ea \ee
and
\be\la{b2}
\int_{\partial \O}  ( u\cdot \na n\cdot u) G dS   \le C \|\na u\|_{L^2} ^2\|\na G\|_{L^2}
\ee
due to \eqref{eee0}.

Now, we will adopt the idea in \cite{C-L} to deal with the last two boundary terms in \eqref{bz8}. In fact, since $u^\bot\times n$ has compact support, (\ref{eee1}) along with \eqref{g1} and integration by parts yields
\be \la{bz3}\ba &- \int_{\partial\Omega} G (u\cdot \na) u\cdot\na n\cdot u dS \\&= -\int_{\partial\Omega}  G u^\bot\times n \cdot\na u^i \nabla_i n\cdot u  dS \\&= - \int_{\partial\Omega} G n\cdot ( \na u^i \times  u^\bot)    \nabla_i n\cdot u dS\\
&= - \int\div( G( \na u^i \times  u^\bot)   \nabla_i n\cdot u) dx \\
&= - \int \na (\nabla_i n\cdot u G) \cdot ( \na u^i \times  u^\bot)   dx  - \int \div( \na u^i \times  u^\bot)    \nabla_i n\cdot u   G  dx \\
&= - \int \na (\nabla_i n\cdot u G) \cdot ( \na u^i \times  u^\bot)   dx  + \int  G \na  u^i \cdot \na\times  u^\bot     \nabla_i n\cdot u     dx \\
& \le C \int |\na G||\na u||u|^2dx+C \int |G| (|\na u|^2|u|+|\na u||u|^2)dx
\\& \le C  \|\na G\|_{L^6}\|\na u\|_{L^2}\|u\|^2_{L^6}
+C  \| G\|_{L^6}\|\na u\|^2_{L^3}\|u\|_{L^6}+C\| G\|_{L^{6}} \|\na      u\|_{L^2}\|u\|_{L^6}^2
\\& \le \de \|\na G\|_{L^6}^2+C(\de) \|\na u\|^6_{L^2}+C\|\na u\|^4_{L^3}+ C  \|  \na G\|_{L^2}^2 (\|\na u\|^2_{L^2} +1 ).\ea\ee
Similarly, it holds that
\be \la{bz4}\ba &-  \int_{\partial\Omega}G  u\cdot\na n\cdot ({u}\cdot\na) u dS\\& \le \de \|\na G\|_{L^6}^2+C(\de) \|\na u\|^6_{L^2}+C\|\na u\|^4_{L^3}+ C \|  \na G\|_{L^2}^2(\|\na u\|^2_{L^2} +1 ).\ea\ee

Next, it follows from \eqref{hj1} and \eqref{jia1} that
\be \ba \la{pt11}
G_t =& (2\mu+\lambda)\div u_t -P_t\\
=& (2\mu+\lambda) \div \dot u - (2\mu+\lambda) \div (u\cdot \na u) - R\rho \dot\te + \div(Pu)\\
=& (2\mu+\lambda) \div \dot u - (2\mu+\lambda) \na u : (\na u)^{\rm tr} -  u \cdot  \na G + P \div u - R\rho \dot\te.
\ea \ee
Then, integration by parts combined with \eqref{pt11} gives
\be\la{m5} \ba
N_2 =&  - \int  \sigma^m [\div \dot{u} G_t + u \cdot \na \dot u \cdot \na G]dx \\
=& - (2\mu+\lambda) \int  \sigma^m (\div \dot{u} )^2 dx +  (2\mu+\lambda)  \int \sigma^m \div \dot{u} \na u : (\na u)^{\rm tr} dx \\
&+\int  \sigma^m \div \dot{u}  u \cdot  \na G  dx - \int \sigma^m \div \dot{u} P \div u dx  \\
&+ R \int \sigma^m \div \dot{u} \rho \dot\te dx - \int  \sigma^m u \cdot \na \dot u \cdot \na Gdx \\
\le &   - (2\mu+\lambda) \int  \sigma^m (\div \dot{u} )^2 dx\\
& + C \si^m \|\na \dot{u}\|_{L^2} \|\na u\|_{L^4}^2 + C \si^m \|\na \dot{u}\|_{L^2} \|\na G\|_{L^2}^{1/2} \|\na G\|_{L^6}^{1/2} \|u\|_{L^6} \\
&+ C(\on) \si^m \|\na \dot{u}\|_{L^2} \|\te \na u\|_{L^2}+ C(\on) \si^m \|\na \dot{u}\|_{L^2} \|\rho^{1/2} \dot \te\|_{L^2}.\ea \ee

Note that
$$\curl u_t=\curl \dot u-u\cdot \na \curl u-\na u^i\times \nabla_iu,$$
this together with some straight calculations yields
\be\la{ax3999}\ba
N_3 &=-\mu\int\sigma^{m}|\curl \dot{u}|^{2}dx-\mu\int\sigma^{m}\curl \dot{u}\cdot\curl(u\cdot\nabla u)dx \\
&\quad+\mu\int\sigma^{m}(\curl u\times\dot{u})\cdot\nabla\div udx -\mu\int\sigma^{m}\div u\,\curl u\cdot\curl \dot{u}dx\\
&\quad-\mu\int\sigma^{m} u^{i}\div(\nabla_i\curl u\times\dot{u})dx+\mu\int\sigma^{m} u^i\nabla_i\curl u\cdot\curl\dot{u}dx  \\
&=-\mu\int\sigma^{m}|\curl\dot{u}|^{2}dx-\mu\int\sigma^{m}\curl\dot{u}\cdot(\nabla u^i\times\nabla_i u) dx \\
&\quad+\mu\int\sigma^{m}(\curl u\times\dot{u})\cdot\nabla\div udx-\mu\int\sigma^{m}\div u\,\curl u\cdot\curl \dot{u}dx \\
&\quad-\mu\int\sigma^{m}  u\cdot\nabla\div(\curl u\times\dot{u})dx+\mu\int\sigma^{m} u^i\div(\curl u\times\nabla_i\dot{u})dx \\
&=-\mu\int\sigma^{m}|\curl \dot{u}|^{2}dx+\mu\int\sigma^{m}(\curl u\times\nabla u^i)\cdot\nabla_i\dot{u}dx \\
&\quad-\mu\int\sigma^{m}\curl\dot{u}\cdot(\nabla u^i\times\nabla_i u)dx-\mu\int\sigma^{m}\div u\,\curl u\cdot\curl \dot{u}dx\\
&\leq-\mu\int\sigma^{m}|\curl \dot{u}|^{2}dx+C\sigma^{m}\|\nabla\dot{u}\|_{L^2}\|\nabla u\|_{L^4}^2\\
&\leq-\mu\int\sigma^{m}|\curl \dot{u}|^{2}dx+\delta\sigma^{m}\|\nabla\dot{u}\|_{L^2}^2+ C(\delta)\sigma^{m}\|\nabla u\|_{L^4}^4.
\ea\ee

Finally, it is easy to deduce from  Lemma   \ref{le4}  that
\be \la{bz6}  \|\na G\|_{L^6}
\le C\|\n \dot u\|_{L^6}\le C(\on)  \|\na \dot u\|_{L^2}. \ee

Hence, submitting  \eqref{bz8},  \eqref{m5}, and \eqref{ax3999} into \eqref{m4}, one obtains after  using \eqref{bz3}, \eqref{bz4}, \eqref{z1},  \eqref{h19},  \eqref{bz6}, and \eqref{infty12}
that
\be\la{ax40}\ba
&\left(\frac{\sigma^{m}}{2}\|\rho^{1/2}\dot{u}\|_{L^2}^2\right)_t+(2\mu+\lambda)\sigma^{m}\|\div\dot{u}\|_{L^2}^2+\mu\sigma^{m}\|\curl\dot{u}\|_{L^2}^2\\
&\le -\left(\int_{\p \O}  \sigma^m (u \cdot \na n \cdot u) G dS\right)_t+ C(\on) \de \si^m \|\na \dot{u}\|_{L^2}^2 \\
&\quad+ C(\de,\on) \si^m \|\rho^{1/2} \dot \te\|_{L^2}^2 + C(\de, \on, M)(\si^{m-1}\si'+\si^m)  \|\rho^{1/2} \dot u\|_{L^2}^2\\
&\quad+C(\de, \on, M) \|\na u\|^2_{L^2} +C(\de)\si^m \|\na u\|^4_{L^4}+ C(\de,\on) \si^m  \|\te \na u\|_{L^2}^2.
\ea\ee
Applying  \eqref{tb11} to \eqref{ax40} and choosing $\de$ small enough yield \eqref{ae0} directly.

\noindent{}\textbf{Part III: The proof of (\ref{nle7}).}

For $m\ge 0,$
multiplying $(\ref{a1})_3 $ by $\sigma^m \dot\te$ and integrating
the resulting equality over $\Omega $ yield  that
 \be\la{e1} \ba &\frac{\ka
{\sigma^m}}{2}\left( \|\na\te\|_{L^2}^2\right)_t+\frac{R\sigma^m}{\ga-1} \int\rho|\dot{\te}|^2dx
\\&=-\ka\sigma^m\int\na\te\cdot\na(u\cdot\na\te)dx
+\lambda\sigma^m\int  (\div u)^2\dot\te dx\\&\quad
+2\mu\sigma^m\int |\mathfrak{D}(u)|^2\dot\te dx-R\si^m\int\n\te \div
u\dot\te dx \triangleq \sum_{i=1}^4I_i . \ea\ee

 First,  the combination of (\ref{g1}) and (\ref{z1}) gives that
\be\la{e2} \ba
I_1
&\le C\sigma^{m}   \|\na u\|_{L^2}\|\na\te\|^{1/2}_{L^2}
\|\na^2\te\|^{3/2}_{L^2}  \\
&\le  \de\sigma^{m} \|\n^{1/2}\dot\te\|^2_{L^2}+\si^m \left(\|\na u\|_{L^4}^4+\|\te\na u\|_{L^2}^2\right) +C(\de,\on,M)\sigma^{m}
\|\na\te\|^2_{L^2}   ,\ea\ee
where in the last inequality we have used the following estimate:
\be  \la{lop4}\ba
	\|\na^2\te\|_{L^2} &\le C (\on)\left(\|\n\dot \te\|_{L^2}+ \|\na u\|_{L^4}^2+\|\te\na u\|_{L^2}\right),
	\ea\ee
which is derived from Lemma \ref{zhle} to the following elliptic problem:
\be\la{3.29}
\begin{cases}
		\ka\Delta \te=\frac{R}{\ga-1}\n\dot\te +R\n\te\div
		u-\lambda (\div u)^2-2\mu |\mathfrak{D}(u)|^2,& \text{in}\,\O\times [0,T],\\
		\na \theta  \cdot n =0,& \text{on}\,\p\O \times [0,T],\\
		\nabla\te\rightarrow0,\,\,&\text{as}\,\,|x|\rightarrow\infty.
	\end{cases}\ee

Next, it holds  that  for any $\eta\in (0,1],$
\be\la{e3}\ba I_2 =&\lambda\si^m\int (\div u)^2 \te_t
dx+\lambda\si^m\int (\div u)^2u\cdot\na\te
dx\\=&\lambda\si^m\left(\int (\div u)^2 \te
dx\right)_t-2\lambda\si^m \int \te \div u \div (\dot u-u\cdot\na u)
dx\\&+\lambda\si^m\int (\div u)^2u\cdot\na\te
dx \\=&\lambda\si^m
\left(\int (\div u)^2 \te dx\right)_t-2\lambda\si^m\int \te \div u
\div \dot udx\\&+2\lambda\si^m\int \te \div u \pa_i u^j\pa_j  u^i dx
+ \lambda\si^m\int u \cdot\na\left(\te   (\div u)^2 \right)dx
 \\
\le &\lambda\left(\si^m\int (\div u)^2 \te dx\right)_t-\lambda
m\si^{m-1}\si'\int (\div u)^2 \te dx\\& +\eta\si^m\|\na \dot
u\|_{L^2}^2+C\eta^{-1}\si^m\|\te\na u\|_{L^2}^2+C\si^m\|\na u\|_{L^4}^4,\ea\ee
and
 \be \la{e5}\ba I_3&\le 2\mu\left(\si^m\int
|\mathfrak{D}(u)|^2 \te dx\right)_t-2\mu m\si^{m-1}\si'\int
|\mathfrak{D}(u)|^2 \te dx
 \\&\quad+ \eta\si^m\|\na \dot
u\|_{L^2}^2+C\eta^{-1}\si^m\|\te\na u\|_{L^2}^2+C\si^m\|\na u\|_{L^4}^4 .    \ea\ee

Finally, Cauchy's inequality gives
 \be\la{e39}\ba
 |I_4|    \le  \delta \si^m \int \n |\dot\te|^2dx+C( \delta,\on)\si^m \|\te\na u\|_{L^2}^2.  \ea\ee

Substituting  (\ref{e2}) and \eqref{e3}--(\ref{e39}) into (\ref{e1}), we obtain \eqref{nle7}
after using  \eqref{h3} and choosing $\de$ suitably
small.

The proof of Lemma \ref{a113.4} is completed.
\end{proof}

Next, with the help of the estimates \eqref{an1}--\eqref{nle7}, we now derive a priori estimate on $A_3(T)$.

\begin{lemma}\la{le6} Under the conditions of Proposition \ref{pr1},   there exists a positive constant $\ve_2$
depending only on   $\mu,\,\lambda,\, \ka,\, R,\, \ga,\, \on,\,\bt, $ $\O$, and $M$
such that if $(\rho,u,\te)$ is a smooth solution to the problem (\ref{a1})--(\ref{ch2})  on $\Omega\times (0,T] $ satisfying      (\ref{z1}) with $K$ as
in Lemma \ref{le2}, the following estimate holds: \be\la{b2.34}  A_3(T) \le C_0^{1/6},\ee provided $C_0\le
\ve_2$ with $\ve_2$ defined in (\ref{xjia4}).
\end{lemma}

\begin{proof}
First, by virtue of (\ref{ljq01}), (\ref{h20}), (\ref{h21}), (\ref{z1}), \eqref{p}, and (\ref{infty12}), one gets
\be\la{ae9}\ba
&\|\na u\|_{L^4}^4\\
&\le C \|G\|_{L^4}^4+C \|\curl u\|_{L^4}^4 +C \|\n\te-1\|_{L^4}^4
+C\|\na u\|_{L^2}^4\\
&\le C(\on) \left(\|\na u\|_{L^2}+1\right)\|\n\dot u\|_{L^2}^3 +C(\on)\|\na\te\|_{L^2}^3+C\|\n-1\|_{L^4}^4 +C\|\na u\|_{L^2}^4\\	&\le C(\on,M)\left( \|\n \dot u\|_{L^2}^3 + \|\na \te\|_{L^2}^3\right)  +C\|\n-1\|_{L^4}^4+C(\on,M)\|\na u\|_{L^2}^2,
\ea\ee
which together with	(\ref{z1})  yields
\be \la{m22}\ba
\si\|\na u\|_{L^4}^4   &\le C(\on,M) \left(C_0^{1/12}\|\n \dot u\|_{L^2}^2  + \|\na \te\|_{L^2}^2+\|\na u\|_{L^2}^2\right)
+C\si\|\n-1\|_{L^4}^4.
\ea\ee
Combining this with \eqref{an1} yields
\be\ba  \la{ant}
&(\sigma B_1)'(t) + \int \sigma \rho |\dot u|^2dx\\
&\le   C(\on,M) C_0^{1/4} \sigma' +2\beta\si^2\|\n^{1/2}\dot\theta\|_{L^2}^2+C(\on,M)\si^2\|\n-1\|_{L^4}^4\\ &+C(\on,M)\beta^{-1}\left(\|\na u\|_{L^2}^2+\|\na\te\|_{L^2}^2\right),
\ea\ee
provided that $C_0\le \epsilon_{2,1}\triangleq \min\{1,(2C(\on,M))^{-12}\}$.
	
Next, we estimate the third term on the righthand side of \eqref{ant},  it follows from $(\ref{a1})_1$ and (\ref{hj1}) that $\n-1$ satisfies
\be \ba\label{eee3} &(\n-1)_t+\frac{R}{2\mu+\lambda}(\n-1)\\
=&-u\cdot\na(\n-1)-(\n-1)\div u-\frac{G}{2\mu+\lambda}-\frac{R\n(\te-1)}{2\mu+\lambda}.
\ea\ee
Multiplying \eqref{eee3} by $4(\n-1)^3$ and integrating the resulting equality over $\Omega$, we obtain after integrating by parts that
\be\ba\label{eee5} &\left(\|\n-1\|_{L^4}^4\right)_t+\frac{4R}{2\mu+\lambda}\|\n-1\|_{L^4}^4\\
=&-3\int(\n-1)^4\div u dx-\frac{4}{2\mu+\lambda}\int(\n-1)^3Gdx\\
&-\frac{4R}{2\mu+\lambda}\int(\n-1)^3\n(\te-1)dx\\
\leq &\frac{2R}{2\mu+\lambda}\|\n-1\|_{L^4}^4+C(\on)\|\na u\|_{L^2}^2+C\|\n-1\|_{L^4}^3\|G\|_{L^2}^{1/4}\|\na G\|_{L^2}^{3/4}\\
&+C(\on)\|\n-1\|_{L^4}^3\|\n(\te-1)\|_{L^2}^{1/4}\|\na\te\|_{L^2}^{3/4}\\
\leq &\frac{3R}{2\mu+\lambda}\|\n-1\|_{L^4}^4+C(\on)\|\na u\|_{L^2}^2+C(\on,M)(\|\n^{1/2}\dot u\|_{L^2}^3+\|\na\te\|_{L^2}^3),\ea\ee
where one has used \eqref{p}, \eqref{z1}, \eqref{infty12}, and \eqref{h19}.
It follows from \eqref{eee5} that
\be\ba\label{eee6} &\left(\|\n-1\|_{L^4}^4\right)_t+\frac{R}{2\mu+\lambda}\|\n-1\|_{L^4}^4\\
\leq &C(\on)\|\na u\|_{L^2}^2+C(\on,M)(\|\n^{1/2}\dot u\|_{L^2}^3+\|\na\te\|_{L^2}^3).\ea\ee
Multiplying \eqref{eee6} by $\si^n$ with $n\ge 1$, integrating the resulting inequality over $(0,T)$, one gets from \eqref{z1} and \eqref{a2.112} that
\be\ba\la{67}
&\int_0^T\si^n\|\n-1\|_{L^4}^4 dt\\
&\le C(\on,M)A_3^{1/2}(T)\int_0^T\si^{n-1}(\|\n^{1/2}\dot u\|_{L^2}^2+\|\na \te\|_{L^2}^2) dt\\
&\quad+C(\on)C_0^{1/4}+C\int_0^{\si(T)}\|\n-1\|_{L^4}^4dt\\
&\le C(\on,M)C_0^{1/4}+C(\on,M))C_0^{1/12}\int_0^T\si^{n-1}\|\n^{1/2}\dot u\|_{L^2}^2dt.
\ea\ee
Taking $n=2$, 	this together with \eqref{z1} directly gives
\be \ba\label{eee2} \int_0^T\si^2\|\n-1\|_{L^4}^4 dt\leq C(\on,M)C_0^{1/4}.\ea\ee

To estimate the second therm on the righthand side of \eqref{ant}, for $C_2$ as in  (\ref{ae0}), adding  (\ref{nle7}) multiplied by $C_2+1$ to (\ref{ae0}) and choosing $ \eta$ suitably small give
	\be\la{e8}\ba
	&\left(\sigma^m\varphi\right)'(t) + \sigma^m \int\left(\frac{C_1}{2} |\nabla\dot{u}|^2 +\n|\dot \te|^2\right)dx\\
	&\le - 2\left(\int_{\p \O}  \sigma^m( u \cdot \na n \cdot u )G dS\right)_t + C( \on, M) (\si^{m-1}\si'+\si^m ) \|\rho^{1/2} \dot u\|_{L^2}^2\\
	&\quad+ C(\on,  M) (\|\na u\|_{L^2}^2 + \|\na \te\|_{L^2}^2)+C(\on) \si^m \|\na u\|^4_{L^4} + C (\on)\si^m \|\te\na u\|_{L^2}^2,\\
	\ea\ee
	where  $\varphi(t)$ is defined by
	\be\la{wq3} \varphi(t) \triangleq   \|\n^{1/2}\dot u\|_{L^2}^2(t) +(C_2+1)  B_2(t).\ee
	Then it can be deduced from (\ref{e6}) and \eqref{2.48} that
	\be\la{wq2}  \varphi(t) \ge  \frac{1}{2}\|\n^{1/2}\dot u\|_{L^2}^2 +\frac{\ka(\ga-1)}{2R}\|\na\te\|_{L^2}^2-C_3(\on,  M)\|\na u\|_{L^2}^2. \ee
Next, it follows from (\ref{3.30}) that
\be \la{m20}\ba
\|\te\na u\|_{L^2}^2
&\le \|\te-1 \|_{L^6}^2 \|\na u\|_{L^2} \|\na u\|_{L^6}+\|\na u\|_{L^2}^2 \\
&\le C  \left( \|\na u\|_{L^2}^2+\|\na\te\|_{L^2}^2\right)\left( \|\n\dot u\|^2_{L^2} +\|\na\te\|^2_{L^2} +1\right).
\ea\ee
Thus, taking $m=2$ in \eqref{e8}, one obtains after  using  \eqref{z1}, \eqref{m20}, and \eqref{m22}  that
\be\la{e25}\ba	
& \left(\sigma^2\varphi\right)'(t) + \sigma^2 \int\left(\frac{C_1}{2} |\nabla\dot{u}|^2 +\n|\dot \te|^2\right)dx\\	
&\le  - 2\left(\int_{\p \O}  \sigma^2 (u \cdot \na n \cdot u) G dS\right)_t+ C_4(\on,  M)\si \int\n|\dot u|^2 dx  \\
&\quad +  C(\on,  M)\left(\|\na u\|_{L^2}^2+\|\na\te\|_{L^2}^2\right)
+C\si^2\|\n-1\|_{L^4}^4.
\ea\ee
		
Finally, due to \eqref{an2}, (\ref{ljq01}), and (\ref{p}), we have
\be\ba \label{antt} B_1(t)
&\ge  C   \|\na u\|_{L^2}^2- C \|\n\te-1\|^2_{L^2}  \ge C_5  \|\nabla u\|_{L^2}^2-C(\on) C_0^{1/4},\ea\ee
For $C_3$ in \eqref{wq2} and $C_4$ in \eqref{e25} , define
\be\notag B_3(t)\triangleq \si^2\varphi+C_5^{-1}(C_3+C_5(C_4+1)+1)\si B_1 \ee
satisfying
\be\la{e444}\ba	
& B_3(t)\geq \frac{1}{2}\si^2\int\n|\dot u|^2dx+\frac{\ka(\ga-1)}{2R}\si^2\|\na\te\|_{L^2}^2+\si\|\na u\|_{L^2}^2-C(\on,M)C_0^{1/4},
\ea\ee
which comes from \eqref{wq2} and \eqref{antt}. Adding \eqref{ant} multiplied by $C_5^{-1}(C_3+C_5(C_4+1)+1)$ to \eqref{e25} and choosing $\beta$ small enough, we obtain
\be\la{e26}\ba	
& B_3'(t) + \frac{1}{2} \int\left(\si\n|\dot u|^2+C_1\si^2|\nabla\dot{u}|^2 +\si^2\n|\dot \te|^2\right)dx\\	
&\le  - 2\left(\int_{\p \O}  \sigma^2 (u \cdot \na n \cdot u) G dS\right)_t +C(\on,  M) C_0^{1/4} \sigma'\\
&\quad +  C(\on,  M)\left(\|\na u\|_{L^2}^2+\|\na\te\|_{L^2}^2\right)
+C\si^2\|\n-1\|_{L^4}^4.
\ea\ee
It's easy to deduce from \eqref{h19}, \eqref{eee0}, \eqref{z1}, and \eqref{b2} that
\be \ba\label{jia4}
\sup_{ 0\le t\le T} \int_{\p \O}  \sigma^2 u \cdot \na n \cdot u G dS
\le & C(\on) \sup_{ 0\le t\le T} (\si \|\na u\|_{L^2}^2) \sup_{ 0\le t\le T} (\si \|\rho^{1/2} \dot u\|_{L^2})\\
\le &  C(\on) C_0^{1/4}.
\ea \ee
Combining this with \eqref{eee2}, \eqref{e444}--(\ref{jia4}), and (\ref{z1}) yields
\be\ba\notag
A_3(T)\le C(\on,M)C_0^{1/4}\le C_0^{1/6},
\ea\ee
which implies (\ref{le2}) provided
\be \ba\label{xjia4} C_0\le\ve_2\triangleq \min\{\epsilon_{2,1},C(\on,M)^{-12}\}.\ea\ee

The proof of Lemma \ref{le6} is completed.
\end{proof}

Now, in order to control $A_2(T)$, we first re-establish the basic energy estimate  for short time $[0, \si(T)]$, and then show that the spatial $L^2$-norm of $\te-1$ could be  bounded by the combination of the initial energy and the spatial $L^2$-norm of $\na \te$, which is indeed the key ingredient to obtain the estimate of $A_2(T)$.

\begin{lemma}\la{a13} Under the conditions of Proposition \ref{pr1},
	there exist  positive constants $C$ and $\varepsilon_{3,1}$
	depending only on
	$\mu,\,\lambda,\, \ka,\, R,\, \ga,\, \on,\,\bt, $ $\O$, and $M$ such
	that  if $(\rho,u,\te)$ is a smooth solution to the problem (\ref{a1})--(\ref{ch2})  on $\Omega\times (0,T] $ satisfying      (\ref{z1}) with $K$ as
	in Lemma \ref{le2},
	the following estimates  hold:
	\be \la{a2.121} \ba
	&\sup_{0\le t\le \si(T)}\int\left( \n |u|^2+(\n-1)^2 + \n(\te-\log \te-1) \right)dx\le C C_0,\ea\ee
	and
	\be  \la{a2.17}  \ba
	\|\te(\cdot,t)-1\|_{L^2} \le C \left(C_0^{1/2} +C_0^{1/3}\|\na\te(\cdot,t)\|_{L^2}\right),
	\ea\ee
	for al $t\in(0,\si(T)].$
\end{lemma}
\begin{proof}
	The proof is divided into the following two steps.
	
	\noindent\textbf{Step I: The proof of \eqref{a2.121}.}
	
	First, multiplying (\ref{a11}) by $u$, one deduces from integration by parts and \eqref{a1}$_1$ that
	\be \la{a2.12} \ba
	&\frac{d}{dt}\int\left(\frac{1}{2}\n |u|^2+R(1+\n\log \n-\n)
	\right)dx+ \int(\mu|\curl u|^2+(2\mu+\lambda)(\div u)^2)dx \\
	&=  R \int \rho (\te -1) \div u dx\\
	&\le \de \|\na u\|_{L^2}^2 + C(\de, \on) \int \n(\te-1)^2dx\\
	&\le\de \|\na u\|_{L^2}^2 + C(\de, \on)(\|\te(\cdot,t)\|_{L^{\infty}} +1)  \int \rho (\te -\log \te -1)dx.
	\ea\ee
	Using \eqref{ljq01} and choosing $\de$ small enough in \eqref{a2.12}, it holds that
	\be \la{a2.222} \ba
	&\frac{d}{dt}\int\left(\frac{1}{2}\n |u|^2+R(1+\n\log \n-\n)
	\right)dx + C_3 \int|\na u|^2dx \\
	&\le C(\on) (\|\te(\cdot,t)\|_{L^{\infty}} +1)  \int \rho  (\te -\log \te -1)dx.
	\ea\ee
	Then, adding \eqref{a2.222} multiplied by $(2\mu+1) {C_3}^{-1}$  to \eqref{la2.7}, one has
	\be \la{a2.22} \ba
	&\xl((2\mu+1) {C_3}^{-1}+1\xr)\frac{d}{dt}\int\left(\frac{1}{2}\n |u|^2+R(1+\n\log \n-\n)\right)dx
	\\&+ \frac{R}{\ga-1} \frac{d}{dt} \int \n(\te-\log \te-1)dx+\int|\na u|^2dx\\
	&\le C(\on) (\|\te(\cdot,t)\|_{L^{\infty}} +1)  \int \rho (\te -\log \te -1)dx.
	\ea\ee
	
	Next, we claim that
	\be \ba\la{k}
	\int_0^{\si(T)}\|\te\|_{L^\infty}dt \le C(\on, M).
	\ea \ee
	Combining this with \eqref{a2.22}, \eqref{a2.9}, and Gr\"onwall inequality  implies \eqref{a2.121} directly.
	
	Finally, it remains to prove \eqref{k}. The combination of \eqref{67} with $n=1$  and \eqref{z1} yields	\be\ba\la{eee23}	\int_0^T\si\|\n-1\|_{L^4}^4 dt\leq C(\on,M).	\ea\ee
	Taking $m=1$ in \eqref{e8} and integrating the resulting inequality over $(0,T)$, one deduces from \eqref{wq2}, \eqref{m20}, \eqref{m22}, (\ref{z1}), \eqref{infty12}, \eqref{eee23}, and \eqref{b2} that
	\bnn  \ba
	& \sigma \varphi +  \int_0^t \sigma\int\left(\frac{C_1}{2} |\nabla\dot{u}|^2 +\n|\dot \te|^2\right)dxd\tau\\
	&\le  2\left|\int_{\p \O}  \sigma (u \cdot \na n \cdot u )G dS\right|(t)
	+ C(\on,M)\int_0^t (\|\n^{1/2}  \dot u\|_{L^2}^2 +\|\na u\|_{L^2}^2+\|\na \te\|_{L^2}^2)d\tau \\
	&\quad+  C(\on)\int_0^t \si \|\n-1\|_{L^4}^4d\tau+C(\on)\int_0^t \left(\|\na u\|_{L^2}^2+\|\na \te\|_{L^2}^2\right) \si\varphi d\tau\\
	& \le  C(\on) (\sigma \|\na u\|_{L^2}^2 \|\rho^{1/2} \dot u\|_{L^2})(t)+ C(\on,M)+C(\on)\int_0^t \left(\|\na u\|_{L^2}^2+\|\na \te\|_{L^2}^2\right) \si\varphi d\tau\\
	& \le \frac{1}{2} \sigma \varphi + C(\on,M)+C(\on)\int_0^t \left(\|\na u\|_{L^2}^2+\|\na \te\|_{L^2}^2\right) \si\varphi d\tau.
	\ea\enn
	Then Gr\"onwall inequality together with (\ref{z1}) and \eqref{wq2} yields
	\be\ba\label{ae26}
	\sup_{0\le t\le T}\si \left(\int\rho|\dot{u}|^2dx+\|\na\te\|_{L^2}^2\right)+\int_0^T\si \int\left(|\na\dot u|^2+\n|\dot\te|^2\right)dxdt
	\le  C(\on,M).\ea\ee
	
	Next, it follows from (\ref{z1}), (\ref{lop4}), \eqref{m20},  \eqref{m22}, \eqref{eee23},  and \eqref{ae26} that
	\be\ba  \la{k1}
	\int_0^{T }\si \|\na^2\te\|_{L^2}^2dt
	&\le   C(\on,M)\int_0^{T } \left(\si \|\n \dot\te\|_{L^2}^2+
	\|\n\dot u \|_{L^2}^2  \right) dt\\
	&\quad+ C(\on,M)\int_0^{T } \left( \| \na u\|_{L^2}^2+\| \na\te\|_{L^2}^2+\si\|\n-1\|_{L^4}^4\right) dt \\
	&\le C(\on,M).
	\ea\ee
	Furthermore, one deduces from \eqref{g2} and \eqref{g1} that
	\begin{equation}
		\label{6yue}\ba
		\|\te-1\|_{L^\infty} \le  C \|\te-1\|_{L^6}^{1/2} \|\na\te\|_{L^6}^{1/2}
		\le& C  \|\na \te\|_{L^2}^{1/2} \|\na^2 \te\|_{L^2}^{1/2},
		\ea
	\end{equation}
	which together with  (\ref{z1}) and \eqref{k1}  gives that
	\be\la{3.88}\ba
	  \int_0^{\si(T)}\|\te-1\|_{L^\infty}dt
	&\le C  \int_0^{\si(T)}\|\na\te\|_{L^2}^{1/2} \left(\si\|\na^2\te\|^2_{L^2}\right)^{1/4}\si^{-1/4}dt\\
	&\le C \left(\int_0^{\si(T)} \|\na \te\|_{L^2}^2dt \int_0^{\si(T)}\si\|\na^2\te\|_{L^2}^2dt\right)^{1/4}	\\
	&\le C(\on,M)C_0^{1/16}.
	\ea\ee
This implies	\eqref{k}   directly.

	\noindent \textbf{Step II: The proof of (\ref{a2.17}).}
	
	Denote
	\begin{equation*}
		(\te(\cdot,t)> 2)\triangleq \left.\left\{x\in
		\Omega\right|\te(x,t)> 2\right\},
	\end{equation*}
	\begin{equation*}
		(\te(\cdot,t)< 3)\triangleq
		\left.\left\{x\in \Omega\right|\te(x,t)<3\right\}.
	\end{equation*}
	Direct calculations lead to
	\be\la{eee9}\ba
	\te-\log\te-1&\ge(\te-1)^2\int_0^1\frac{1-\al}{\al (\te-1)+1}d\al\\
	&\ge \frac{1}{8} (\te-1)1_{(\te(\cdot,t)>2)
	}+\frac{1}{12}(\te-1)^21_{(\te(\cdot,t)<3)}.  \ea\ee
	Combining this with \eqref{a2.121} gives
	\be \la{a2.11}\ba
	\sup_{0\le t\le \si(T)}\int \left(\n(\te-1)1_{(\te(\cdot,t)>2)}+\n(\te-1)^21_{(\te(\cdot,t)<3)}\right)dx \le C(\hat\n,M) C_0.
	\ea\ee
	
	Next, for $t \in (0,\sigma(T)]$, taking $g(x)=\n(x,t)$, $f(x)=\te(x,t)-1$, $s=2$ and $\Sigma=(\te(\cdot,t)< 3)$ in \eqref{eee8}, we obtain after using \eqref{a2.11} and \eqref{a2.112} that
	\begin{equation}\la{eee11}
		\|\te -1\|_{L^2(\te(\cdot,t)<3)}
		\le C(\on)C_0^{1/2}+ C(\on) C_0^{1/3} \|\na \te\|_{L^2(\Omega)}.
	\end{equation}
	Similarly, taking $g(x)=\n(x,t)$, $f(x)=\te(x,t)-1$, $s=1$ and $\Sigma=(\te(\cdot,t)>2)$ in \eqref{eee8}, we conclude using \eqref{a2.11} that
	\begin{equation}\la{eee12}
		\|\te -1\|_{L^1(\te(\cdot,t)>2)}
		\le C(\on)C_0+ C(\on) C_0^{5/6} \|\na \te\|_{L^2(\Omega)},
	\end{equation}
	which together with H\"older's inequality and \eqref{g1} yeilds
	\be\la{eee13}\ba
	&\|\te-1\|_{L^2(\te(\cdot,t)> 2)}\\
	&\le \|\te-1\|^{2/5}_{L^1(\te(\cdot,t)> 2)} \| \te-1\|_{L^6}^{3/5}\\
	&\le C(\on) \left(C_0^{2/5}+C_0^{1/3} \|\te-1\|_{L^2}^{2/5}\right) \|\na \te\|_{L^2}^{3/5}\\
	&\le C(\on)C_0^{1/2}+ C(\on) C_0^{1/3} \|\na \te\|_{L^2(\Omega)}.
	\ea\ee
	This along with  \eqref{eee11}   yields \eqref{a2.17} directly. 	
	The proof of Lemma \ref{a13}  is completed.
\end{proof}

\begin{remark}	It is easy to deduce from \eqref{a2.8} and \eqref{eee9} that for all $t\in(0,\si(T)],$ 
	\be  \la{eee22}  \ba
	 \|\te(\cdot,t)-1\|_{L^2} \le C \left(C_0^{1/8} +C_0^{1/12}\|\na\te(\cdot,t)\|_{L^2}\right).
	\ea\ee
\end{remark}

With the help of \eqref{a2.17}, we are now in a position to establish the  estimate on $A_2(T)$.

\begin{lemma}\la{le3} Under the conditions of Proposition \ref{pr1},  there exists a positive constant $\ve_3$ depending only on $\mu,\,\lambda,\, \ka,\, R,\, \ga,\, \on,\,\bt $, $\O$, and $M$
	such that if $(\rho,u,\te)$ is a smooth solution to the problem (\ref{a1})--(\ref{ch2})  on $\Omega\times (0,T] $ satisfying  (\ref{z1})   with $K$
	as in Lemma \ref{le2}, the following estimate holds:
	\be\la{a2.34} A_2(T) \le C_0^{1/4},\ee
	provided $C_0\le \ve_3$ with  $\ve_3$ defined in (\ref{jia11}).
\end{lemma}

\begin{proof}
	Multiplying (\ref{a11}) by $u$ and integrating by parts give that
	\be \la{a2.225}\ba
	&\frac{d}{dt}\int\left(\frac{1}{2}\n |u|^2+R(1+\n\log \n-\n)\right)dx\\
	&\quad+ \int\left(\mu|\curl u|^2+(2\mu+\lambda)(\div u)^2\right)dx \\
	&= \int R\rho ( \te -1) \div u dx.
	\ea\ee
	Multiplying $(\ref{a1})_3$ by $(\te-1)$, one obtains after  integrating the resulting
	equality over $\Omega $ by parts that
	\be\la{a2.231} \ba
	& \frac{R}{2(\ga-1)} \frac{d}{dt}\int
	\n {(\te-1)^2}dx+ {\ka} \|\na\te\|_{L^2}^2\\
	&\quad=-\int R \n \te{(\te-1)} \div u dx + \int  {(\te-1)} (\lambda (\div u)^2+2\mu |\mathfrak{D}(u)|^2) dx.
	\ea\ee
	Adding \eqref{a2.225} and \eqref{a2.231} together yields that
	\be\la{a2.23} \ba
	&  \frac{d}{dt}\int
	\left(\frac{1}{2}\n |u|^2+R(1+\n\log \n-\n)+\frac{R}{2(\ga-1)} {\n(\te-1)^2}\right)dx\\
	&\quad+ \mu \|\curl u\|_{L^2}^2 +(2\mu+\lambda)\|\div u\|_{L^2}^2+  \ka \|\na\te\|_{L^2}^2\\
	&=-\int R \n {(\te-1)^2} \div u dx + \int  {(\te-1)} (\lambda (\div u)^2+2\mu |\mathfrak{D}(u)|^2) dx,
	\ea\ee
where	
	\be\la{eee14} \ba
	\int \n|\te-1|^2 \div u dx
	\leq& C\|\n^{1/2}(\te-1)\|_{L^2}^{1/2}\|\n^{1/2}(\te-1)\|_{L^6}^{3/2}\|\na u\|_{L^2}\\
	\leq& C(\on)A_2^{1/4}\|\na\te\|_{L^2}^{3/2}\|\na u\|_{L^2}\\
	\leq& C(\on,M)C_0^{1/16}\left(\|\na\te\|_{L^2}^2+\|\na u\|_{L^2}^2\right)
	 \ea\ee
	owing to \eqref{g1}, \eqref{infty12}, and \eqref{z1}.

 For the second one on the righthand side of \eqref{a2.23}, we will handle it for the short time $[0,\si(T)]$ and the large time $[\si(T),T]$, respectively.
	
	For $t\in[0,\si(T)]$, in light of \eqref{a2.17}, \eqref{g1}, \eqref{3.30}, and \eqref{z1}, one has
	\be\la{eee15} \ba
	 \int|\te-1| |\na u|^2 dx
	&\leq C\|\te-1\|_{L^2}^{1/2}\|\te-1\|_{L^6}^{1/2}\|\na u\|_{L^2}\|\na u\|_{L^6}\\
	&\leq C \left(C_0^{1/4}\|\na\te\|_{L^2}^{1/2}+C_0^{1/6}\|\na\te\|_{L^2}\right)\|\na u\|_{L^2}\\
	&\quad\cdot\left(\|\n \dot
	u\|_{L^2}+\|\na u\|_{L^2}+ \|\na \te\|_{L^2}+C_0^{1/24}\right)\\
	&\leq  C(\on,M) C_0^{7/24}(\|\n^{1/2} \dot u\|_{L^2}^2+1)\\
	&\quad+C(\on,M)C_0^{1/24}\left(\|\na\te\|_{L^2}^2+\|\na u\|_{L^2}^2\right). \ea\ee
	
	For $t\in[\si(T),T]$, it follows from \eqref{g1} and \eqref{z1} that
	\be\la{eee16} \ba
	\int|\te-1| |\na u|^2 dx&\leq C\|\te-1\|_{L^6}\|\na u\|_{L^2}^{3/2}\|\na u\|_{L^6}^{1/2}\\
	&\leq C(\on,M)C_0^{1/16}\left(\|\na\te\|_{L^2}^2+\|\na u\|_{L^2}^2\right). \ea\ee
	where one has used the following fact:
	\be\la{aa0}\ba \sup_{0\le t\le T}\xl(\si\|\na
	u\|_{L^6}\xr)\le C(\on)C_0^{1/24} \ea\ee
	due to \eqref{z1} and \eqref{3.30}.
	
	Substituting \eqref{eee14}--\eqref{eee16}  into \eqref{a2.23}, one gets after using \eqref{z1} that
	\be \ba  \label{jia9}
	&  \sup_{ 0\le t\le T} \int \left(\frac{1}{2}\n |u|^2+ R(1+\n\log \n-\n)+\frac{R}{2(\ga-1)} \n {(\te-1)^2}
	\right)dx\\
	&\quad+ \int_{0}^{T} (\mu \|\curl u\|_{L^2}^2 +(2\mu+\lambda)\|\div u\|_{L^2}^2+ \ka \|\na\te\|_{L^2}^2)dt\\
	&\le  C(\on,M) C_0^{1/24} \int_{ 0}^{T}  ( \|\na u\|_{L^2}^2+\| \na \te \|_{L^2}^2 ) dt
	\\&\quad + C(\on,M) C_0^{7/24}\left(\int_{0}^{\si(T)} \|\rho \dot u\|_{L^2}^2 dt+1\right)\\
	&\le  C(\on, M) C_0^{7/24},
	\ea \ee
	Thus, one deduces from \eqref{jia9} and \eqref{ljq01} that
	\be\notag A_2(T)\le C(\on, M) C_0^{7/24} \le C_0^{1/4},\ee
	provided \be \ba \label{jia11} C_0\le \ve_3\triangleq\min
	\left\{1,  (C(\on,M))^{-24}\right\}.\ea\ee The proof
	of Lemma \ref{le3} is completed.
\end{proof}

We now in position to derive a uniform (in time) upper bound for the density, which turns out to be the key to obtaining all the higher order estimates and thus to extending the classical solution globally.

\begin{lemma}\la{le7}
Under the conditions of Proposition \ref{pr1}, there exists a positive constant $\ve_0$ depending only on   $\mu,\,\lambda,\, \ka,\, R,\, \ga,\, \on,\, \bt $, $\O$, and $M$ such that if $(\rho,u,\te)$ is a smooth solution to the problem (\ref{a1})--(\ref{ch2}) on $\Omega\times (0,T] $ satisfying  (\ref{z1}) with $K$ as
in Lemma \ref{le2}, the following estimate holds:
\be \la{a3.7}
\sup_{0\le t\le T}\|\n(\cdot,t)\|_{L^\infty}  \le
\frac{3\on }{2},
\ee
provided $C_0\le \ve_0$ with $\ve_0$ defined in (\ref{xjia11}).
\end{lemma}

\begin{proof}
First, it follows from  \eqref{k1}, \eqref{6yue}, and \eqref{z1} that
\be\la{3.89}\ba
\int_{\si(T)}^T\|\te-1\|^2_{L^\infty}dt
\le & C(\on)\int_{\si(T)}^T\|\na\te\|_{L^2} \|\na^2\te\|_{L^2}dt\\
\le& C(\on) \left(\int_{\si(T)}^T\|\na\te\|^2_{L^2}dt\right)^{1/2}\left(\int_{\si(T)}^T \|\na^2\te\|^2_{L^2}dt\right)^{1/2} \\
\le& C(\on,M) C_0^{1/8}.
\ea\ee

Next, it can be deduced from \eqref{g1}, \eqref{g2}, (\ref{h19}), (\ref{ae26}), and (\ref{z1}) that
\be\la{3.90}\ba &\int_0^{\si(T)}\|G\|_{L^\infty}dt\\
&\le C\int_0^{\si(T)}\|\na G\|_{L^2}^{1/2} \|\na G\|_{L^6}^{1/2}dt\\
&\le C \int_0^{\si(T)}\|\n \dot u\|_{L^2}^{1/2}\|\na\dot u\|_{L^2}^{1/2}dt\\
&\le C \int_0^{\si(T)}\left(\si\|\n \dot u\|_{L^2}\right)^{1/4} \left(\si\|\n \dot u\|^{2}_{L^2}\right)^{1/8} \left(\si\|\na \dot u\|^2_{L^2}\right)^{1/4}\si^{-5/8}dt\\
&\le C(\on,M)C_0^{1/48}\left(\int_0^{\si(T)} \si\|\na \dot u\|^2_{L^2} dt\right)^{1/4}\left(\int_0^{\si(T)} \si^{-5/6}dt\right)^{3/4} \\
&\le C(\on,M)C_0^{1/48},
\ea\ee
and
 \be\la{3.91}\ba  \int_{\si(T)}^T\|G\|^2_{L^\infty}dt
  &\le C\int_{\si(T)}^T\|\na G\|_{L^2} \|\na G\|_{L^6} dt
   \\ &\le C(\on,M)\int_{\si(T)}^T\left(\|\n^{1/2} \dot u\|^2_{L^2}+
 \|\na\dot u\|_{L^2}^2 \right)dt\\ &\le C(\on,M)C_0^{1/6}.
    \ea\ee

Using (\ref{hj1}), one can rewrite   $(\ref{a1})_1$ in terms of the Lagrangian coordinates as follows
\bnn\ba
(2\mu+\lambda) D_t \n&=-R \n(\n-1)- \n^2(\te-1)-\n G\\
&\le -R(\n-1)+C(\on)\|\te-1\|_{L^\infty}+C(\on)\| G\|_{L^\infty},
\ea\enn
which gives
\be\la{3.92}\ba
D_t (\n-1)+\frac{R}{2\mu+\lambda} (\n-1)\le C(\on)\|\te-1\|_{L^\infty}+C(\on)\| G\|_{L^\infty}.
\ea\ee

Finally, applying Lemma \ref{le1} with
$$y=\n-1, \quad\al=\frac{R}{2\mu+\lambda} ,\quad g=C(\on)\|\te-1\|_{L^\infty}+C(\on)\| G\|_{L^\infty},\quad T_1=\si(T),$$
 we thus deduce from (\ref{3.92}), (\ref{3.88}), \eqref{3.89}--(\ref{3.91}), and (\ref{2.34}) that
 \bnn\ba\n
 & \le \on+1 +C\left(\|g\|_{L^1(0,\si(T))}+\|g\|_{L^2(\si(T),T)}\right) \le \on+1 +C(\on,M)C_0^{1/48} ,
 \ea\enn
 which gives \eqref{a3.7}
 provided
  \be \label{xjia11}C_0\le \ve_0\triangleq\min\left\{\ve_1,\cdots,\ve_4\right\},\ee
  with $\ve_4\triangleq\left(\frac{\hat \n-2 }{2C(\on,M) }\right)^{48}$.
We thus complete the proof of Lemma \ref{le7}.
\end{proof}

Finally, we end this section by summarizing some uniform estimates on $(\n,u,\te)$ which will be useful for higher-order ones in the next section.
\begin{lemma}\la{le8}
	In addition to the conditions of Proposition \ref{pr1}, assume that $(\rho_0,u_0,\te_0)$ satisfies     (\ref{z01}) with $\ve_0$ as in Proposition \ref{pr1}.
Then there exists a  positive constant    $C $     depending only   on  $\mu,\,\lambda,\, \ka,\, R,\, \ga,\, \on,\,\bt$, $\O$, and $M$  such that if $(\rho,u,\te)$  is a smooth solution to the problem (\ref{a1})--(\ref{ch2}) on $\Omega\times (0,T] $  satisfying (\ref{z1}) with $K$ as in Lemma \ref{le2}, it holds:
	\be \la{ae3.7}\sup_{0\le t\le T}\si^2\int \n|\dot\te|^2dx + \int_0^T\si^2 \|\na\dot\te\|_{L^2}^2dt\le C.\ee
Moreover, it holds that
\be\la{vu15}\ba
&\sup_{0\le t\le T}\left(  \si\|\na u \|^2_{L^6}+\si^2\| \na\te\|^2_{H^1}\right)\\
&+\int_0^T(\si \|\na u \|_{L^4}^4+\si^2\|\na\te_t\|^2_{L^2}+\si\|\n -1\|_{L^4}^4)dt\le C.
\ea\ee
\end{lemma}

\begin{proof}
First, applying the operator $\pa_t+\div(u\cdot) $ to (\ref{a1})$_3 $ and using  (\ref{a1})$_1$, one   gets
\be\la{3.96}\ba
&\frac{R}{\ga-1} \n \left(\pa_t\dot \te+u\cdot\na\dot \te\right)\\
&=\ka \Delta  \te_t +\ka \div (\Delta \te u)+\left( \lambda (\div u)^2+2\mu |\mathfrak{D}(u)|^2\right)\div u +R\n \te  \pa_ku^l\pa_lu^k\\
&\quad -R\n \dot\te \div u-R\n \te\div \dot u +2\lambda \left( \div\dot u-\pa_ku^l\pa_lu^k\right)\div u\\
&\quad + \mu (\pa_iu^j+\pa_ju^i)\left( \pa_i\dot u^j+\pa_j\dot u^i-\pa_iu^k\pa_ku^j-\pa_ju^k\pa_ku^i\right).
\ea\ee
Direct calculations show that
\be \ba\la{bea}
 \int  (\Delta  \te_t + \div (\Delta \te u)) \dot \te dx &=  - \int  (\na  \te_t \cdot \na \dot\te + \Delta \te u \cdot \na \dot \te) dx\\
&= - \int  |\na \dot\te|^2 dx  + \int ( \na(u\cdot \na \te) \cdot \na \dot \te - \Delta \te u \cdot \na \dot \te) dx.
\ea \ee
Multiplying (\ref{3.96}) by $\dot \te$ and integrating the resulting equality over $\O$, it holds that
\be\la{3.99}\ba
& \frac{R}{2(\ga-1)}\left(\int \n |\dot\te|^2dx\right)_t + \ka   \|\na\dot\te\|_{L^2}^2 \\
&\le  C  \int|\na \dot \te|\left( |\na^2\te||u|+ |\na \te| |\na u|\right)dx+C\int  \n|\te-1| |\na\dot u| |\dot \te|dx\\
&\quad +C(\on)  \int|\na u|^2|\dot\te|\left(|\na u|+|\te-1| \right)dx+C   \int |\na\dot u|\n|\dot \te| dx \\
&\quad +C (\on)  \int\left( |\na u|^2|\dot \te|+\n  |\dot
\te|^2|\na u|+|\na u| |\na\dot u| |\dot \te|\right)dx \\
&\le C\|\na u\|^{1/2}_{L^2}\|\na u\|^{1/2}_{L^6}\|\na^2\te\|_{L^2}\|\na \dot \te\|_{L^2}\\
&\quad+C(\on)\|\n(\te-1)\|_{L^2}^{1/2}\|\na\te\|_{L^2}^{1/2}\|\na\dot u\|_{L^2} \|\dot\te\|_{L^6}\\
&\quad+C(\on)  \|\na u\|_{L^2}\|\na u\|_{L^6}\left(\|\na u\|_{L^6}+\|\na \te\|_{L^2}\right)
\|\dot\te\|_{L^6} +C  \|\na\dot u\|_{L^2} \|\n\dot\te\|_{L^2} \\
&\quad+C(\on)  \|\na u\|^{1/2}_{L^6}\|\na u\|^{1/2}_{L^2} \|\dot\te\|_{L^6}\left(\|\na u\|_{L^2}
+\|\n\dot\te\|_{L^2}+\|\na\dot u\|_{L^2}\right)  \\
&\le \frac{\ka}{2}  \|\na\dot\te\|_{L^2}^2 + C(\on,M) \|\na u\|_{L^6}^2\|\na \te \|_{L^2}^2+C(\on)\|\na u\|_{L^2}^2\|\na u\|_{L^6}^4\\
&\quad+C(\on,M) \left(1+\|\na u\|_{L^6}+\|\na\te\|_{L^2}^2\right) \left(\|\na^2\te\|_{L^2}^2+\|\na\dot u\|_{L^2}^2+\|\rho^{1/2} \dot \te\|_{L^2}^2\right)\\
&\quad+ C(\on,M) \|\na u\|_{L^6}\|\na u\|_{L^2}^2,
\ea\ee
where we have used \eqref{bea}, \eqref{g1}, \eqref{g2}, and \eqref{z1}.

Multiplying (\ref{3.99}) by $\si^2$ and integrating the resulting inequality over $(0,T),$
we obtain after integration by parts that
\bnn\ba
& \sup_{0\le t\le T}\si^2\int \n|\dot\te|^2dx + \int_0^T\si^2 \|\na\dot\te\|_{L^2}^2dt  \\
&\le C(\on,M) \sup_{0\le t\le T} \left(\si^2\|\na u\|_{L^6}^2 \right) \int_0^T\left( \|\na u\|_{L^2}^2 \|\na u\|_{L^6}^2+\|\na \te \|_{L^2}^2\right)dt \\
&\quad + C(\on,M) \sup_{0\le t\le T} \left(\si\left(1+\|\na u\|_{L^6}+\|\na\te\|_{L^2}^2\right)\right)\\
&\quad \quad \quad \quad \quad \quad \quad \cdot\int_0^T\si\left(\|\na^2\te\|_{L^2}^2+\|\na\dot u\|_{L^2}^2+\|\rho^{1/2} \dot \te\|_{L^2}^2\right)dt\\
&\quad +C(\on,M) \sup_{0\le t\le T} \left(\si\|\na u\|_{L^6} \right) \int_0^T\|\na u\|_{L^2}^2dt+C\int_0^T\si\|\rho^{1/2} \dot \te\|_{L^2}^2dt\\
&\le  C(\on,M), \ea\enn
where we have used   (\ref{z1}), \eqref{aa0}, (\ref{ae26}),  (\ref{k1}), and the following fact:
\be\ba\notag&\int_0^T\|\na u\|_{L^2}^2\|\na u\|_{L^6}^2dt\\
&\le C(\on)\int_0^T\|\na u\|_{L^2}^2\left(\|\n^{1/2} \dot u\|_{L^2}^2+\|\na u\|_{L^2}^2+ \|\na \te\|_{L^2}^2+C_0^{1/12}\right)dt \\
&\le C(\on,M)\int_0^T\left(\|\n^{1/2} \dot u\|_{L^2}^2+\|\na u\|_{L^2}^2+ \|\na \te\|_{L^2}^2\right)dt\\
&\le C(\on,M)\ea\ee
due to \eqref{3.30} and \eqref{infty12}.

Finally, it's easy to deduce from (\ref{z1}), (\ref{3.30}),  (\ref{lop4}), (\ref{m20}), (\ref{ae9}), (\ref{ae3.7}), and (\ref{eee23})--(\ref{k1}) that
\be \la{vu02}\ba
&\sup_{0\le t\le T}\left(  \si\|\na u \|^2_{L^6}+\si^2\| \na\te  \|^2_{H^1}\right)\\
 & +\int_0^T \left(\si\|\na u  \|_{L^4}^4+\si\|\na\te \|_{H^1}^2+\si\|\n-1\|_{L^4}^4\right)dt
 \le C(\on,M), \ea\ee
which along with  (\ref{z1}), (\ref{k1}), and \eqref{ae3.7} gives
\be\la{vu01}\ba
\int_0^T  \si^2 \|  \na\te _t\|_{L^2}^2dt
&\le C\int_0^T  \si^2\|\na \dot \te \|_{L^2}^2  dt+ C\int_0^T  \si^2\|\na(u \cdot\na  \te )\|_{L^2}^2dt\\
&\le C(\on,M) +C(\on,M)\int_0^T\si^2\left(\|\na u \|_{L^3}^2+\|u \|_{L^\infty}^2\right)\|\na^2 \te \|_{L^2}^2dt  \\ &\le C(\on,M).\ea\ee
Hence, (\ref{vu15}) is derived from (\ref{vu02}) and (\ref{vu01}) immediately.

The proof of Lemma \ref{le8} is finished.
\end{proof}

\section{\la{se4} A priori estimates (II): higher-order estimates}

In this section, we will derive the higher-order estimates of smooth solution $(\rho, u, \te)$ to problem (\ref{a1})--(\ref{ch2})  on $ \Omega\times (0,T]$ with initial data $(\n_0 ,u_0,\te_0)$ satisfying (\ref{co3}) and (\ref{3.1}).

We shall assume that both (\ref{z1}) and (\ref{z01}) hold as well. To proceed,
we define $\tilde g $ as
\be \la{co12}\tilde g\triangleq\n_0^{-1/2}\left(
-\mu \Delta u_0-(\mu+\lambda)\na\div u_0+R\na (\n_0\te_0)\right).\ee
Then it follows from (\ref{co3}) and (\ref{3.1}) that
\be\la{wq01}\tilde g\in L^2.\ee
From now on, the generic constant $C $ will depend only  on \bnn
T, \,\, \| \tilde g\|_{L^2},    \,\|\n_0-1\|_{H^2 \cap W^{2,q}}  ,   \,  \,\| u_0\|_{H^2},  \ \,
\| \te_0-1\|_{H^1} , \enn
besides  $\mu,\,\lambda,\, \ka,\, R,\, \ga,\, \on,\,\bt, $ $\O$, and $M.$

We begin with the following estimates on the spatial gradient of
the smooth solution $(\rho,u,\te).$

\begin{lemma}\la{le11}
	The following estimates hold:
	\be\label{lee2}\ba
	&\sup_{0\le t\le T} \left(\|\rho^{1/2}\dot u\|_{L^2}^2 + \sigma\|\rho^{1/2}\dot \te\|_{L^2}^2 +\|\te-1\|_{H^1}^2 + \sigma \|\na^2 \theta\|_{L^2}^2
 \right)  \\
	  &\quad+\ia\left(  \|\nabla\dot u\|_{L^2}^2  + \|\rho^{1/2}\dot \te\|_{L^2}^2+ \|\nabla^2 \theta\|_{L^2}^2 +\sigma \|\nabla\dot \te\|_{L^2}^2 \right) dt\le C,
	\ea\ee
	and
	\be\la{qq1}
	\sup_{0\le t\le T}\left(\|u\|_{H^2} +\|\n-1\|_{H^2}\right)
	+ \int_0^{T}\left( \|\nabla u\|_{L^{\infty}}^{3/2} + \si \| \na^3 \te\|_{L^2}^2+\|u\|_{H^3}^2 \right)dt\le C.
	\ee
\end{lemma}

\begin{proof}
	The proof is divided into the following two steps.
	
	\noindent \textbf{Step I: The proof of (\ref{lee2}).}

First, for $\varphi(t)$ as in \eqref{wq3}, taking $m=0$ in \eqref{e8}, one gets
\be\la{ae8}\ba
& \varphi'(t) +   \int\left( \frac{C_1}{2}|\nabla\dot{u}|^2   +\n|\dot \te|^2\right)dx \\
&\le -2 \left(\int_{\p\O} G\left( u\cdot \na n \cdot u \right)dS\right)_t+C\left(\|\n^{1/2} \dot u\|_{L^2}^2+\|\na u\|_{L^2}^2+\|\na    \te\|_{L^2}^2\right)\\
&\quad +C\left(\|\n^{1/2}\dot u\|_{L^2}^3+\|\na\te\|_{L^2}^3+\|\na u\|_{L^2}^2+\|\n-1\|_{L^2}^2\right) \\
& \quad +C \left(\|\na u\|_{L^2}^2+\|\na\te\|_{L^2}^2\right)\left( \|\n\dot u\|_{L^2}^2+ \|\na \te \|_{L^2}^2+1 \right) \\
&\le -2 \left(\int_{\p\O} G\left( u\cdot \na n \cdot u \right)dS\right)_t+C   \left( \|\n^{1/2}\dot u\|_{L^2}^2 + \|\na\te\|_{L^2}^2\right) (\varphi(t)+1)+C
\ea  \ee
due to   (\ref{z1}), \eqref{a2.112}, \eqref{wq2}, (\ref{ae9}),  and \eqref{m20}.
Taking into account of the compatibility condition \eqref{co2}, we can define
\bnn \ba
\sqrt{\n} \dot u(x,t=0) \triangleq
-\tilde g,
\ea\enn
which along with    (\ref{e6}), (\ref{2.48}), and \eqref{wq01} yields that
\be \la{wq02}
|\varphi(0)|\le C\| \tilde g\|_{L^2}^2+C\le C.
\ee
Then, integrating \eqref{ae8} over $(0,t)$, one obtains after using \eqref{z1}, \eqref{b2}, \eqref{h19}, \eqref{wq2}, and \eqref{wq02} that
\be\ba \la{q1}
& \varphi(t)+\int_0^t\int\left(\frac{C_1}{2} |\nabla\dot{u}|^2   +\n|\dot \te|^2\right)dxds\\
&\le  2\left| \int_{\p\O} G\left( u\cdot \na n \cdot u \right)dS\right|(t)+C\int_0^t \left(\|\n^{1/2}\dot u\|_{L^2}^2+ \|\na \te \|_{L^2}^2\right)\varphi ds+C\\
&\le  C (\|\na u\|_{L^2}^2\|\n^{1/2}\dot u\|_{L^2})(t)+C\int_0^t \left(\|\n^{1/2}\dot u\|_{L^2}^2+ \|\na \te \|_{L^2}^2\right)\varphi ds+C\\
&\le  \frac{1}{2} \varphi(t) + C \int_{0}^{t} \left(\|\n^{1/2}\dot u\|_{L^2}^2+ \|\na \te \|_{L^2}^2\right)\varphi ds+C.
\ea\ee
Applying Gr\"{o}nwall's inequality to \eqref{q1} and 
 using \eqref{z1} and (\ref{wq2}), it holds
\be\label{lee3}
\sup_{0\le t\le T} \left(\|\rho^{1/2}\dot u\|_{L^2}^2
+\|\na \te\|_{L^2}^2 \right) + \ia\int\left(|\nabla\dot
u|^2+\n|\dot\te|^2\right)dxdt\le C,
\ee
which together with \eqref{eee22} implies
\be\la{ff1} \ba
\sup_{0\le t\le T} \|\te-1\|_{L^2} \le C.
\ea \ee

Next,
multiplying (\ref{3.99}) by $\sigma$ and integrating over $(0,T)$  lead to
\be  \ba \la{a5}
&\sup\limits_{0\le t\le T} \si \int \n|\dot\te|^2dx+\int_0^T \si \|\na\dot\te\|_{L^2}^2dt\\
&\le
C\int_0^T\left(  \|\na^2\te\|_{L^2}^2+ \|\n^{1/2}\dot\te\|_{L^2}^2
+ \|\na \te\|_{L^2}^2+\|\na\dot u\|_{L^2}^2 \right)dt + C\\
&\le  C,
\ea\ee
where we have used (\ref{lee3}),  (\ref{z1}), (\ref{m20}), \eqref{ae9},   (\ref{lop4}), and the following fact:
\be\la{w1}\sup_{0\le t\le T}\|\na u\|_{L^6}\le C\ee
due to \eqref{3.30}, \eqref{z1}, and \eqref{lee3}.
Then, it can be deduced from \eqref{lee3}, \eqref{a5}, (\ref{lop4}), (\ref{ae9}), and (\ref{w1}) that
\be \la{va6} \sup\limits_{0\le t\le T} \si\|\na^2\te\|_{L^2}^2 + \int_{0}^{T}\|\na^2\te\|_{L^2}^2 dt \le C,\ee
which along with \eqref{lee3}, \eqref{ff1}, and \eqref{a5} gives \eqref{lee2}.

\noindent \textbf{Part II: The proof of (\ref{qq1}).}

First, standard calculations show  that for $ 2\le p\le 6$,
\be\la{L11}\ba
\partial_t\norm[L^p]{\nabla\rho}
&\le C\norm[L^{\infty}]{\nabla u} \norm[L^p]{\nabla\rho}+C\|\na^2u\|_{L^p}\\
&\le C\left(1+\|\na u\|_{L^{\infty}}+\|\na^2\te \|_{L^2}\right)
\norm[L^p]{\nabla\rho}\\
&\quad +C\left(1+\|\na\dot u\|_{L^2}+\|\na^2\te \|_{L^2}\right), \ea\ee
where we have used
\be\ba\la{ua1}
\|\na^2 u\|_{L^p} & \le C(\|\div u\|_{W^{1,p}} + \|\curl u\|_{W^{1,p}} + \|\na u\|_{L^2}  ) \\
& \le C\left(\|\n\dot u\|_{L^p}+\|\n\dot u\|_{L^2}+\|\n\te-1\|_{W^{1,p}}\right)+C\left(\|\na u\|_{L^2}+\|\n \te-1\|_{L^2}\right)\\
&\le  C\left(1+\|\na\dot u\|_{L^2}+\|\na^2\te\|_{L^2}+(\|\na^2\te \|_{L^2} + 1)\|\nabla\n\|_{L^p}\right)
\ea\ee
due to \eqref{g1}, \eqref{uwkq}, \eqref{z1},  and  \eqref{lee2}.
It follows from Lemma \ref{le9}, \eqref{z1}, and (\ref{ua1})  that
\be\la{u13}\ba
\|\na u\|_{L^\infty }
&\le C\left(\|{\rm div}u\|_{L^\infty}+\|\curl u\|_{L^\infty}\right)\log(e+\|\na^2 u\|_{L^6}) +C\|\na u\|_{L^2}+C \\
&\le C\left( \|{\rm div}u\|_{L^\infty } + \|\curl u\|_{L^\infty }
\right)\log(e+ \|\na\dot u\|_{L^2 } + \|\na^2\te \|_{L^2})\\
&\quad +C\left(\|{\rm div}u\|_{L^\infty }+ \|\curl u\|_{L^\infty } \right)
\log\left(e  + \|\na \n\|_{L^6}\right)+C.
\ea\ee
Denote
\bnn\la{gt}\begin{cases}
	f(t)\triangleq  e+\|\na	\n\|_{L^6},\\
	g(t)\triangleq 1+  \|{\rm div}u\|_{L^\infty }^2+ \|\curl u\|_{L^\infty }^2
	+ \|\na\dot u\|_{L^2 }^2 +\|\na^2\te \|_{L^2}^2,
\end{cases}\enn
then one obtains after submitting \eqref{u13} into (\ref{L11}) with  $p=6$ that
\bnn f'(t)\le   C g(t) f(t)\ln f(t) ,\enn
which implies \be\la{w2}  (\ln(\ln f(t)))'\le  C g(t).\ee
Note that it follows from
(\ref{hj1}),  (\ref{g1}), (\ref{lee2}),  (\ref{z1}), and (\ref{bz6})  that
\be \la{p2}\ba
& \int_0^T\left(\|\div u\|^2_{L^\infty}+\|\curl u\|^2_{L^\infty} \right)dt \\
& \le  C\int_0^T\left(\|G\|^2_{L^\infty}+ \|\curl u\|^2_{L^\infty}+\|\n\te-1\|^2_{L^\infty}\right)dt \\
&\le  C\ia\left(\| \na G\|^2_{L^2}+\| \na G\|^2_{L^6} + \| \curl u\|^2_{W^{1,6}} + \|\te-1\|_{L^\infty}^2\right)dt + C \\
& \le   C\ia\left(\|\na G\|^2_{L^2}+\|\n\dot u\|^2_{L^6}+\|\na\curl  u\|^2_{L^2}+\|\na^2\te \|_{L^2}^2\right)dt+C \\
&\le C\ia(\|\n\dot u\|_{L^2}^2+\|\na \dot u\|^2_{L^2}+\|\na^2\te \|_{L^2}^2)dt+C \\
&\le  C,
\ea\ee
which combined with  Gr\"{o}nwall's inequality, \eqref{p2}, and \eqref{lee2}  yields that
\be \la{u113} \sup\limits_{0\le t\le T}\|\nabla \rho\|_{L^6}\le C.\ee
This together with (\ref{u13}),  \eqref{p2}, and (\ref{lee2}) leads to
\be \la{v6}\ia\|\nabla u\|_{L^\infty}^{3/2}dt \le C.\ee
Moreover, taking $p=2$ in \eqref{ua1}, we get by using \eqref{v6}, (\ref{lee2}), and  Gr\"{o}nwall's inequality that
  \bnn\la{aa94}\ba
\sup\limits_{0\le t\le T}\|\nabla \n\|_{L^2}
\le C\ea\enn
which gives
  \bnn\la{aaa94}\ba
\sup\limits_{0\le t\le T}\|\nabla P\|_{L^2}
\le C\sup\limits_{0\le t\le T}\left(\|\na\te\|_{L^2}+\|\na\n\|_{L^2}+\|\te-1 \|_{L^6}\|\na\n\|_{L^3}\right)
\le C \ea\enn
due to (\ref{lee2}) and (\ref{u113}).
Combining this with  \eqref{lee2} and \eqref{ua1} leads to
\be\ba\la{aa95}
\sup\limits_{0\le t\le T} \|\na^2 u\|_{L^2}
\le &C \sup\limits_{0\le t\le T}\left(\|\n\dot u\|_{L^2}+\|\nabla P\|_{H^1}+\| \na u\|_{L^2}\right)\le C.
\ea\ee

Next, applying operator $\p_{ij}~(1\le i,j\le 3)$ to $(\ref{a1})_1$ gives
\be\la{4.52}  (\p_{ij} \n)_t+\div (\p_{ij} \n u)+\div (\n\p_{ij} u)+\div(\p_i\n\pa_j u+\p_j\n\p_i u)=0. \ee
Multiplying (\ref{4.52}) by $2\p_{ij} \n$ and  integrating the resulting equality over $\O,$ it holds
\be\la{ua2}\ba
\frac{d}{dt}\|\na^2\n\|^2_{L^2}
& \le C(1+\|\na u\|_{L^{\infty}})\|\na^2\n\|_{L^2}^2+C\|\na u\|^2_{H^2}\\
& \le C(1+\|\na u\|_{L^{\infty}} +\|\na^2\te \|_{L^2}^2)(1+\|\na^2\n\|_{L^2}^2) +C\|\na\dot u \|^2_{L^2}, \ea\ee
where one has used \eqref{z1}, \eqref{u113}, and the following estimate:
\be\ba\la{va2}
\|\nabla u\|_{H^2}&\le C\left(\|\div u\|_{H^2}+\|\curl u\|_{H^2}+\|\na u\|_{L^2}\right)\\
&\le C\left(\|G\|_{H^2}+\|\curl u\|_{H^2}+\|\n\te-1\|_{H^2}+\|\na u\|_{L^2}\right)\\
&\le C+C\|\na(\n\dot u)\|_{L^2}+C\|(\n-1)(\te-1)\|_{H^2}\\
&\quad+ C\|\n-1\|_{H^2}+C\|\te-1\|_{H^2}\\
&\le C\|\na\n\|_{L^3}\|\dot u\|_{L^6}+C\|\na \dot u\|_{L^2}+C\|\n-1\|_{H^2}\|\te-1\|_{H^2}
\\&\quad+C\|\na^2\n\|_{L^2}+C\|\na^2\te\|_{L^2}+C
\\ &\le  C\|\na\dot u \|_{L^2}+ C  (1+\|\na^2\te \|_{L^2})(1+\|\na^2\n\|_{L^2}) +C
\ea\ee
due to (\ref{hs}), (\ref{uwkq}), (\ref{lee2}), (\ref{u113}),  and (\ref{z1}).
Applying Gr\"{o}nwall's inequality to \eqref{ua2}, one gets after using  (\ref{lee2})  and (\ref{v6}) that
\be\la{ja3} \sup_{0\le t\le T} \|\na^2\n \|_{L^2}  \le C,\ee
which together with \eqref{eee8}, \eqref{va2}, and \eqref{lee2} gives
\be\la{ja4} \int_0^T\|u\|_{H^3}^2 dt\le C . \ee

Finally,
applying the standard $H^1$-estimate to   elliptic problem (\ref{3.29}), one derives from \eqref{z1},  \eqref{lee2}, \eqref{u113}, \eqref{eee8}, and \eqref{aa95} that
\be\la{ex4}\ba
\|\na^2\te\|_{H^1}
& \le C\left(\|\n \dot \te\|_{H^1}+\|\n\te\div u\|_{H^1}+\||\na u|^2\|_{H^1}+\|\na\te\|_{L^2}^2\right)\\
& \le C\left(1+ \|\na \dot \te\|_{L^2} +  \|\rho^{1/2} \dot \theta\|_{L^2}  + \|\na(\n\te\div u)\|_{L^2}+ \||\na u||\na^2u|\|_{L^2} \right) \\
& \le C\left(1+ \|\na \dot \te\|_{L^2} +  \|\rho^{1/2} \dot \theta\|_{L^2}  +\|\te-1\|_{H^1}\|\na u\|_{H^1}+C\|\na^3 u\|_{L^2} \right).
\ea\ee
This along with \eqref{z1}, \eqref{a2.112}, \eqref{ff1}, \eqref{eee8}, \eqref{ja3}, \eqref{ja4}, \eqref{u113}, \eqref{v6}, \eqref{lee2}, and \eqref{aa95} yields (\ref{qq1}).

 The proof of Lemma \ref{le11} is finished.
\end{proof}

\begin{lemma}\la{le9-1}
	The following estimates hold:
	\be\la{va5}\ba&
	\sup\limits_{0\le t\le T}
	\|\n_t\|_{H^1}
	+\int_0^T\left(\|  u_t\|_{H^1}^2+\si \| \te_t\|_{H^1}^2+\| \n u_t\|_{H^1}^2+\si \|\n \te_t\|_{H^1}^2
	\right)dt\le C,
	\ea  \ee
	and
	\be\la{vva5}\ba
	\int_0^T \sigma \left( \|(\n u_t)_t\|_{H^{-1}}^2+\|(\n \te_t)_t\|_{H^{-1}}^2
	\right)dt\le C.\ea\ee
\end{lemma}

\begin{proof}

First, it   follows from  (\ref{lee2}) and (\ref{qq1})   that
\be\label{va1}\ba
&\sup_{0\le t\le T}\int\left( \n|u_t|^2 + \si
\n\te_t^2\right)dx +\int_0^T \left(\|\na u_t\|_{L^2}^2+ \si \|\na\te_t\|_{L^2}^2\right)dt\\
&\le \sup\limits_{0\le t\le T}\int \left(\n|\dot u|^2+ \si\n|\dot\te|^2 \right)dx
+\int_0^T\left(\|\na\dot u\|_{L^2}^2+\si \|\na\dot\te\|_{L^2}^2 \right)dt\\
&\quad + \sup\limits_{0\le t\le T}\int \n\left(|u\cdot\na u|^2+\si |u\cdot\na\te|^2\right)dx\\
&\quad +\int_0^T\left(\|\na(u\cdot\na u)\|_{L^2}^2+\si \|\na(u\cdot\na \te)\|_{L^2}^2\right)dt\\
&\le C,\ea\ee
which together with (\ref{lee2}) and (\ref{qq1}) gives
\be\la{vva1} \ba
& \int_0^T\left(\|\na(\n u_t)\|_{L^2}^2+ \si \|\na(\n\te_t)\|_{L^2}^2\right)dt\\
& \le  C\int_0^T\left(\|  \na u_t \|_{L^2}^2+\|  \na \n\|_{L^3}^2\| u_t \|_{L^6}^2
+ \si \|  \na \te_t \|_{L^2}^2+ \si \|  \na \n\|_{L^3}^2\| \te_t \|_{L^6}^2 \right)dt\\
&\le C.\ea\ee

Next, one deduces from $(\ref{a1})_1$, \eqref{lee2}, (\ref{qq1}), and \eqref{hs} that
\bnn\ba\la{sp1}
\|\n_t\|_{H^1}\le& \|\div (\rho u)\|_{H^1}
\le  C \|u\|_{H^2}
(\|\n-1\|_{H^2}+1)
\le C.
\ea \enn
Combining this with (\ref{va1}) and  (\ref{vva1}) gives    (\ref{va5}).

Finally, differentiating $(\ref{a1})_2$ with respect to $t$ yields that
\be   \la{va7}\ba (\n u_t)_t
=-(\n u\cdot\na u)_t + \left(\mu\Delta u+(\mu+\lambda)\na\div u \right)_t -\na P_t.\ea\ee
It   deduces from (\ref{va5}), \eqref{qq1}, \eqref{lee2}, and \eqref{lop4} that
\be   \la{va9}\ba  \|(\n u\cdot\na u)_t  \|_{L^{2}}
&=\| \n_t u\cdot\na u+ \n u_t\cdot\na u + \n u\cdot\na u_t  \|_{L^{2}}\\
&\le C\|\n_t\|_{L^6} \|\na u\|_{L^3}+  C\|u_t\|_{L^6} \|\na u\|_{L^3}+  C\|u \|_{L^\infty} \|\na u_t\|_{L^2}\\
&\le C+   C\|\na u_t\|_{L^2},\ea\ee
and
\be   \la{va10}\ba  \|\na P_t  \|_{L^2}
&=R\|\n_t\na\te +\n \na\te_t +\na\n_t\te +\na\n\te_t\|_{L^2}\\
&\le C\left(\|\n_t\|_{L^6}\|\na\te\| _{L^3}+\|\na\te_t\|_{L^2}
+\|\te\|_{L^\infty}\|\na \n_t\|_{L^2}+\|\na\n\|_{L^3}\|\te_t\| _{L^6}\right)\\
&\le C+C\|\na\te_t\|_{L^2}+C\|\n^{1/2}\te_t\|_{L^2}.\ea\ee
Combining (\ref{va7})--(\ref{va10}) with (\ref{va5}) shows
\be\la{vva04}\ba\int_0^T \si \|(\n u_t)_t\|_{H^{-1}}^2dt\le C.\ea\ee
Similarly, we have \bnn\int_0^T \si  \|(\n \te_t)_t\|_{H^{-1}}^2dt\le C,  \enn
which combined with (\ref{vva04}) implies (\ref{vva5}). The proof of Lemma \ref{le9-1} is completed.
\end{proof}

\begin{lemma}\la{pe1}
	The following estimate holds:
	\be\la{nq1}
	\sup\limits_{0\le t\le T} \si\left(\|\nabla u_t\|^2_{L^2}+\|\n_{tt} \|^2_{L^2} + \|u\|_{H^3}^2\right)
	+ \int_0^T\si \left(\|\rho^{1/2} u_{tt}\|_{L^2}^2+ \|\nabla u_t\|_{H^1}^2\right)dt\le C.
	\ee
\end{lemma}
\begin{proof}
Differentiating  $(\ref{a11})$  with respect to $t$ leads to \be\la{nt0}\ba \begin{cases}
(2\mu+\lambda)\na\div u_t-\mu \na\times \curl u_t\\
= \n u_{tt} +\n_tu_t+\n_tu\cdot\na u+\n u_t\cdot\na u +\n u\cdot\na u_t+\na P_t\triangleq \tilde f
,&\, \text{in}\,\O\times[0,T],\\
u_{t}\cdot n=0,\  \curl u_{t}\times n=0,  &\,\text{on}\,\p\O\times[0,T],\\
u_t\rightarrow0,\,\,&\,as\,\,|x|\rightarrow\infty.
 \end{cases}\ea\ee
Multiplying (\ref{nt0})$_1$ by $u_{tt}$   and integrating  the resulting equality  by parts, one gets  \be\la{sp9} \ba&
\frac{1}{2}\frac{d}{dt}\int \left(\mu|\curl u_t|^2 + (2\mu +\lambda)({\rm div}u_t)^2\right)dx
+\int \rho| u_{tt}|^2dx\\
&=\frac{d}{dt}\left(-\frac{1}{2}\int_{ }\rho_t |u_t|^2 dx- \int_{}\rho_t u\cdot\nabla u\cdot u_tdx
+ \int_{ }P_t {\rm div}u_tdx\right)\\
& \quad + \frac{1}{2}\int_{ }\rho_{tt} |u_t|^2 dx+\int_{ }(\rho_{t} u\cdot\nabla u )_t\cdot u_tdx
-\int_{ }\rho u_t\cdot\nabla u\cdot u_{tt}dx\\
&\quad - \int_{ }\rho u\cdot\nabla u_t\cdot u_{tt}dx - \int_{ }\left(P_{tt}-
\ka(\ga-1)\Delta\te_t\right){\rm div}u_tdx\\
&\quad +\ka(\ga-1)\int_{ } \na\te_t\cdot\na {\rm div}u_tdx
\triangleq\frac{d}{dt}\tilde{I}_0+ \sum\limits_{i=1}^6 \tilde{I}_i. \ea \ee

 Each term $\tilde{I}_i(i=0,\cdots,6)$ can be estimated as follows:

First, it follows from simple calculations, $(\ref{a1})_1,$   (\ref{va5}), \eqref{qq1}, \eqref{lee2}, \eqref{eee8}, and (\ref{va1}) that
\be \ba \la{sp10}
|\tilde{I}_0|&=\left|-\frac{1}{2}\int\rho_t |u_t|^2 dx- \int \rho_t u\cdot\nabla
u\cdot u_tdx+ \int P_t {\rm div}u_tdx\right|\\
&\le C\int  \n |u||u_t||\nabla u_t| dx+C\norm[L^3]{\rho_t}\norm[L^2]
{\nabla u}\norm[L^6]{u_t}+C\|(\n\te)_t\|_{L^2}\|\nabla u_t\|_{L^2}\\
&\le C \|\n^{1/2}u_t\|_{L^2}  \|\nabla u_t\|_{L^2} +C(1+\|\n^{1/2} \te_t\|_{L^2}+\|\n_t\|_{L^2}\\
&\quad+\|\n_t\|_{L^3}\|\te-1\|_{L^6})\|\nabla u_t\|_{L^2}\\
&\le C (1+\|\n^{1/2}\te_t\|_{L^2}) \|\nabla u_t\|_{L^2} ,\ea\ee
\be \la{sp11}\ba
2|\tilde{I}_1|&=\left|\int \rho_{tt} |u_t|^2 dx\right|\\
& \le  C\|\n_{tt}\|_{L^2}\|u_t\|_{L^2}^{1/2}\|u_t\|_{L^6}^{3/2}\\
& \le  C\|\n_{tt}\|_{L^2}(1+\|\na u_t\|_{L^2})^{1/2}\|\na u_t\|_{L^2}^{3/2}\\
& \le  C\|\n_{tt}\|_{L^2}^2+C \|\na u_t\|_{L^2}^4+C,
\ea \ee
\be \la{sp12}\ba
|\tilde{I}_2|&=\left|\int \left(\rho_t u\cdot\nabla u \right)_t\cdot u_{t}dx\right|\\
& = \left|  \int\left(\rho_{tt} u\cdot\nabla u\cdot u_t +\rho_t
u_t\cdot\nabla u\cdot u_t+\rho_t u\cdot\nabla u_t\cdot u_t\right)dx\right|\\
&\le   C\norm[L^2]{\rho_{tt}}\norm[L^6]{\na u}\norm[L^6]{u}\norm[L^6]{u_t}
+C\norm[L^2]{\rho_t}\norm[L^6]{u_t}^2\norm[L^6]{\nabla u} \\
&\quad+C\norm[L^3]{\rho_t}\norm[L^{\infty}]{u}\norm[L^2]{\nabla u_t}\norm[L^6]{u_t}\\
& \le C\norm[L^2]{\rho_{tt}}^2 + C\norm[L^2]{\nabla u_t}^2, \ea \ee
and
\be\ba\la{sp13}
|\tilde{I}_3|+|\tilde{I}_4|&= \left| \int \rho u_t\cdot\nabla u\cdot u_{tt} dx\right|
+\left| \int \rho u\cdot\nabla u_t\cdot u_{tt}dx\right|\\
& \le   C\|\n^{1/2}u_{tt}\|_{L^2}\left(\|u_t\|_{L^6}\|\na u\|_{L^3}
+\|u\|_{L^\infty}\|\na u_t\|_{L^2}\right) \\
& \le  \frac{1}{4}\norm[L^2]{\rho^{{1/2}}u_{tt}}^2 + C \norm[L^2]{\nabla u_t}^2.\ea\ee
Then, by virtue of
(\ref{op3}), (\ref{va5}), \eqref{va10}, and Lemma \ref{le11}, it holds
\bnn\ba & \|P_{tt}-\ka(\ga-1)\Delta \te_t\|_{L^2}\\
&\le  C\|(u\cdot\na P)_t\|_{L^2}+C\|(P\div u)_t\|_{L^2}+C\||\na u||\na u_t|\|_{L^2}\\
&\le  C\|u_t\|_{L^6}\|\na P\|_{L^3}+C\|u\|_{L^\infty}\|\na P_t\|_{L^2}
+C\|P_t\|_{L^6}\|\na u\|_{L^3}\\
&\quad  +C\|P\|_{L^\infty}\|\na u_t\|_{L^2}+C\|\na u\|_{L^\infty}\|\na u_t\|_{L^2}\\
&\le C\left(1+\|\na u\|_{L^\infty}+\|\na^2\te \|_{L^2}\right)\|\na u_t\|_{L^2} +C\left(1+\|\na\te_t\|_{L^2}+\|\n^{1/2}\te_t\|_{L^2}\right),
\ea\enn
which yields
\be\ba\la{sp15}
|\tilde{I}_5|&=\left|\int\left(P_{tt}-\ka(\ga-1)\Delta \te_t\right){\rm div}u_tdx\right|\\
&\le\norm[L^2]{P_{tt}-\ka(\ga-1)\Delta \te_t}\norm[L^2]{\na u_t}\\
&\le C\left(1+\|\na u\|_{L^\infty}+C\|\na^2\te \|_{L^2}\right)\|\na u_t\|_{L^2}^2\\
&\quad+C\left(1+\|\na\te_t\|^2_{L^2}+\|\n^{1/2}\te_t\|_{L^2}^2\right).
\ea\ee
Next, combining a priori estimate on Lam\'{e}'s system \eqref{nt0} similar to \eqref{ua1} with Lemmas \ref{zhle},   \ref{le11}, \eqref{va5}, and (\ref{va1}) gives that
\be\la{nt4}\ba
\|\na^2u_t\|_{L^2}&\le C\|\tilde f\|_{L^2}+C\|u_t\|_{L^2}\\
&\le C\|\n u_{tt}\|_{L^2}
+C\|\n_t\|_{L^3}\|u_t\|_{L^6}+C\|\n_t\|_{L^3}\|\na u\|_{L^6}\|u\|_{L^\infty}\\
&\quad  +C\|u_t\|_{L^6}\|\na u\|_{L^3}+C\|\na u_t\|_{L^2}\|u\|_{L^\infty}+C\|\na P_t\|_{L^2}\\
&\le C\left(\|\n u_{tt}\|_{L^2}+\|\na u_t\|_{L^2}+\|\n^{1/2} \te_{t}\|_{L^2}+\|\na \te_t\|_{L^2}+1\right),\ea\ee
which immediately leads  to\be \la{asp16}\ba
|\tilde{I}_6| &= \left| \ka(\ga-1)\int_{ } \na\te_t\cdot\na{\rm div}u_tdx \right|   \\
& \le  C \|\na^2u_t\|_{L^2}\|\na\te_t\|_{L^2}\\
& \le \frac{1}{4} \|\n^{1/2} u_{tt}\|^2_{L^2}+ C\left(1+\|\na u_t\|^2_{L^2}+\|\n^{1/2} \te_t\|^2_{L^2}+\|\na\te_t\|^2_{L^2}\right).
\ea\ee

Putting   (\ref{sp11})--(\ref{sp15})  and (\ref{asp16}) into
(\ref{sp9}) yields
\be\la{4.052} \ba
& \frac{d}{dt}\int \left(\mu|\curl u_t|^2 + (2\mu +\lambda)({\rm div}u_t)^2-2\tilde{I}_0\right)dx
+\int \rho| u_{tt}|^2dx\\
&\le  C\left(1+\|\na u\|_{L^\infty}+\|\na u_t\|_{L^2}^2+\|\na^2\te \|_{L^2}^2 \right)\|\na u_t\|_{L^2}^2 \\
&\quad +C\left(1+\|\n_{tt}\|_{L^2}^2+\|\n^{1/2} \te_t\|^2_{L^2}+\|\na \te_t\|_{L^2}^2\right).\ea\ee
Furthermore, it follows from $(\ref{a1})_1$, \eqref{qq1}, and (\ref{va5}) that
\be \la{s4} \ba
\|\n_{tt}\|_{L^2} &= \|\div(\rho u)_t\|_{L^2} \\
& \le C\left(\|\n_t\|_{L^6}\|\nabla u\|_{L^3}+ \|\nabla u_t\|_{L^2}
+\|u_t\|_{L^6}\|\nabla \n\|_{L^3}+\|\nabla \n_t\|_{L^2}\right) \\ &\le C+C\|\na u_t\|_{L^2}.\ea\ee
Multiplying \eqref{4.052} by $\sigma$ and integrating the resulting inequality over $(0,T)$, one thus  deduces from \eqref{ljq01}, \eqref{lee2}, (\ref{qq1}),  (\ref{va1}), (\ref{sp10}),  (\ref{va5}),     (\ref{s4}),
and Gr\"{o}nwall's inequality that
\be\la{nq11}
\sup\limits_{0\le t\le T} \si \|\nabla u_t\|^2_{L^2}
+ \int_0^T\si\int\rho |u_{tt}|^2dxdt
\le C.
\ee

Finally, it follows from Lemma \ref{le11}, \eqref{s4}, (\ref{va2}), \eqref{nt4}, \eqref{va5}, and  (\ref{nq11}) that
\be\la{sp21}
\sup\limits_{0\le t\le T}\si\left(\|\n_{tt}\|_{L^2}^2+\|u\|^2_{H^3}\right) + \int_0^T \si\|\nabla u_t\|_{H^1}^2 dt\le C,
\ee
which along with \eqref{nq11} gives (\ref{nq1}).
We  complete  the proof of Lemma \ref{pe1}.
\end{proof}

\begin{lemma}\la{pr3} For $q\in (3,6)$ as in Theorem \ref{th1}, it holds that
	\be\la{y2}\ba
	&\sup_{0\le t\le T} \|\n-1\|_{W^{2,q}} +\int_0^T
	 \|\na^2u\|_{W^{1,q}}^{p_0}  dt\le C,
	\ea \ee
	where  \be 1< \la{pppppp} p_0<\frac{4q }{5q-6} \in (1,4/3).\ee
\end{lemma}

\begin{proof}
First, it follows from  \eqref{x2666}, \eqref{x268}, and Lemma \ref{le11} that
\be\la{a4.74}\ba
\|\na^2u\|_{W^{1,q}}
\le & C\left(\|\div u\|_{W^{2,q}}+\|\curl u\|_{W^{2,q}}+\|\na u\|_{L^2}\right)\\
\le & C\left(\|\na G\|_{W^{1,q}}+\|\na\curl u\|_{W^{1,q}}+\|\na P\|_{W^{1,q}}+\|\na u\|_{L^2}+\|\na u\|_{L^q}\right)\\
\le & C\left(\|\n \dot u\|_{W^{1,q}}+\|\na P\|_{W^{1,q}}+\|\n\dot u\|_{L^2}+\|\na u\|_{L^2}+\|\na u\|_{L^q}\right)\\
\le & C
(\|\na\dot u\|_{L^2}+\|\na(\n\dot u)\|_{L^q}+  \|\na^2\te\|_{L^2}+ \|\te\na^2\n\|_{L^q}\\
& +\| |\na\n||\na\te|\|_{L^q}+  \|\na^2\te\|_{L^q}+1)\\
\le & C\left(\|\na\dot u\|_{L^2}+\|\na(\n\dot u)\|_{L^q}  + \| \na^2 \theta\|_{L^q} \right)\\
&+ C(1+ \|\na^2 \theta\|_{L^2})(\| \na^2 \n\|_{L^q} +1).
\ea\ee
Multiplying (\ref{4.52}) by $q|\na^2 \n|^{q-2}\p_{ij} \n$ and  integrating the resulting equality over $\O,$ we obtain after using  (\ref{qq1}) and (\ref{a4.74}) that
\be\la{sp28}\ba
&\frac{d}{dt}\|\na^2\n\|_{L^q}^q\\
&\le C\|\na u\|_{L^\infty}\|\na^2\n\|_{L^q}^q
+C\|\na^{2} u\|_{W^{1,q}}\|\na^2\n\|_{L^q}^{q-1}(\|\na\n\|_{L^q}+1)\\
&\le C\| u\|_{H^3}\|\na^2\n\|_{L^q}^q+C\|\na^2 u\|_{W^{1,q}}\|\na^2\n\|_{L^q}^{q-1}\\
&\le C\left(\| u\|_{H^3}   + \|\na\dot u\|_{L^2} +\|\na(\n\dot u)\|_{L^q}+\|\na^2 \theta\|_{L^q}+1\right)\left(\|\na^2 \n\|_{L^q}^q+1\right).
\ea\ee

Next, it can be deduced from Lemma \ref{le11}, (\ref{g1}),  and \eqref{nq1} that
\be\la{4.49}\ba     \|\na(\n\dot u)\|_{L^q}
&\le C\|\na \n\|_{L^6}\|\dot u \|_{L^{6q/(6-q)}}+C\|\na\dot u \|_{L^q}\\
&\le C\|\na\dot u \|_{L^{6q/(6+q)}}+C\|\na\dot u \|_{L^q}\\
&\le C\|\na\dot u \|_{L^2}+ C\|\na u_t \|_{L^q}+C\|\na(u\cdot \na u ) \|_{L^q}\\
&\le C\|\na\dot u \|_{L^2}+ C\|\na u_t \|_{L^2}^{(6-q)/2q}\|\na u_t \|_{L^6}^{3(q-2)/{2q}}\\
& \quad+C\|\na u \|_{L^6}^{6/q}\| \na u \|_{L^\infty}^{2(q-3)/{q}}
+C\| u \|_{L^\infty}\|\na^2 u \|_{L^q}\\
&\le C\si^{-1/2} \left(\si\|\na u_t \|^2_{H^1}\right)^{3(q-2)/{4q}}
+C\|u\|_{H^3}+C,\ea \ee
and
\begin{equation}\label{n2}
    \begin{aligned}
    \| \na^2 \theta\|_{L^q} \le& C \|\na^2 \theta\|_{L^2}^{(6-q)/2q} \|\na^3 \theta\|_{L^2}^{3(q-2)/2q}\\
    \le & C \si^{-1/2} \left(\si\|\na^3 \theta \|^2_{L^2} \right)^{3(q-2)/{4q}},
\end{aligned}
\end{equation}
which combined with Lemma \ref{le11} and \eqref{nq1} shows that, for $p_0$ as in (\ref{pppppp}),
\be \la{4.53}\int_0^T \left(\|\na(\n \dot u)\|^{p_0}_{L^q} + \|\na^2 \theta\|_{L^q}^{p_0} \right) dt\le C. \ee

Applying  Gr\"{o}nwall's
inequality to (\ref{sp28}), we obtain after using \eqref{lee2},  (\ref{qq1}), and (\ref{4.53}) that
\bnn  \sup\limits_{0\le t\le T}\|\na^2 \n\|_{L^q}\le C,\enn
which combined with   Lemma \ref{le11},    (\ref{4.53}),
and (\ref{a4.74}) gives (\ref{y2}).  We finish  the proof of Lemma \ref{pr3}.
\end{proof}

\begin{lemma}\la{sq90} For $q\in (3,6)$ as in Theorem \ref{th1}, the following estimate holds:
	\be \ba\la{eg17}
	&\sup_{ 0\le t\le T}\si \left(\|\te_t\|_{H^1}+\|\na^2\te\|_{H^1}+\| u_t\|_{H^2}
	+\| u\|_{W^{3,q}}\right)   +\int_0^T   \si^2\|\na u_{tt}\|_{L^2}^2 dt\le C.\ea  \ee
	
\end{lemma}

\begin{proof}
First,  differentiating $(\ref{nt0})$ with
respect to $t$ gives
\be\la{sp30}\ba \begin{cases}
\n u_{ttt}+\n u\cdot\na u_{tt}-(2\mu+\lambda)\nabla{\rm div}u_{tt} +\mu\na \times \curl u_{tt}\\
= 2{\rm div}(\n u)u_{tt} +{\rm div}(\n u)_{t}u_t-2(\n u)_t\cdot\na u_t\\\quad -(\n_{tt} u+2\n_t u_t) \cdot\na u
- \n u_{tt}\cdot\na u-\na P_{tt}, & \text{in}\,\O\times[0,T],\\
u_{tt} \cdot n=0,\quad \curl u_{tt} \times n=0, &  \text{on}\,\p\O\times[0,T],\\
u_{tt}\rightarrow0,\,\,&as\,\,|x|\rightarrow\infty.
\end{cases}\ea \ee
Multiplying (\ref{sp30})$_1$ by $u_{tt}$ and integrating the resulting equality over ${\Omega}$ by parts implies that
\be \la{sp31}\ba
&\frac{1}{2}\frac{d}{dt}\int \n |u_{tt}|^2dx
+\int \left((2\mu+\lambda)({\rm div}u_{tt})^2+\mu|\curl u_{tt}|^2\right)dx\\
&=-4\int  u^i_{tt}\n u\cdot\na u^i_{tt} dx
-\int (\n u)_t\cdot \left(\na (u_t\cdot u_{tt})+2\na u_t\cdot u_{tt}\right)dx\\
&\quad -\int (\n_{tt}u+2\n_tu_t)\cdot\na u\cdot u_{tt}dx
-\int   \n u_{tt}\cdot\na u\cdot  u_{tt} dx\\
& \quad+\int  P_{tt}{\rm div}u_{tt}dx\triangleq\sum_{i=1}^5\tilde{J}_i.\ea\ee
It follows from Lemmas \ref{le11}--\ref{pe1}, (\ref{va1}),  and \eqref{s4} that, for $\eta\in(0,1],$
\be \la{sp32} \ba
|\tilde{J}_1| &\le C\|\n^{1/2}u_{tt}\|_{L^2}\|\na u_{tt}\|_{L^2}\| u \|_{L^\infty} \le \eta \|\na u_{tt}\|_{L^2}^2+C(\eta) \|\n^{1/2}u_{tt}\|^2_{L^2},\ea\ee
\be \la{sp33}\ba
|\tilde{J}_2| &\le C\left(\|\n u_t\|_{L^3}+\|\n_t u\|_{L^3}\right)
\left(\| \na u_{tt}\|_{L^2}\| u_t\|_{L^6}+\| u_{tt}\|_{L^6}\| \na u_t\|_{L^2}\right)\\
&\le C\left(\|\n^{1/2} u_t\|^{1/2}_{L^2}\|u_t\|^{1/2}_{L^6}+\|\n_t\|_{L^6}\| u\|_{L^6}\right)\| \na u_{tt}\|_{L^2}\| \na u_t\|_{L^2}\\
&\le \eta \|\na u_{tt}\|_{L^2}^2+C(\eta)\| \na u_t\|_{L^2}^{3}+C(\eta)\\
&\le \eta \|\na u_{tt}\|_{L^2}^2+C(\eta)\si^{-3/2}  ,\ea\ee
\be  \la{sp34}\ba
|\tilde{J}_3| &\le C\left(\|\n_{tt}\|_{L^2}\|u\|_{L^6}+
\|\n_{t}\|_{L^2}\|u_{t}\|_{L^6} \right)\|\na u\|_{L^6}\|u_{tt}\|_{L^6} \\
&\le \eta \|\na u_{tt}\|_{L^2}^2+C(\eta)\si^{-1}  ,\ea\ee
and\be  \la{sp36}\ba &
|\tilde{J}_4|+|\tilde{J}_5|\\
&\le  C\|\n u_{tt}\|_{L^2} \|\na u\|_{L^3}\|u_{tt}\|_{L^6}
+C \|(\n_t\te+\n\te_t)_t\|_{L^2}\|\na u_{tt}\|_{L^2}\\
&\le  \eta \|\na u_{tt}\|_{L^2}^2+C(\eta) \left(\|\n^{1/2}u_{tt}\|^2_{L^2}
+\|\n_{tt}\te\|_{L^2}^2+\|\n_{t}\te_t\|_{L^2}^2 +\|\n^{1/2}\te_{tt}\|_{L^2}^2\right) \\
&\le  \eta \|\na u_{tt}\|_{L^2}^2+C(\eta)\left( \|\n^{1/2}u_{tt}\|^2_{L^2}
+\|\na\te_t\|_{L^2}^2 +\|\n^{1/2}\te_{tt}\|_{L^2}^2+\sigma^{-2}\right). \ea\ee
Substituting (\ref{sp32})--(\ref{sp36}) into (\ref{sp31}), we obtain after using \eqref{ljq01} and choosing $\eta$ suitably small  that
\be \la{ex12}\ba
& \frac{d}{dt}\int \n |u_{tt}|^2dx+C_4 \int|\na u_{tt}|^2dx \\
& \le  C\si^{-2} +C\|\n^{1/2}u_{tt}\|^2_{L^2}
+C\|\na\te_t\|_{L^2}^2+C_5\|\n^{1/2}\te_{tt}\|_{L^2}^2.\ea\ee

Then, differentiating \eqref{3.29} with respect to $t$ infers
\be\la{eg1}\ba \begin{cases}
-\frac{\ka(\ga-1)}{R}\Delta \te_t+\n\te_{tt}\\
=-\n_t\te_{t}- \n_t\left(u\cdot\na \te+(\ga-1)\te\div u\right)-\n\left( u\cdot\na
\te+(\ga-1)\te\div u\right)_t\\
\quad+\frac{\ga-1}{R}\left(\lambda (\div u)^2+2\mu |\mathfrak{D}(u)|^2\right)_t,  \qquad \qquad \qquad \qquad \qquad \quad \text{in}\,\O\times[0,T],\\
\na \te_t\cdot n=0,  \qquad \qquad \qquad \qquad\qquad\qquad\qquad \qquad\qquad\ \ \ \ \ \ \  \text{on}\,\p\O\times[0,T],\\
\nabla\te_t\rightarrow0,\qquad \qquad \qquad \qquad\qquad\qquad\qquad \qquad\qquad\qquad\ \ \ \ \   \text{as}\,|x|\rightarrow\infty.
\end{cases}\ea\ee
Multiplying  (\ref{eg1})$_1$ by $\te_{tt}$ and integrating the resulting
equality over $\Omega$ lead to
\be\la{ex5}\ba
& \left(\frac{\ka(\ga-1)}{2R}\|\na \te_t\|_{L^2}^2+H_0\right)_t+ \int\n\te_{tt}^2dx \\
&=\frac{1}{2}\int\n_{tt}\left( \te_t^2
+2\left(u\cdot\na \te+(\ga-1)\te\div u\right)\te_t\right)dx\\
&\quad + \int\n_t\left(u\cdot\na\te+(\ga-1)\te\div u \right)_t\te_{t}dx\\
& \quad-\int\n\left(u\cdot\na\te+(\ga-1)\te\div u\right)_t\te_{tt}dx\\
& \quad -\frac{\ga-1}{R}\int \left(\lambda (\div u)^2+2\mu |\mathfrak{D}(u)|^2\right)_{tt}\te_t dx
\triangleq\sum_{i=1}^4H_i,\ea\ee
where
\bnn\ba H_0\triangleq & \frac{1}{2}\int \n_t\te_{t}^2dx
+\int\n_t\left(u\cdot\na\te+(\ga-1)\te\div u\right) \te_tdx\\
&- \frac{\ga-1}{R}\int\left(\lambda (\div u)^2+2\mu |\mathfrak{D}(u)|^2 \right)_t\te_t dx. \ea\enn
It's easy to deduce from  $(\ref{a1})_1,$  (\ref{va5}), (\ref{va1}),  (\ref{nq1}),  \eqref{s4}, and Lemma \ref{le11} that
\be\la{ex6}\ba
|H_0|\le & C\int \n|u||\te_{t}||\na\te_{t}|dx+ C\|\na u\|_{L^3}\|\na u_t\|_{L^2} \|\te_t\|_{L^6} \\
&+C\|\n_t\|_{L^3}\|\te_t\|_{L^6}\left( \|\na\te\|_{L^2} \|u\|_{L^\infty}+ \|\na u\|_{L^2}+ \|\theta-1\|_{L^6}\|\na u\|_{L^3}\right)\\
\le &C  (\|\rho^{1/2}\theta_t \|_{L^2}+\|\na\te_t\|_{L^2})\left(\|\n\te_t\|_{L^2}+\|\na u_t\|_{L^2}+1\right)\\
\le &\frac{\ka(\ga-1)}{4R} \|\na\te_t\|_{L^2}^2+C\si^{-1},\ea\ee
and
\be\la{ex7}\ba
|H_1|& \le C\|\n_{tt}\|_{L^2}\left(\|\te_t\|_{L^4}^{2}
+\|\te_t\|_{L^6}\left(\|u\cdot\na \te\|_{L^3}+\|\na u\|_{L^3}+\|\te-1\|_{L^6}\|\na u\|_{L^6} \right)\right)\\
&\le C\|\n_{tt}\|_{L^2}\left(\|\rho^{1/2} \te_t\|_{L^2}^{2} + \|\na \te_t\|_{L^2}^{2}
+\si^{-1/2}  \right) \\
& \le  C(1+\|\na u_{t}\|_{L^2} )\|\na \te_t\|^2_{L^2}+C\si^{-3/2}.\ea\ee
We deduce from Lemma \ref{le11} that
\be\la{eg12}\ba
 &\|\left(u\cdot\na\te+(\ga-1)\te\div u \right)_t\|_{L^2}\\
& \le  C\left(\|u_t\|_{L^6}\|\na\te\|_{L^3}+\|\na\te_t\|_{L^2}+\|\te_t\|_{L^6}\|\na u\|_{L^3}
+\|\te\|_{L^\infty}\|\na u_t\|_{L^2}\right)\\
& \le  C\|\na u_t\|_{L^2}(\|\na^2 \te\|_{L^2}+1)+ C\|\na \te_t\|_{L^2}+C\|\rho^{1/2} \theta_t \|_{L^2},\ea\ee
which together with \eqref{lee2}, \eqref{va5},  and (\ref{va1}) shows
\be\la{ex9}\ba
|H_2|+|H_3|&\le C\left(\si^{-1/2}(\|\na u_t\|_{L^2}+1)+\|\na\te_t\|_{L^2}\right)
\left(\|\n_t\|_{L^3} \|\te_t\|_{L^6}+\|\n \te_{tt}\|_{L^2}\right)\\
&\le \frac{1}{2}\int\n\te_{tt}^2dx+C\|\na\te_t\|_{L^2}^2+C\si^{-1 } \|\na u_t\|^2_{L^2}+C\si^{-1 } . \ea\ee
One deduces from \eqref{qq1}, (\ref{va1}), and (\ref{nq1}) that
\be\la{ex10}\ba
|H_4|&\le C\int \left(|\na u_t|^2+|\na u||\na u_{tt}|\right)|\te_t|dx\\
&\le C\left(\|\na u_t\|_{L^2}^{3/2}\|\na u_t\|_{L^6}^{1/2}
+ \|\na u\|_{L^3} \|\na u_{tt}\|_{L^2}\right)\|\te_t\|_{L^6}\\
&\le \de\|\na u_{tt}\|^2_{L^2}+C\|\na^2 u_t\|^2_{L^2}
+C(\de)(\|\na\te_t\|_{L^2}^2+\si^{-1 })+C\si^{-2}\|\na u_t\|_{L^2}^2.\ea\ee
Substituting (\ref{ex7}), (\ref{ex9}), and (\ref{ex10}) into (\ref{ex5}) gives
\be\la{ex11}\ba
& \left(\frac{\ka(\ga-1)}{2R}\|\na \te_t\|_{L^2}^2+H_0\right)_t
+\frac{1}{2}\int\n\te_{tt}^2dx \\
&\le  \de\|\na u_{tt}\|^2_{L^2}+C(\de)((1+\|\na u_{t}\|_{L^2}) \|\na \te_t\|^2_{L^2}+\si^{-3/2})\\
&\quad +C(\|\na^2 u_t\|^2_{L^2} +\si^{-2}\|\na u_t\|_{L^2}^2).\ea\ee

Finally, for $C_5$ as in (\ref{ex12}), adding (\ref{ex11}) multiplied by
$2 (C_5+1) $ to  (\ref{ex12}),
we obtain after choosing $\de$ suitably small that
\be\la{ex13}\ba
& \left[ 2 (C_5+1)\left(\frac{\ka(\ga-1)}{2R}\|\na \te_t\|_{L^2}^2+H_0\right)
+\int \n |u_{tt}|^2dx\right]_t\\
&\quad + \int\n\te_{tt}^2dx+\frac{C_4}{2}\int |\na u_{tt}|^2dx\\
&\le C (1+\|\na u_{t}\|_{L^2}^2) (\si^{-2} +\|\na \te_t\|^2_{L^2})
+C\|\n^{1/2}u_{tt}\|^2_{L^2} + C\|\na^2 u_t\|^2_{L^2}.\ea\ee
Multiplying (\ref{ex13}) by $\si^2$ and integrating the resulting inequality over $(0,T),$
we  obtain after using (\ref{ex6}),  (\ref{nq1}),  (\ref{va5}), and Gr\"{o}nwall's inequality that
\be \la{eg10}
\sup_{ 0\le t\le T}\si^2\int \left(|\na\te_t|^2+\n |u_{tt}|^2\right)dx
+\int_{0}^T\si^2\int \left(\n\te_{tt}^2+|\nabla u_{tt}|^2\right)dxdt\le C,\ee
which together with Lemmas \ref{le11}, \ref{pe1}, \ref{pr3}, (\ref{nt4}), (\ref{ex4}),  \eqref{va1}, (\ref{a4.74}),
and (\ref{4.49}) gives
\be\la{sp20} \sup_{ 0\le t\le T}\si \left(\|\na u_t\|_{H^1}
+ \|\na^2\te\|_{H^1}+\|\na^2u\|_{W^{1,q}} \right)\le C.\ee

We thus derive (\ref{eg17}) from  (\ref{eg10}),   (\ref{sp20}), \eqref{va1}, \eqref{eee8},
and (\ref{qq1}). The proof of Lemma \ref{sq90} is completed.
\end{proof}

\begin{lemma}\la{sq91} The following estimate holds:
	\be \la{egg17}\sup_{ 0\le t\le T}\si^2  \left(\|\na^2\te\|_{H^2}+\| \te_t\|_{H^2}+\|\n^{1/2}\te_{tt}\|_{L^2}  \right)
	+\int_0^T\si^4\|\na \te_{tt}\|_{L^2}^2 dt\le C.\ee
\end{lemma}

\begin{proof}
First, differentiating $(\ref{eg1})$ with respect to $t$ yields
\be\la{eg2}\ba \begin{cases}
\n\te_{ttt}-\frac{\ka(\ga-1)}{R}\Delta \te_{tt}\\
=-\n u
\cdot\na\te_{tt}+ 2\div(\n u)\te_{tt} - \n_{tt}\left(\te_t+ u\cdot\na \te+(\ga-1)\te\div u\right)\\
\quad - 2\n_t\left(u\cdot\na \te+(\ga-1)\te\div u\right)_t\\
\quad - \n\left(u_{tt}\cdot\na \te+2u_t\cdot\na\te_t+(\ga-1)(\te\div u)_{tt}\right)\\
\quad +\frac{\ga-1}{R}\left(\lambda (\div u)^2+2\mu |\mathfrak{D}(u)|^2 \right)_{tt},\qquad\qquad\qquad \qquad\qquad\quad   &\text{in}\,\O\times[0,T],\\
\na\te_{tt}\cdot n=0,  \qquad \qquad \qquad \qquad\qquad\qquad\qquad \qquad &\text{on}\,\p\O\times[0,T],\\
\na\te_{tt}\rightarrow 0\qquad\qquad\qquad\qquad\qquad\qquad\qquad\qquad\qquad\qquad\qquad\ \ \ \ \   &\text{as}\,|x|\rightarrow\infty.
\end{cases}\ea\ee
Multiplying (\ref{eg2})$_1$ by $\te_{tt}$ and integrating the resulting equality over $\Omega$ yield that
\be\la{eg3}\ba
&\frac{1}{2}\frac{d}{dt}\int\n|\te_{tt}|^2dx +\frac{\ka(\ga-1)}{R}\int|\na \te_{tt}|^2dx\\
&=-4\int \te_{tt}\n u\cdot\na\te_{tt}dx  -\int \n_{tt}\left(\te_t
+ u\cdot\na \te+(\ga-1)\te\div u\right)\te_{tt}dx\\
&\quad - 2\int\n_t\left(u\cdot\na \te+(\ga-1)\te\div u\right)_t\te_{tt}dx\\
& \quad - \int\n\left(u_{tt}\cdot\na \te+2u_t\cdot\na\te_t
+(\ga-1)(\te\div u)_{tt}\right)\te_{tt}dx\\
& \quad +\frac{\ga-1}{R}\int  \left(\lambda (\div u)^2+2\mu |\mathfrak{D}(u)|^2\right)_{tt}\te_{tt}dx
\triangleq \sum_{i=1}^5K_i.\ea\ee

It follows from
Lemmas \ref{le11}--\ref{pe1}, \ref{sq90},  \eqref{eee8}, (\ref{eg10}), and (\ref{eg17})  that
 \be \la{eg4} \ba
\si^4|K_1|&\le C\si^4\|\n^{1/2}\te_{tt}\|_{L^2}\|\na \te_{tt}\|_{L^2}\|u\|_{L^\infty}\\
&\le \de \si^4\|\na \te_{tt}\|_{L^2}^2+C(\de) \si^4\|\n^{1/2}\te_{tt}\|^2_{L^2} ,\ea\ee
\be \la{eg16}\ba
\si^4|K_2|&\le C \si^4\|\n_{tt}\|_{L^2}\|\te_{tt}\|_{L^6} \left( \|\te_t\|_{H^1}+\|\na\te\|_{L^3}
+\|\na u\|_{L^3}\right) \\
&\le C\si^2 (\|\na \te_{tt}\|_{L^2}+\|\n^{1/2} \te_{tt}\|_{L^2})\\
&\le \de\si^4\|\na \te_{tt}\|_{L^2}^2+C(\de)(\si^4\|\n^{1/2} \te_{tt}\|_{L^2}^2+1),\ea\ee
\be \la{eg7}\ba
\si^4|K_4|&\le C\si^4\|\te_{tt}\|_{L^6}
\left( \|\na\te\|_{L^3}\|\n u_{tt}\|_{L^2}
+\|\na\te_t\|_{L^2}\|\na u_t\|_{L^2}\right)\\
&\quad+ C\si^4\|\te_{tt}\|_{L^6} \left( \|\na u\|_{L^3}\|\n \te_{tt}\|_{L^2}
+ \|\na u_t\|_{L^2}\| \te_t\|_{L^3}\right)\\
&\quad+C\si^4\|\te\|_{L^\infty}\|\n\te_{tt}\|_{L^2} \|\na u_{tt}\|_{L^2} \\
&\le \de\si^4\|\na \te_{tt}\|_{L^2}^2
+C(\de)\left(\si^4 \|\n^{1/2} \te_{tt}\|_{L^2}^2+\si^3\|\na u_{tt}\|_{L^2}^2  \right)+C(\de),\ea\ee
\be \la{eg8}\ba
\si^4|K_5|&\le C\si^4\|\te_{tt}\|_{L^6}
\left( \|\na u_t\|_{L^2}^{3/2}\|\na u_{t}\|_{L^6}^{1/2}+\|\na u\|_{L^3}\|\na u_{tt}\|_{L^2}\right)  \\
&\le \de\si^4\|\na \te_{tt}\|_{L^2}^2
+C(\de)\si^4\left(\|\n^{1/2} \te_{tt}\|_{L^2}^2+\|\na u_{tt}\|_{L^2}^2  \right) +C(\de),\ea\ee
and
\be \la{eg6}\ba
\si^4|K_3|&\le C \si^4\|\n_t\|_{L^3} \|\te_{tt}\|_{L^6}
\left( \si^{-1/2}\|\na u_t\|_{L^2} +\|\rho^{1/2} \theta_t\|_{L^2} +\|\na\te_{t}\|_{L^2} \right) \\
&\le \de\si^4\|\na \te_{tt}\|_{L^2}^2+ C(\de) \si^4 \|\n^{1/2} \te_{tt}\|_{L^2}^2 +C(\de),\ea\ee
where in the last inequality we have used (\ref{eg12}).

Then, multiplying (\ref{eg3})  by $\si^4,$ substituting (\ref{eg4})--(\ref{eg6}) into the resulting equality and choosing $\de$ suitably small, one obtains
\bnn \ba
& \frac{d}{dt}\int\si^4\n|\te_{tt}|^2dx +\frac{\ka(\ga-1)}{R}\int\si^4|\na \te_{tt}|^2dx\\
& \le  C\si^2\left(\|\n^{1/2} \te_{tt}\|_{L^2}^2
+\|\na u_{tt}\|_{L^2}^2  \right)+C,\ea\enn
which together with (\ref{eg10})   gives
\be\la{eg13} \sup_{ 0\le t\le T}\si^4\int  \n |\te_{tt}|^2dx
+\int_{0}^T\si^4\int_{ } |\nabla \te_{tt}|^2 dxdt\le C.\ee

Finally, applying the standard $L^2$-estimate  to (\ref{eg1}), one obtains after using Lemmas \ref{le11}--\ref{pe1}, (\ref{va1}), (\ref{va5}), (\ref{eg13}),
and (\ref{eg17}) that
\be\la{eg14}\ba
&\sup_{0\le t\le T}\si^2\|\na^2\te_t\|_{L^2}\\
&\le  C\sup_{0\le t\le T}\si^2\left(\|\n\te_{tt}\|_{L^2}
+  \|\n_t\|_{L^3}\|\te_t\|_{L^6}+\|\n_t\|_{L^6} \left(\|\na\te\|_{L^3}+\|\na u\|_{L^3}\right)\right)\\
& \quad +C\sup_{0\le t\le T}\si^2\left(\|\n^{1/2} \te_t\|_{L^2}+\|\na\te_t\|_{L^2}+ (1+\|\na^2\te\|_{L^2}) \|\na u_t\|_{L^2}+
 \|\na u_t\|_{L^6}\right)\\
& \le  C.\ea\ee
Moreover, it follows from  the standard $H^2$-estimate  of
$(\ref{3.29})$, (\ref{hs}), \eqref{nq1}, and Lemma \ref{le11}   that
\bnn\ba \|\na^2\te\|_{H^2}
&\le C\left(\|\n\te_t\|_{H^2}+\|\n u\cdot\na\te\|_{H^2}
+\|\n\te\div u\|_{H^2}+\||\na u|^2\|_{H^2} \right)\\
&\le C\left((1+\|\n-1\|_{H^2} )\|\te_t\|_{H^2}
+(\|\n-1\|_{H^2}+1) \| u\|_{H^2}\|\na\te\|_{H^2}\right)\\
&\quad+C(1+\|\n-1\|_{H^2})(1+\|\te-1\|_{H^2})  \| \div u\|_{H^2}+C\|\na u\|^2_{H^2}\\
&\le C\si^{-1}+ C\| \na^3\te \|_{L^2}+C\| \te_t\|_{H^2}.\ea\enn
Combining this with
  (\ref{eg17}),  (\ref{eg14}), and  (\ref{eg13}) shows (\ref{egg17}).
The proof of Lemma \ref{sq91} is completed.
\end{proof}

\section{\la{se5}Proof of  Theorem  \ref{th1}}

With all the a priori estimates in Sections \ref{se3} and \ref{se4}
at hand, we are ready to prove the main result  of this paper in
this section.

\begin{pro} \la{pro2}

 For  given numbers $M>0$ (not necessarily small),
  $\on> 2,$ and $\bt>1,$   assume that  $(\rho_0,u_0,\te_0)$ satisfies (\ref{2.1}),  (\ref{3.1}),
and   (\ref{z01}). Then    there exists a unique classical solution  $(\rho,u,\te) $      of problem (\ref{a1})--(\ref{ch2})
 in $\Omega\times (0,\infty)$ satisfying (\ref{mn5})--(\ref{mn2}) with $T_0$ replaced by any $T\in (0,\infty).$
  Moreover,  (\ref{zs2}), (\ref{a2.112}), (\ref{ae3.7}), and  (\ref{vu15})  hold for any $T\in (0,\infty).$ 
 \end{pro}

\begin{proof}
First, by the standard local existence result (Lemma \ref{th0}), there exists a $T_0>0$ which may depend on
$\inf\limits_{x\in \Omega}\n_0(x), $  such that the  problem
 (\ref{a1})--(\ref{ch2})  with   initial data $(\n_0 ,u_0,\te_0 )$
 has   a unique classical solution $(\n,u,\te)$ on $\O\times(0,T_0],$  which satisfies (\ref{mn6})--(\ref{mn2}).
It follows from (\ref{As1})--(\ref{3.1}) that
\bnn A_1(0)\le M^2,\quad  A_2(0)\le  C_0^{1/4},\quad A_3(0)=0, \quad  \n_0<
 \hat{\rho},\quad \te_0\le \bt.\enn  Then  there exists a
$T_1\in(0,T_0]$ such that (\ref{z1}) holds for $T=T_1.$
 We set \bnn \notag T^* =\sup\left\{T\,\left|\, \sup_{t\in [0,T]}\|(\n-1,u,\te-1)\|_{H^3}<\infty\right\},\right.\enn  and \be \la{s1}T_*=\sup\{T\le T^* \,|\,{\rm (\ref{z1}) \
holds}\}.\ee Then $ T^*\ge T_* \geq T_1>0.$
 Next, we claim that
 \be \la{s2}  T_*=\infty.\ee  Otherwise,    $T_*<\infty.$
  Therefore, by Proposition \ref{pr1},  (\ref{zs2})
  holds for all $0<T<T_*,$  which combines with (\ref{z01}) yields
    Lemmas \ref{le11}--\ref{sq91} still hold for all  $0< T< T_* .$
Note here that all  constants $C$  in  Lemmas \ref{le11}--\ref{sq91}
depend  on $T_*  $ and $\inf\limits_{x\in \Omega}\n_0(x)$, are in fact independent  of  $T$.
Then,  we claim that  there
exists a positive constant $\tilde{C}$ which may  depend  on $T_* $
and $\inf\limits_{x\in \Omega}\n_0(x)$   such that, for all  $0< T<
 T_*,$  \be\la{y12}\ba \sup_{0\le t\le T}
\| \n-1\|_{H^3}   \le \tilde{C},\ea \ee which together with Lemmas \ref{le11}, \ref{pe1},  \ref{sq90},  \eqref{mn2}, and (\ref{3.1}) gives
 \be\notag
 \|(\n(x,T_*)-1,u(x,T_*),\te(x,T_*)-1)\|_{H^3}
 \le \tilde{C},\quad\inf_{x\in \Omega}\n(x,T_*)>0,\quad\inf_{x\in \Omega}\te(x,T_*)>0.\ee
  Thus, Lemma \ref{th0} implies that there exists some $T^{**}>T_*,$  such that
(\ref{z1}) holds for $T=T^{**},$   which contradicts (\ref{s1}).
Hence, (\ref{s2}) holds. This along with Lemmas \ref{th0}, \ref{a13.1}, \ref{le8}, and Proposition \ref{pr1}, thus
finishes the proof of   Proposition \ref{pro2}.

 Finally, it remains to prove (\ref{y12}). Using  $(\ref{a1})_3$ and (\ref{mn6}), we can define
 \bnn
 \theta_t(\cdot,0)\triangleq - u_0 \cdot\na \te_0 + \frac{\ga-1}{R} \rho_0^{-1}
\left(\ka\Delta\te_0-R\rho_0 \theta_0 \div u_0+\lambda (\div u_0)^2+2\mu |\mathfrak{D}(u_0)|^2\right),
 \enn
 which along with  (\ref{2.1}) gives
 \be \la{ssp91}\|\theta_t(\cdot,0)\|_{L^2}\le \tilde{C}.\ee
 Thus, one deduces from  (\ref{3.99}) and Lemma \ref{le11} that
\be  \ba \la{a51}
\sup\limits_{0\le t\le T} \int \n|\dot\te|^2dx+\int_0^T \|\na\dot\te\|_{L^2}^2dt \le \tilde{C},
\ea\ee
which together with  (\ref{lop4})  yields
\be\la{sp211}
\sup\limits_{0\le t\le T}\|\na^2 \theta\|_{L^2} \le \tilde{C}.\ee
Using  $(\ref{a1})_2$ and  (\ref{mn6}), we can define
\bnn
u_t(\cdot,0) \triangleq -u_0\cdot\na u_0+\n_0^{-1}\left( \mu \Delta u_0 + (\mu+\lambda) \na \div u_0 - R\na (\n_0\te_0)\right),
\enn
which along with  (\ref{2.1}) gives
\be \la{ssp9}\|\na u_t(\cdot,0)\|_{L^2}\le \tilde{C}.\ee
Thus, one deduces from Lemmas \ref{le11}, \ref{le9-1},  (\ref{4.052}),  (\ref{s4}), (\ref{a51})--(\ref{ssp9}), and Gr\"{o}nwall's inequality that
\be \la{ssp1}
\sup_{0\le t\le T}\|\na u_t\|_{L^2}+\int_0^T\int \n |u_{tt}|^2dxdt\le \tilde{C},
\ee
which combined with  (\ref{va2}), \eqref{sp211}, and (\ref{qq1}) yields
\be\la{sp221} \sup\limits_{0\le t\le T}\|u\|_{H^3} \le \tilde{C}.\ee
Combining this with Lemma \ref{le11}, (\ref{a51}), (\ref{sp211}), (\ref{ssp1}),   and  (\ref{sp221}) gives
\be\la{ssp24} \ia\left(\|\na^3\te\|_{L^2}^2+ \|\nabla u_t\|_{H^1}^2\right)dt\le \tilde{C} . \ee
Then, it can be deduced  from (\ref{uwkq}),\eqref{a4.74}, \eqref{hs}, \eqref{sp211}, \eqref{ssp1}, \eqref{sp221}, and Lemma \ref{le11}  that
\be \notag\ba \|\na^2 u\|_{H^2}
&\le\tilde{C}\left(\| \div u \|_{H^3}+\|\curl u\|_{H^3} + \|\na u\|_{L^2}\right)\\
&\le\tilde{C}\left(\| \n\dot u \|_{H^2}+\|\na P\|_{H^2} + 1\right)\\
&\le \tilde{C}(1+\| \n-1\|_{H^2})(\|  u_t \|_{H^2}+\|  u \|_{H^2}\|  \na u \|_{H^2})+\tilde{C}\\
&\quad+\tilde C(1+\|\n-1\|_{H^2}+\|\te-1 \|_{H^2})(\|\na\n\|_{H^2}+\|\na\te\|_{H^2})\\
& \le \tilde{C} (1+ \|\na^2  u_t\|_{L^2}+\|\na^3 \n \|_{L^2}+\|\na^3 \te \|_{L^2}),\ea\ee
which along with some standard calculations leads to
 \bnn\la{sp134}\ba  & \left(\|\na^3 \n\|_{L^2} \right)_t \\
&\le \tilde{C}\left(\| |\na^3u| |\na \n| \|_{L^2}+ \||\na^2u||\na^2
      \n|\|_{L^2}+ \||\na u||\na^3 \n|\|_{L^2} +\| \na^4u \|_{L^2} \right)\\
&\le \tilde{C}\left(\| \na^3 u\|_{L^2}\|\na \n \|_{H^2}+ \| \na^2u\|_{L^3}\|\na^2 \n \|_{L^6}
       +\|\na u\|_{L^\infty}\|\na^3 \n\|_{L^2} +\| \na^4u \|_{L^2}\right) \\
&\le \tilde{C}(1+ \| \na^3\n \|_{L^2}+ \| \na^2 u_t\|^2_{L^2}+ \|\na^3\te\|^2_{L^2}), \ea\enn where we have used (\ref{sp221}) and Lemma \ref{le11}.
 Combining this with (\ref{ssp24})  and
Gr\"{o}nwall's inequality  yields   \bnn\la{sp26} \sup\limits_{0\le t\le
T}\|\nabla^3  \n\|_{L^2} \le \tilde{C},\enn which together with
(\ref{qq1}) gives (\ref{y12}).
The proof of Proposition \ref{pro2} is completed.
\end{proof}

With  Proposition \ref{pro2} at hand, we are now in a position to prove  Theorem \ref{th1}.

\begin{proof}[Proof of  Theorem   \ref{th1}]
 Let $(\n_0,u_0,\te_0)$  satisfying (\ref{co3})--(\ref{co2}) be the initial data in Theorem \ref{th1}.  Assume that  $C_0$  satisfies (\ref{co14}) with \be\la{xia}\ve\triangleq \ve_0/2,\ee
  where  $\ve_0$  is given in Proposition \ref{pr1}.

First, we construct the approximate initial data $(\n_0^{m,\eta},u_0^{m,\eta}, \te_0^{m,\eta})$ as follows. For constants
\be\la{5d0}
m \in \mathbb{Z}^+,\ \  \eta \in \left(0, \eta_0 \right),\ \  \eta_0\triangleq \min\xl\{1,\frac{1}{2}(\on-\sup\limits_{x\in \O}\n_0(x)) \xr\},
\ee
we define
\begin{align*}
\n_0^{m,\eta} = \n_0^{m}+ \eta,\ \  u_0^{m,\eta}=\frac{u_0^m }{1+\eta},\ \  \te_0^{m,\eta}= \frac{\te_0^{m} + \eta}{1+2\eta},
\end{align*}
where $\n_0^{m}$ satisfies
\begin{align*}
0 \le \n_0^{m} \in C^{\infty},\ \  \lim_{m \to \infty} \|\n_0^{m} -\rho_0\|_{H^2 \cap W^{2,q}}=0,
\end{align*}
and $u_0^m$ is the unique smooth solution to the following elliptic equation:
\be\notag\begin{cases}
	\Delta u_0^m=\Delta \tilde{u}_0^m,&\text{in}\,\, \O,\\
	u_0^m\cdot n=0 ,\,\,\curl u_0^m\times n=0,&\text{on}\,\,\p\O,\\
	u_0^m\rightarrow0,\,\,&\text{as}\,\,|x|\rightarrow\infty,
\end{cases}\ee
with $\tilde{u}_0^m \in C^{\infty}$ satisfying $\lim_{m \to \infty}\| \tilde{u}_0^m -{u}_0\|_{H^2}=0$, and $\te_0^m$ is the unique smooth solution to the following Poisson equation:
\be\notag\begin{cases}
	\Delta \te_0^m=\Delta \tilde{\te}_0^m,&\text{in}\,\,\O,\\
	\na \te_0^m\cdot n=0 ,&\text{on}\,\,\p\O,\\
	 \te_0^m\rightarrow1,\,\,&\text{as}\,\,|x|\rightarrow\infty,
\end{cases}\ee
with $\tilde{\te}_0^m=\tilde\te_0\ast j_{m^{-1}}$, $\tilde\te_0$ is the nonnegative $H^1$-extension of $\te_0$,  and  $j_{m^{-1}}(x)$ is the standard mollifying kernel of width $m^{-1}$.

Then for any $\eta\in (0, \eta_0)$, there exists $m_1(\eta)\ge 0$ such that for $m \ge m_1(\eta)$, the approximate initial data
$(\n_0^{m,\eta},u_0^{m,\eta}, \te_0^{m,\eta})$ satisfies
\be \la{de3}\begin{cases}(\n_0^{m,\eta}-1,u_0^{m,\eta}, \te_0^{m,\eta}-1)\in C^\infty ,\\
	\dis \eta\le  \n_0^{m,\eta}  <\hat\n,~~\, \frac{\eta}{4}\le \te_0^{m,\eta} \le \hat \te,~~\,\|\na u_0^{m,\eta}\|_{L^2} \le M, \\
	\dis u_0^{m,\eta}\cdot n=0,~~\,\curl u_0^{m,\eta}\times n=0,~~\,\na \te_0^{m,\eta}\cdot n=0,\,\, \text{on}\,\p\O,\end{cases}
\ee
and
 \be \la{de03}
\lim\limits_{\eta\rightarrow 0} \lim\limits_{m\rightarrow \infty}
\left(\| \n_0^{m,\eta} - \n_0 \|_ {H^2 \cap W^{2,q}}+\| u_0^{m,\eta}-u_0\|_{H^2}+\| \te_0^{m,\eta}- \te_0  \|_{H^1}\right)=0.
\ee
Moreover,  the initial norm $C_0^{m,\eta}$
for $(\n_0^{m,\eta},u_0^{m,\eta}, \te_0^{m,\eta}),$ which is defined by  the right-hand side of (\ref{e})
with $(\n_0,u_0,\te_0)$   replaced by
$(\n_0^{m,\eta},u_0^{m,\eta}, \te_0^{m,\eta}),$
satisfies \bnn \lim\limits_{\eta\rightarrow 0} \lim\limits_{m\rightarrow \infty} C_0^{m,\eta}=C_0.\enn
Therefore, there exists  an  $\eta_1\in(0, \eta_0) $
such that, for any $\eta\in(0,\eta_1),$ we can find some $m_2(\eta)\geq m_1(\eta)$  such that   \be \la{de1} C_0^{m,\eta}\le C_0+\ve_0/2\le  \ve_0 , \ee
provided that\be  \la{de7}0<\eta<\eta_1 ,\,\, m\geq m_2(\eta).\ee

  We assume that $m,\eta$ satisfy (\ref{de7}).
  Proposition \ref{pro2} together with (\ref{de1}) and (\ref{de3}) thus yields that
   there exists a smooth solution  $(\n^{m,\eta},u^{m,\eta}, \te^{m,\eta}) $
   of problem (\ref{a1})--(\ref{ch2}) with  initial data $(\n_0^{m,\eta},u_0^{m,\eta}, \te_0^{m,\eta})$
   on $\Omega\times (0,T] $ for all $T>0. $
    Moreover, one has (\ref{h8}), \eqref{zs2}, (\ref{a2.112}), \eqref{ae3.7}, and (\ref{vu15})   with $(\n,u,\te)$  being replaced by $(\n^{m,\eta},u^{m,\eta}, \te^{m,\eta}).$

 Next, for the initial data $(\n_0^{m,\eta},u_0^{m,\eta}, \te_0^{m,\eta})$, the function $\tilde g$ in (\ref{co12})  is
 \be \la{co5}\ba \tilde g & \triangleq(\n_0^{m,\eta})^{-1/2}\left(-\mu \Delta u_0^{m,\eta}-(\mu+\lambda)\na\div
 u_0^{m,\eta}+R\na (\n_0^{m,\eta}\te^{m,\eta}_0)\right)\\
& = (\n_0^{m,\eta})^{-1/2}\sqrt{\n_0}g+\mu(\n_0^{m,\eta})^{-1/2}\Delta(u_0-u_0^{m,\eta})\\
&\quad+(\mu+\lambda) (\n_0^{m,\eta})^{-1/2} \na \div(u_0-u_0^{m,\eta})+ R(\n_0^{m,\eta})^{-1/2} \na(\n_0^{m,\eta}\te_0^{m,\eta}-\n_0\te_0),\ea\ee
where in the second equality we have used (\ref{co2}).
Since $g \in L^2,$ one deduces from (\ref{co5}),  (\ref{de3}), (\ref{de03}), and  (\ref{co3})  that for any $\eta\in(0,\eta_1),$ there exist some $m_3(\eta)\geq m_2(\eta)$ and a positive constant $C$ independent of $m$ and $\eta$ such that
 \be\la{de4}
 	\|\tilde g\|_{L^2}\le \|g\|_{L^2}+C\eta^{-1/2}\de(m) + C\eta^{1/2},
 	\ee
with   $0\le\de(m) \rightarrow 0$ as
$m \rightarrow \infty.$ Hence,  for any  $\eta\in(0,\eta_1),$ there exists some $m_4(\eta)\geq m_3(\eta)$ such that for any $ m\geq m_4(\eta)$,
 \be \la{de9}\de(m) <\eta.\ee  We thus obtain from (\ref{de4}) and (\ref{de9}) that
there exists some positive constant $C$ independent of $m$ and $\eta$ such that  \be\la{de14}  \|\tilde g \|_{L^2}\le \|g \|_{L^2}+C,\ee provided that\be \la{de10} 0<\eta<\eta_1,\,\,  m\geq m_4(\eta).\ee

 Now, we   assume that $m,$  $\eta$ satisfy (\ref{de10}).
 It thus follows from (\ref{de03}), (\ref{de1}),  (\ref{de14}), Proposition \ref{pr1}, 
 and Lemmas \ref{le8}, \ref{le11}--\ref{sq91} that for any $T>0,$
 there exists some positive constant $C$ independent of $m$ and $\eta$ such that
 (\ref{h8}), (\ref{zs2}),   (\ref{a2.112}),  \eqref{ae3.7}, (\ref{vu15}), \eqref{lee2}, \eqref{qq1},  (\ref{va5}),  (\ref{vva5}), (\ref{y2}),  (\ref{eg17}),
  and  (\ref{egg17})  hold for  $(\n^{m,\eta},u^{m,\eta}, \te^{m,\eta}) .$
   Then passing  to the limit first $m\rightarrow \infty,$ then $\eta\rightarrow 0,$
   together with standard arguments yields that there exists a solution $(\n,u,\te)$ of the problem (\ref{a1})--(\ref{ch2})
   on $\Omega\times (0,T]$ for all $T>0$, such that the solution  $(\n,u,\te)$
   satisfies  (\ref{h8}), (\ref{a2.112}),    \eqref{ae3.7},  (\ref{vu15}),  \eqref{lee2}, \eqref{qq1}, (\ref{va5}),  (\ref{vva5}), (\ref{y2}),  (\ref{eg17}), (\ref{egg17}),
   and  the estimates of $A_i(T)\,(i=1,2,3)$ in
   (\ref{zs2}). Hence,    $(\n,u,\te)$ satisfies  (\ref{h8}) and \eqref{h9}.

Finally,  the proof of the uniqueness of $(\n,u,\te)$ is similar to that  of \cite[Theorem 1]{choe1} and will be omitted here for simplicity. To finish the proof of Theorem \ref{th1}, it remains to prove (\ref{h11}).
It follows from $\eqref{a1}_1$ that
\be\la{vvv}(\n-1)_t+\div((\n-1)u)+\div u=0.\ee
Multiplying (\ref{vvv}) by $4(\n-1)^4,$
we obtain after integration by parts that, for $t\ge 1,$
\be  \notag(\|\n-1\|_{L^4}^4)'(t) =-3\int(\n-1)^4\div u dx-4\int(\n-1)^3\div u dx,\ee
which implies that
\be \notag
\int_{1}^{\infty}|(\|\n-1\|_{L^4}^4)'(t)| dt\le C\int_{1}^{\infty}\|\n-1\|_{L^4}^4dt +C\int_{1}^{\infty} \|\na u\|_{L^4}^4 dt \le C\ee
due to (\ref{vu15}). Then it follows from \eqref{vu15}, \eqref{z1}, \eqref{a2.112} and H\"oulder's inequality that for $p\in(2,\infty)$
\be\la{hq1115} \lim_{t\rightarrow\infty}\|\rho-1\|_{L^p} = 0. \ee
Next, we will prove \be \la{vu36} \lim_{t\rightarrow \infty}\left(\|\na u\|_{L^2}+\|\na\te\|_{L^2}\right)=0,\ee
which combined with (\ref{vu15}) and (\ref{hq1115})   gives (\ref{h11}).
In fact, one deduces from $A_2(T)$, $A_3(T)$ in  (\ref{zs2}) and (\ref{vu15}) that
\be\la{mmq}\ba
\int_1^\infty |(\|\na u\|_{L^2}^2)'(t)|dt &=2\int_1^\infty \left|\int \pa_j u^i\pa_j u^i_t dx\right|dt\\
&=2\int_1^\infty \left|\int \pa_j u^i\pa_j(\dot u^i-u^k\pa_k u^i ) dx\right|dt \\&=\int_1^\infty \left|\int (2\pa_j u^i\pa_j \dot u^i-2\pa_j u^i\pa_j u^k\pa_k u^i+|\na u|^2\div u ) dx\right|dt\\&\le C\int_1^\infty \left(\|\na u\|_{L^2}\|\na \dot u\|_{L^2}+\|\na u\|_{L^3}^3\right)dt\\&\le C\int_1^\infty \left(\|\na  \dot u\|_{L^2}^2+\|\na u\|^2_{L^2}+\|\na u\|_{L^4}^4\right)dt \le C,\\
\ea\ee
and
\be \la{vu34}\ba \int_1^\infty|\left(\|\na\te\|_{L^2}^2\right)'(t)
|dt&= 2\int_1^\infty\left| \int \na\te\cdot \na\te_tdx\right|dt\\
&\le C\int_1^\infty\left(\|\na  \te \|_{L^2}^2+\|\na \te_t\|^2_{L^2}\right)dt \le C.\ea\ee
Thus, we derive directly  (\ref{vu36})    from $A_2(T)$ in (\ref{zs2}), (\ref{mmq}), and (\ref{vu34}). The proof of Theorem \ref{th1} is completed.
\end{proof}

\section*{Acknowledgements}   The research  is
partially supported by the National Center for Mathematics and Interdisciplinary Sciences, CAS,
NSFC Grant (Nos.  11688101,  12071200, 11971217),   Double-Thousand Plan of Jiangxi Province (No. jxsq2019101008),  Academic and Technical Leaders Training Plan of Jiangxi Province (No. 20212BCJ23027), and Natural Science
Foundation of Jiangxi Province (No.  20202ACBL211002).

\end{document}